\documentclass[12pt,a4paper]{amsart}
\makeatletter
\renewcommand\normalsize{%
    \@setfontsize\normalsize{11.7}{14pt plus .3pt minus .3pt}%
    \abovedisplayskip 10\p@ \@plus4\p@ \@minus4\p@
    \abovedisplayshortskip 6\p@ \@plus2\p@
    \belowdisplayshortskip 6\p@ \@plus2\p@
    \belowdisplayskip \abovedisplayskip}
\renewcommand\small{%
    \@setfontsize\small{9.5}{12\p@ plus .2\p@ minus .2\p@}%
    \abovedisplayskip 8.5\p@ \@plus4\p@ \@minus1\p@
    \belowdisplayskip \abovedisplayskip
    \abovedisplayshortskip \abovedisplayskip
    \belowdisplayshortskip \abovedisplayskip}
\renewcommand\footnotesize{%
    \@setfontsize\footnotesize{8.5}{9.25\p@ plus .1pt minus .1pt}
    \abovedisplayskip 6\p@ \@plus4\p@ \@minus1\p@
    \belowdisplayskip \abovedisplayskip
    \abovedisplayshortskip \abovedisplayskip
    \belowdisplayshortskip \abovedisplayskip}
\ifdefined\pdfpagewidth
\else
\fi
\calclayout
\makeatother

\usepackage[utf8]{inputenc}
\usepackage[T1]{fontenc}
\usepackage[french, main=english]{babel}
\usepackage{amssymb,latexsym}

\usepackage{mathrsfs,tgschola}
\usepackage[normalem]{ulem}
\usepackage{pgfplots}
\usepackage{tkz-fct}
\usepackage{url}
\usepackage{xcolor}
\usepackage{comment}
\definecolor{violet}{rgb}{0.0,0.2,0.7}
\definecolor{rouge2}{rgb}{0.8,0.0,0.2}
\usepackage{tikz}
\usepackage{empheq}
\usepackage{tikz-cd}
\usetikzlibrary{matrix,arrows,decorations.pathmorphing}
\usepackage{hyperref}
\usepackage{mathpazo}
\usepackage{enumerate}

\hypersetup{
    bookmarks=true,         
    unicode=false,          
    pdftoolbar=true,        
    pdfmenubar=true,        
    pdffitwindow=false,     
    pdfstartview={FitH},    
    pdftitle={},    
    pdfauthor={},     
    colorlinks=true,       
   linkcolor=rouge2,          
    citecolor=violet,        
    filecolor=black,      
    urlcolor=cyan}           
\usepackage{enumitem}
\usepackage{appendix}

 \theoremstyle{plain}    
 \newtheorem{thm}{Theorem}[section]
\theoremstyle{plain} 
\theoremstyle{thm*} \newtheorem*{mthm}{Main Theorem}

 \numberwithin{equation}{section} 
 \numberwithin{figure}{section} 
 \theoremstyle{plain}    
 \newtheorem{prop}[thm]{Proposition} 
 \theoremstyle{plain}    
  \theoremstyle{thm*}    
  \newtheorem*{question*}{Question}
 \newtheorem{lem}[thm]{Lemma} 
 \theoremstyle{remark}
  \newtheorem{claim}[thm]{Claim} 
 \theoremstyle{remark}
 \newtheorem{rem}[thm]{Remark}
 \theoremstyle{definition}
\newtheorem{exa}[thm]{Example}
\theoremstyle{plain}  

\theoremstyle{plain}

\theoremstyle{definition}
\newtheorem{defi}[thm]{Definition}

\newcommand{\Q}{{\mathbb{Q}}}

\def\1{\bold{1}}

\newcommand{\gk}{\bar g_k}

\newcommand{\om}{\omega}
\newcommand{\omb}{\omega_\beta}

\newcommand{\keb}{\mathrm{KE}_\beta}

\newcommand{\ub}{u_\beta}
\newcommand{\wub}{\widehat{u}_\beta}

\newcommand{\Upe}{\Upsilon_{\ep}}

\newcommand{\vp}{\varphi}

\newcommand{\ep}{\varepsilon}

\DeclareMathOperator{\Ric}{Ric}

\renewcommand{\d}{\partial}

\newcommand{\dbar}{\overline{\partial}}
\renewcommand{\ge}{\geqslant}
\renewcommand{\le}{\leqslant}

\newcommand{\vol}{\operatorname{vol}}

\title{Degenerating conic Kähler-Einstein metrics to the normal cone}

\author{Olivier Biquard}
\address{Sorbonne Université, Université Paris Cité, CNRS, IMJ-PRG, F-75005 Paris, France}
\email{olivier.biquard@sorbonne-universite.fr}

\author{Henri Guenancia}
\address{Institut de Mathématiques de Toulouse; UMR 5219, Université de Toulouse; CNRS, UPS, 118 route de Narbonne, F-31062 Toulouse Cedex 9, France}
\email{henri.guenancia@math.cnrs.fr}

\date{}

\begin{document}

\begin{abstract}  
  Let $X$ be a Fano manifold of dimension at least $2$ and $D$ be a smooth divisor in a multiple of the anticanonical class, $\frac1\alpha(-K_X)$ with $\alpha>1$. It is well-known that Kähler-Einstein metrics on $X$ with conic singularities along $D$ may exist only if the angle $2\pi\beta$ is bigger than some positive limit value $2\pi\beta_*$. Under the hypothesis that the automorphisms of $D$ are induced by the automorphisms of the pair $(X,D)$, we prove that for $\beta>\beta_*$ close enough to $\beta_*$, such Kähler-Einstein metrics do exist. We identify the limits at various scales when $\beta\rightarrow\beta_*$ and, in particular, we exhibit the appearance of the Tian-Yau metric of $X\setminus D$.
\end{abstract}

\maketitle

\tableofcontents

\section{Introduction}
\subsection{Kähler-Einstein metrics with cone singularities along a divisor}
Let $X$ be a compact Kähler manifold of dimension $n$ and let $D$ be a smooth divisor on $X$. Given $\beta \in (0,1)$, we say that a Kähler metric $\omega$ on $X\setminus D$ is a Kähler-Einstein metric with cone angle $2\pi \beta$ along $D$ ($\keb$ for short) if
\begin{enumerate}[label=$\bullet$]
\item $\Ric \om = \mu\om$ on $X\setminus D$, for some $\mu \in \mathbb R$,
\item $\om$ has cone singularities with cone angle $2\pi \beta$ along $D$.
\end{enumerate}
In this paper, we will only focus on the case $\mu>0$, so that a $\keb$ metric will be assumed to have positive scalar curvature unless stated otherwise. 

The second item in the definition above means that if $(z_1, \ldots, z_n)$ is any holomorphic system of coordinates on a neighborhood $U$ of a point in $D$ such that $D=(z_1=0)$, then $\omega|_{U\setminus D}$ is quasi-isometric to the model flat cone metric
\[\frac{idz_1\wedge d\bar z_1}{|z_1|^{2(1-\beta)}}+\sum_{j\ge 2} idz_j\wedge d\bar z_j. \]
In particular, $\omega$ extends canonically to a Kähler current on $X$ in the cohomology class $-\mu^{-1}c_1(K_X+(1-\beta)D)$ which satisfies
\[\Ric \om = \om + (1-\beta) [D].\]
A natural and important question is, given a pair $(X,D)$, to determine the set $I:=\{\beta \in (0,1); \exists \, \keb \, \text{metric} \, \omega_\beta\}$ and understand what happens to $\omega_\beta$ when $\beta$ approached a critical value $\beta_*\in \partial I$. 

A necessary condition for the existence of a $\keb$ metric is that $-(K_X+(1-\beta)D)$ be ample, but just as when $D=0$, this condition is not sufficient in general. We refer to \cite{CDS0,BBEGZ, DR17, BBJ, Li22} and the references therein for an account of the problem. 

Determining the admissible cone angles and understanding degeneration of conic Kähler-Einstein metrics has drawn a lot of attention in the last decade, see e.g. \cite{Sz13,LS,C2C,BG22, RS22, Del23}. 

\subsection{Statement of the main theorem}

Let us start by introducing the objects involved in the main theorem below. 



\subsubsection{Varieties}
Let $X$ be a Fano manifold of dimension $n\ge 2$ such that $-K_X \sim_{\Q} \alpha D$ for some positive rational number $\alpha \in \Q$ satisfying $\alpha >1$. 
Note that the case $\alpha=1$ has already been fully treated in \cite{BG22}. The case $\alpha \in (0,1)$ is quite different and involves the   incomplete Ricci flat cone metrics constructed in \cite{Brendle, GP, JMR}, but we will not discuss it further here. 

From the assumptions above, $D$ is a connected Fano manifold and we will further assume that it admits a Kähler-Einstein metric $\om_D$, i.e. $\Ric \om_D=\om_D$. Let $L:=N_D\simeq \mathcal O_X(D)|_D$  be the normal bundle of $D$ and set $v$ to be the fiber coordinate so that $L=\{(x,v); \, x\in D, v\in L_x\}$. We consider the hermitian metric $h$ on $L$ (unique up to a multiplicative constant) such that $i\Theta(L,h)=\frac{1}{\alpha-1}\omega_D$. 

Next let $\overline L=L\sqcup \{x_\infty\}$ be the one-point compactification of $L$; it can be endowed with a structure of normal projective variety. Moreover, $\overline L\setminus D$ is an affine cone over $D$ with apex $x_\infty$, cf section~\ref{sec setup} below. 

Small neighborhoods $\mathcal U_L \subset X$ of $D$ (resp. $\Delta_L\subset L$ of the zero section $O_L\simeq D\subset L$), which are in general non isomorphic, are pseudoconcave and carry a multivalued meromorphic $n$-form $\Omega$ (resp. $\Omega_L$) with a pole of order $\alpha$ along $D$; it is unique up to constant scaling by pseudoconcavity. One can define 
\begin{equation}
\label{def j0}
j_0 \in \mathbb N_{>0}\cup \{+\infty\}
\end{equation}
to be the largest integer such that there exists a diffeomorphism $\Upsilon:\Delta_L\rightarrow \mathcal{U}_L$ such that $\Upsilon^*\Omega$ and $\Omega_L$ coincide up to order $j_0-1$, cf Proposition~\ref{prop:def-j0}. Observe that since the $n$-form determines the complex structure, $j_0$ coincides with the largest integer such that the formal neighborhoods of $D\subset X$ (resp.  $D\subset L$) of order $j_0-1$ are isomorphic. By definition, this means that $j_0$ is the largest integer such that the ringed spaces $(D, \mathcal O_X/\mathcal I_D^{j_0})$ and $(D, \mathcal O_L/\mathcal I_D^{j_0})$ are isomorphic.



\subsubsection{Metrics}
An important numerical factor in the following is the positive number
\begin{equation}
\label{beta*}
\beta_* :=\frac{\alpha-1}{n}.
\end{equation}
Let us recall that one can always assume that $\alpha<n+1$ (cf section~\ref{sec setup}). That is, we have $\beta_* \in (0,1)$ .

The Calabi Ansatz enables one to construct on $\overline L\setminus D$ two canonical Kähler-Einstein metrics (of positive and zero curvature, respectively) having a conical singularity at the point $x_\infty=\overline L\setminus L$. 

\begin{enumerate}[label=$\bullet$]
\item The Ricci flat Tian-Yau cone metric $\omega_{{\rm TY}, L}:=dd^c |v|_h^{-2\beta_*}$ which is complete near $D\subset L$. 
\item The $\mathrm{KE}_{\beta_*}$ metric $\omega_{\beta_*, L}:=dd^c \log(1+|v|_h^{-2\beta_*})$. It satisfies $\Ric \omega_{\beta_*, L}=\omega_{\beta_*, L} $ on $L\setminus D$, has cone singularities of angle $2\pi \beta_*$ along $D\subset L$ and it is asymptotic to $\omega _{{\rm TY}, L}$ near the conical point. 
\end{enumerate}

We will explain in section~\ref{sec rigidity} that the Calabi Ansatz metric $\omega_{\beta_*, L}$ is rigid in the sense that it is the only Kähler-Einstein metric on $\overline L$ with cone singularities along $D$. 

Finally, Tian and Yau \cite{TY2} proved that under our assumptions, there exists 
\begin{enumerate}[label=$\bullet$]
\item A complete Ricci flat Kähler metric $\omega_{\rm TY}$ on $X\setminus D$ which is asymptotic to $\omega_{{\rm TY}, L}$ at infinity,
\end{enumerate}
and we refer to section~\ref{sec:tian-yau-metric} for more details.

\subsubsection{Results}
For $\beta \in (0,1)$, we consider the Kähler-Einstein equation
\begin{equation}
\tag{$\mathrm{KE}_\beta$}
\label{KEb}
\Ric \omega_\beta = \mu \omega_\beta + (1-\beta) [D]
\end{equation}
where $\mu=\mu(\beta)>0$ is given by $\mu \alpha = \alpha+\beta-1$, so that $\omega_\beta \in c_1(X)=\alpha c_1(D)$ is fixed, independent of $\beta$. Since $\mu(\beta_*)>0$ the difference between normalizing the Einstein constant to be $1$ or $\mu_\beta$ is geometrically irrelevant. Consider the interval \[I:=\{\beta \in (0,1]; \exists \,  \omega_\beta \, \, \mbox{solution of} \eqref{KEb}\}.\] 
It is known that $I$ is either empty or of the form $(\beta_*, \beta^*)$ (or $(\beta_*, \beta^*]$) where $\beta_*$ is defined in \eqref{beta*}, cf Proposition~\ref{prop LS}. The identification of $\beta^*$ is a complicated problem in general and depends very much on the geometry of the pair $(X,D)$, cf \cite{Sz13}. However, $I$ is never empty (i.e. $\beta^*>\beta_*$) as a byproduct of our main result below.

\begin{mthm}
\label{mthm}
Let $X$ be a Fano manifold of dimension $n\ge 2$ and let $D$ be  a smooth divisor such that $-K_X \sim_{\Q} \alpha D$ for some $\alpha \in \Q_{>1}$. Assume that $D$ is Kähler-Einstein and that the restriction map $\mathrm{Aut}^\circ(X,D)\to \mathrm{Aut}^\circ(D)$ is onto. Then there exists $\delta>0$ such that
\begin{enumerate}[label=\arabic*.]
\item For any $\beta\in (\beta_*, \beta_*+\delta)$, there exists a $\keb$ metric $\omb$. 
\item There is convergence
\[(X,\omb)\underset{\beta\to \beta_*}{\longrightarrow} (\overline L, \omega_{\beta_*, L})\] in the Gromov-Hausdorff topology. 
\item We have
\[\ep_\beta^{-1} \omega_\beta \underset{\beta\to \beta_*}{\longrightarrow} \om_{\rm TY}\quad \mbox{ in} \quad  C^\infty_{\rm loc}(X\setminus D),\]
where $\ep_\beta$ is defined by 
\[\ep_\beta=
\begin{cases} 
(\beta-\beta_*)^{\frac{\alpha-1}{nj_0}} & \mbox{if } \, j_0<\alpha-1\\
\left(\frac{\beta-\beta_*}{-\log(\beta-\beta_*)}\right)^{\frac 1n} & \mbox{if } \, j_0=\alpha-1\\
(\beta-\beta_*)^{\frac1n} & \mbox{if } \, j_0>\alpha-1
\end{cases}
\]
where $j_0$ is defined in \eqref{def j0}. 

\end{enumerate} 
\end{mthm}

\begin{rem}
Let us first make a few remarks

\begin{enumerate}
\label{rem intro}
\item As a corollary and under the assumption that $\mathrm{Aut}^\circ(X,D)\to \mathrm{Aut}^\circ(D)$ is onto, we get that the Tian-Yau metric $\omega_{\rm TY}$ is {\it canonically} attached to the pair $(X,D)$, modulo the obvious scaling and the action of  $\mathrm{Aut}^\circ(X,D)$. 
\item The condition that $\mathrm{Aut}^\circ(X,D)\to \mathrm{Aut}^\circ(D)$ is onto is satisfied automatically if $X$ is toric, cf Lemma~\ref{toric}. We are not aware of pairs $(X,D)$ as in our setup for which there are holomorphic vector fields on $D$ which do not lift to $X$. If $X$ is allowed to be singular, such examples can be constructed, cf Example~\ref{ex auto}. 
\item In section~\ref{sec examples} we provide examples where the three possible values for $\ep_\beta$ are achieved. 
\item Thanks to \cite{Berndtsson15}, any two $\keb$ metrics are related by an element in $\mathrm{Aut}^\circ(X,D)$. In particular the isometry class of the metric space $(X,\omb)$ does not depend on the particular $\keb$ metric. 
\end{enumerate}
\end{rem}

In the case of $(\mathbb P^2, Q)$ where $Q$ is the smooth quadric, the result above was conjectured by Li and Sun \cite{LS} and shortly after numerical evidence to that conjecture was provided by Li \cite{Li15}. Recently, Delcroix \cite{Del23} settled the conjecture in full \--- he actually considers the more general situation of a rank one horosymmetric Fano manifold with the divisor being a codimension one orbit. 

\subsection{Strategy of proof}

The proof involves several steps.

\subsubsection{Construction of the model metric near $D$}
We work on the normal bundle $L$ and solve the conic Calabi Ansatz to construct a semi-explicit Kähler-Einstein metric $\omega_{\beta, L}=dd^c \phi_{\beta, L}$ on $L$ with cone singularities of angle $2\pi \beta$ along the zero section $D\subset L$; we refer to section~\ref{sec calabi ansatz} for details relative to the conic Calabi Ansatz in this setting. The Ansatz can be used for every angle $\beta\in (0,1)$ but the angle $\beta_*$ is critical in the following sense. 

For $\beta<\beta_*$, $\omega_{\beta,L}$ acquires a cone singularity along  the divisor at infinity $D_\infty = \mathbb P(0\oplus \mathcal O_D)$ in the smooth compactification $ \mathbb P(L\oplus \mathcal O_D)$ of $L$. 

At $\beta=\beta_*$, $D_\infty$ is contracted, and $\omega_{\beta_*, L}$ has a conical singularity at the singular point at infinity of the compactification $\overline L$ of $L$.  

Finally, for $\beta>\beta_*$, $\omega_{\beta, L}$ is only defined away from some euclidean neighborhood $D_\infty$ (or, equivalently, $x_\infty$). More precisely, the domain of definition of $\omega_{\beta,L}$ is larger and larger as $\beta\searrow \beta_*$ and one can identify a zone (escaping to infinity in $L$) where the potential has the asymptotics
 \[\phi_{\beta, L}=\phi_{\beta_*,L}+\frac{\beta-\beta_*}{r^{2n-2}}+l.o.t.\]
 where $r$ is the radius for the Ricci flat cone metric $g_L$ on $L$ (to which $\omega_{\beta_*, L}$ is asymptotic near the conical point).

\subsubsection{Gluing with the Tian-Yau metric}

Fix two parameters $\ep>0$ and $\beta>\beta_*$. It can be useful to think of $\ep$ as the (modulus of the) base parameter in the deformation $\pi : \mathcal X\to \mathbb C$ of $X$ to the normal cone of $D$. The ultimate goal is to produce for each $\ep>0$ small enough an angle $\beta=\beta(\ep)$ and a $\keb$ metric $\omega_{{\rm KE}, \beta(\ep)}$ on $X$ (thought of as $X_{\ep}=\pi^{-1}(\ep)$). Moreover one wants $\beta(\ep)\to \beta_*$ as $\ep\to 0$.

Using a diffeomorphism identifying a neighborhood of $D\subset X$ with a neighborhood of $D\subset L$, one can cook up a Kähler metric $\omega_{\beta, \ep}$ on $X$ obtained by gluing $\om_{\beta, L}$ and $\ep \om_{\rm TY}$ closer and closer to $D$ as $\ep \to 0$. Constructing a $\keb$ metric $\omega_{\beta, \ep}+dd^c f$ amounts to solving a Monge-Ampère equation 
\[P_{\beta, \ep}(f)=0.\]
The relevant properties of $P_{\beta, \ep}$ are as follows. 
\begin{enumerate}[label=$\bullet$]
\item $P_{\beta, \ep}(0)$ measures the "Kähler-Einstein" defect of $\omega_{\beta, \ep}$. It arises from three factors: the difference between the complex structures of $X$ and $L$, the cut-offs performed to glue the two model metrics, the fact that $\om_{\rm TY}$ is Ricci-flat rather than having positive Einstein constant. 
\item $dP_{\beta, \ep}(0)= \Delta_{\omega_{\beta, \ep}}+\mu_\beta=:L_{\beta, \ep}$ where $\mu_\beta$ is our choice of Einstein constant (it depends linearly on $\beta$ with $\mu_{\beta_*}>0$, so one can think of it as being constant). Small eigenvalues of $L_{\beta, \ep}$ represent the obstructions to deform $\omega_{\beta, \ep}$ to a $\keb$ metric. 
\end{enumerate}

\subsubsection{Resolution modulo the obstruction}

 The obstructions arise at least conceptually from three sources: harmonic functions on $(X\setminus D, g_{\rm TY})$, harmonic functions on the cone $(L, g_{L})$ and functions on $L$ in the kernel of $\Delta_{\omega_{\beta_*,L}}+\mu_{\beta_*}$. One can relatively easily kill the first two using suitable weighted Hölder spaces, hence we are left to dealing with the third kind of obstruction. 
 
 Since $\omega_{\beta_*, L}$ is Kähler-Einstein, there is a one-to-one correspondence between $\mathrm{ker}(\Delta_{\omega_{\beta_*,L}}+\mu_{\beta_*})$ and $H^0(L, T_L)\simeq H^0(D, T_D)\oplus \mathbb C \xi$ where $\xi$ is the radial vector field generating the $\mathbb C^*$ action. The assumption that vector fields on $D$ lift to $X$ combined with the fact that $\mathrm{Aut}^\circ(D)=G^{\mathbb C}$ is reductive (since $D$ is KE) allows us to work equivariantly with respect to a compact group $G\subset \mathrm{Aut}^\circ(X,D)$ and kill the obstruction induced by $H^0(D, T_D)$.
 
 One can find an explicit function $\tau_\beta$ solving $i_{\xi}\omega_{\beta, L}=\dbar \tau_\beta$ which we then transplant to a function $\tau_{\beta, \ep}$ on the whole $X$ satisfying $|L_{\beta, \ep} \tau_{\beta, \ep}|\ll 1$. Therefore $\tau_{\beta, \ep}$ is asymptotically in the cokernel of $L_{\beta, \ep}$ and represents the one obstruction remaining in order to solve $P_{\beta, \ep}(f)=0$. In particular, one can solve modulo obstruction, i.e. one can find a (unique) function $f=f_{\beta, \ep}$ such that 
 \[P_{\beta, \ep}(f_{\beta, \ep})^\perp=0\]
 i.e. $P_{\beta, \ep}(f_{\beta, \ep})=a(f_{\beta, \ep}) \cdot \tau_{\beta, \ep}$ for a constant $a(f_{\beta, \ep})\in \mathbb R$. Actually, even that equation cannot be solved readily because $P_{\beta, \ep}(0)$ is not small enough to apply the implicit function theorem, cf Remark~\ref{rem TY formel nec}. This has to do with the fact that the Tian-Yau metric does not provide a good enough approximation of the solution (e.g. since it is Ricci flat). So it is necessary to construct a formal perturbation of $\om_{\rm TY}$ at a large order and use that metric instead of $\om_{\rm TY}$ in the gluing.

\subsubsection{Deforming the cone angle to kill the obstruction}

 We have two parameters $\ep, \beta$ and one obstruction $a(\beta, \ep):=a(f_{\beta, \ep})$ to kill. A signification fraction of the present paper is devoted to showing that the obstruction has an expansion which looks like
 \[a(\beta, \ep)= F(\ep)- (\beta-\beta_*)+ l.o.t.\]
 where $F(\ep)$ is a {\it positive} quantity which has either the form $F(\ep)=\ep^\nu$ (for some $1\le \nu \le n$) or $F(\ep)=\ep^n \log \frac 1\ep$ depending on the convergence rate of $J_X$ to $J_L$ near $D$. The dominant positive contribution to the obstruction, that is $F(\ep)$, comes from either the complex structure change (between $X$ and $L$), or cutting off the $\omega_{\rm TY}$, or both. The dominant negative contribution to the obstruction, that is $- (\beta-\beta_*)$, comes from cutting off $\omega_{\beta, L}$ and its particular shape is due to the "Green's function like" term $\frac{\beta-\beta_*}{r^{2n-2}}$ in the expansion of $\phi_{\beta, L}$. 
 
 Ensuring that $F(\ep)$ is positive is absolutely crucial since it allows us to fix $\ep$ and deform the cone angle as $\beta=\beta_*+F(\ep)+\gamma$ where $l.o.t. \ll |\gamma| \ll F(\ep)$ will vary to ensure that $a(\beta, \ep)$ takes positive and negative values hence vanishes for some value $\beta=\beta(\ep)$. For technical reasons, we will actually fix $\beta$ and vary $\ep$ to the same effect, in order to circumvent the troubles due to the functional spaces changing as $\beta$ varies. 
 
 The computation of $F(\ep)$ relies on two things : (i) the (delicate) construction of a tubular neighborhood of $D$ where the difference between the holomorphic volume forms on $L$ and $X$ {\it has a sign} and (ii) the computation of the sub-leading term {\it and its sign} in the asymptotic expansion of the Tian-Yau potential. This result (Theorem \ref{thm:dev-TY}) may be of independent interest.
 
 \subsection{Comparison with \cite{BG22}} In the paper \cite{BG22}, we had previously treated the case where $\alpha=1$ and we would like to briefly single out the main differences between the techniques involved in that paper and in the present one. 
 
 In \cite{BG22}, the main difficulties stemmed from the collapsing of the $\keb$ metrics at stake coupled with the fact that the cone angle $\beta$ goes to zero; that made the Schauder estimates extremely delicate to establish. The MA equation was obstructed as well (again because of the radial vector field in the normal bundle) but killing the obstruction with the cone angle turned out to be rather easy.  
 
 In our present situation, there is no collapsing and $\beta$ remains bounded away from $0$ and $1$. However, the Calabi Ansatz is much trickier to analyse since we cannot use a scaling argument anymore. More importantly, it is much harder to kill the obstruction as it relies on a very fine understanding of the geometry of the pair $(X,D)$ and its associated canonical KE metrics, as we have explained in the previous paragraph. 
 
 \subsection*{Acknowledgements} The authors are grateful to Hajo Hein for the many enlightening conversations about this problem, and to Thibaut Delcroix for insightful discussions and explaining Example~\ref{ex auto} to us. We would also like to thank the referee for helping improve the exposition of the article. H.G. is partially supported by the French Agence Nationale de la Recherche (ANR) under reference ANR-21-CE40-0010 (KARMAPOLIS). 

\section{Geometric setup}
\label{sec setup}
Let $X$ be a Fano manifold of dimension $n\ge 2$, and let $D$ be a smooth divisor such that
\begin{enumerate}[label = $\bullet$]
\item $-K_X\sim_{\mathbb Q} \alpha D$ for some rational number $\alpha>1$. 
\item $D$ admits a Kähler-Einstein metrics $\om_D$; i.e. $\Ric \om_D=\om_D$. 
\end{enumerate}

Recall that the identify $-K_X\sim_{\mathbb Q} \alpha D$ means that there exists an integer $m\ge 1$ such that $m\alpha \in \mathbb N$ and that the line bundles $K_X^{\otimes m}$ and $\mathcal O_X(-m\alpha D)$ are isomorphic. 

Let us observe that $\alpha \le n+1$. Indeed, by Mori's bend and break, there is always a (rational) curve $C$ such that $(-K_X\cdot C)\le n+1$. In particular, this implies that $\alpha \le \frac{n+1}{(D\cdot C)} \le n+1$. By Kobayashi-Ochiai, equality occurs if and only if $X=\mathbb P^n$ and $D$ is an hyperplane. The latter case is irrelevant for our purposes, so we will assume that $\alpha<n+1$ in the following.

\subsection{Deformation to the normal cone $\overline L$}
Let $L:=N_D$  be the normal bundle of $D$, which is an ample divisor and set $v$ to be the fiber coordinate so that $L=\{(x,v); \, x\in D, v\in L_x\}$. We consider the hermitian metric $h$ on $L$ (unique up to a multiplicative constant) such that $i\Theta(L,h)=\frac{1}{\alpha-1}\omega_D$. 

Next let $\overline L$ be the one-point compactification of $L$; it can be endowed with a structure of normal projective variety in two equivalent ways. First, one can consider the affine cone $C(D,L^{-1})$ obtained e.g. as the contraction of the zero section in the total space of the negative bundle $L^{-1}$ and define $\overline L$ to be the projectivization of $C(D,L^{-1})$. Or we can view $\overline L$ as the blow-down of $\mathbb P(N_D) \subset \mathbb P(N_D\oplus \mathcal O_D)$. It is a classical fact that $\overline L$ can be achieved as the central fiber of a $\mathbb C^*$ degeneration of $X$ using the so-called degeneration to the normal cone. We recall the construction below.\\

Consider the family $\mathfrak X:=\mathrm{Bl}_{D\times \{0\}}(X\times \mathbb C)\to \mathbb C$. The fibers $\mathfrak X_t$ for $t\neq 0$ are isomorphic to $X$ while $\mathfrak X_0 = \mathrm{Bl}_DX \cup \mathbb P(N_D\oplus \mathcal O_D)\simeq X \cup \mathbb P(N_D\oplus \mathcal O_D)$ and the two $n$-dimensional varieties meet along $D$ and $\mathbb P(N_D\oplus 0) \simeq \mathbb P(N_D)$, the section at infinity. One can contract $\mathrm{Bl}_DX \subset \mathfrak X$ via a map $\mathfrak X \to \mathcal X$; the resulting map is called $\pi: \mathcal X \to \mathbb C$. Its fibers outside the origin are isomorphic to $X$ and identified via a $\mathbb C^*$ action that lifts from $\mathbb C$, while the central fiber is the projective cone $\mathbb P(N_D\oplus \mathcal O_D)$ with the section at infinity contracted to a point, or said otherwise it is the one-point compactification $\overline L$ of $N_D$ mentioned above, cf e.g. \cite[\textsection~4.1]{Lisharp} combined with the fact that for $m\ge 0$, one has $H^1(X,mD)=H^1(X,K_X+(-K_X+mD))=0$ by Kodaira vanishing since $-K_X+mD$ is ample.  


On the singular cone $\overline L$, the Calabi-Ansatz enables to construct a $\mathrm{KE}_{\beta_*}$ metric 
\begin{equation}
\label{ombL}
\omega_{\beta_*, L}:=dd^c \log(1+|v|_h^{-2\beta_*}),
\end{equation}
cf \eqref{phi beta*}. The metric is smooth and satisfies $\Ric \omega_{\beta_*, L}=\omega_{\beta_*, L} $ on $L\setminus D$, has cone singularities of angle $2\pi \beta_*$ along $D\subset L$ and it is asymptotic to the Tian-Yau cone metric $\omega _{{\rm TY}, L}=dd^c |v|^{-2\beta_*}$ near the conical point. The existence of $\omega_{\beta_*, L}$ has important consequences as we now explain. \\

\subsection{Smallest possible cone angle}
For a given value $\beta \in (0,1)$, we consider the Kähler-Einstein equation
\begin{equation}
\tag{$\mathrm{KE}_\beta$}
\label{KEb}
\Ric \omega_\beta = \mu \omega_\beta + (1-\beta) [D]
\end{equation}
where $\mu=\mu(\beta)>0$ is given by $\mu \alpha = \alpha+\beta-1$, so that $\omega_\beta \in c_1(X)=\alpha c_1(D)$ is fixed, independent of $\beta$. 
Consider the interval \[I:=\{\beta \in (0,1]; \exists \,  \omega_\beta \, \, \mbox{solution of} \eqref{KEb}\}.\] 
Recall that if $\mathrm{Aut}^\circ(X,D)= \{1\}$, then $I$ is open by \cite{Don}. Actually, much more can be said even without the assumption on holomorphic vector fields. 

\begin{prop}[\cite{LS}]
\label{prop LS}
There exists $\beta^*\in (0,1]$ such that 
\begin{equation}
\label{I connexe}
I=\emptyset, \,\,(\beta_*, \beta^*) \,\, \mbox{or} \, \,(\beta_*, \beta^*]
\end{equation}
 where $\beta_*$ is defined in \eqref{beta*}.  
 \end{prop}
 
 \begin{proof}[Sketch of proof]
 There are three main steps which we sketch below. 
 
 \smallskip
 
 {\it Step 1.} 
 
 \noindent
 The existence of the $\keb$ metric $\omega_{\beta_*, L}$ on the central fiber of a non-trivial test configuration shows that $\beta_*\notin I$. Actually, one can show $I\subset (\beta_*,1)$. One way to do is it to directly compute the Futaki invariant of $(X,(1-\beta)D)$ associated to that test configuration \cite{LS} and show that when $\beta\le \beta_*$, it has the "wrong" sign. Alternatively, one can use Fujita valuative criterion \cite{Fujita16} with respect to the prime divisor $D$, and polarization $L=-(K_X+(1-\beta)D)=-\mu K_X$ with $\mu=  \frac 1 \alpha(\alpha-1+\beta)$. Then, one has $A_{(X,(1-\beta)D)}(D):=1+\mathrm{ord}_D(K_X-(K_X+(1-\beta)D))=\beta$ and $\mathrm{vol}(L-xD)=(\mu-\frac x \alpha)^n (-K_X^n)$, hence the $\beta$ invariant 
\[\beta(D)= A_{(X,(1-\beta)D)}(D) \cdot (L^n)- \int_0^{+\infty}\mathrm{vol}(L-xD)dx\] satisfies $\beta(D)=\mu^n (-K_X^n)\left[ \beta- \frac{\mu \alpha}{n+1}\right]=\frac{n}{n+1}\mu^n (-K_X^n)\left[ \beta- \beta_*\right]$. In particular, $\beta(D)\ge 0$ if and only if $\beta \ge \beta_*$.

 \smallskip
 
 {\it Step 2.} 
 
 \noindent The Ding functional $F_\beta$ for $(X,(1-\beta)D)$ is bounded below for $\beta=\beta_*$ by a classical argument relying on convexity of Mabuchi along (weak) geodesics \cite{LS}.

 \smallskip
 
 {\it Step 3.} 
 
 \noindent The functional  $F_\beta$ is affine in $\beta$ and existence of a $\keb$ metric is equivalent to a suitable notion of properness by \cite{DR17}. Hence the existence of a $\mathrm{KE}_{\beta_0}$ metric implies properness for all $\beta \in (\beta_*, \beta_0]$, hence the existence of a $\keb$ metric for each of these angles.  That is, $I$ is connected and $\inf I \in \{\beta_*, +\infty\}$. 
 \end{proof}

\subsection{Rigidity of  $\omega_{\beta_*, L}$}
\label{sec rigidity}
The aim of this short paragraph is to explain that the cone angle $2\pi \beta_*$ is the only one for which $(\overline L,D)$ admits a Kähler-Einstein cone metric. The rigidity is related to the existence of holomorphic vector field tangent to $D$ (namely, the radial vector field induced by the $\mathbb C^*$ action), since in the absence of such vector fields, the set of possible cone angles is open \cite{Don}. Such a rigidity phenomenon had already been observed for surfaces in \cite[Example~2.8]{CR18}.

\begin{prop}
\label{prop rigidity}
The pair $(\overline L,D)$ admits a $\keb$ metric if, and only if $\beta=\beta_*$. 
\end{prop}

We provide below only a rough sketch of proof based on the computation of the log Futaki invariant associated to the radial vector field $\xi$. Note that Proposition~\ref{prop rigidity} is also a direct consequence of the main theorem (i.e. the existence of $\keb$ metrics on $(X,D)$ for $\beta\in (\beta_*, \beta_*+\delta)$) and the non-existence of $\keb$ metrics on $(X,D)$ for $\beta\le \beta_*$, cf Proposition~\ref{prop LS}.

\begin{proof}
We claim that there exists on $L$ a Kähler metric $\omega:=dd^c \varphi(u)$ with $u=\log |v|_h^2$ such that the following properties are met
\begin{enumerate}[label=$(\roman*)$]
\item $\varphi$ is nonincreasing, convex and coincides with $-u+e^u$ near $-\infty$ and $e^{-u}$ near $+\infty$. 
\item $\varphi'$ is a potential for $\xi$, i.e. $i_\xi \omega= \dbar \varphi'$. 
\end{enumerate}
Indeed, with the help of a drawing one can easily convince oneself that such a function $\varphi$ exists. It induces a Kähler metric on $L$ thanks to the Calabi Ansatz, cf \eqref{omb} below, and the metric actually extends smoothly to a Kähler metric on the singular cone $\overline L$. The second item is also a direct consequence of \eqref{omb} and the identity $\dbar \varphi'= \varphi'' \dbar u$. 

The two properties above imply that $f:=-\varphi' \ge 0$ is identically $1$ on $D$ and decreases with $u$.  Moreover, we have $\omega=\varphi''(u)du \wedge d^c u +p^*{\omega|_D}$ hence we deduce
\begin{eqnarray*}
\frac{1}{\mathrm{vol}(\overline L,\omega)}\int_{\overline L} f \frac{\omega^n}{n!} & = & \frac{1}{\mathrm{vol}(\overline L,\omega)}\int_{-\infty}^{+\infty} f(u) \varphi''(u) du \wedge d^c u \int_D\frac{\omega^{n-1}}{(n-1)!}\\
&<& \frac{1}{\mathrm{vol}(\overline L,\omega)} \int_{-\infty}^{+\infty} \varphi''(u) du \wedge d^c u \int_D\frac{\omega^{n-1}}{(n-1)!} \\
&=& 1 \\
&=&  \frac{1}{\mathrm{vol}(D,\omega)} \int_Df\frac{\omega^{n-1}}{(n-1)!}.
\end{eqnarray*}
Thanks to (the singular version of) \cite{Hashimoto19}, the log Futaki invariant associated to $(\overline L, D, \beta, -\xi,f)$ satisfies
\[\mathrm{Fut}_{\xi, \beta}([\omega])=\mathrm{Fut}_{\xi, 1}([\omega])-a(1-\beta)\]
for some constant $a>0$. In particular, there is at most one value of $\beta$ such that the Futaki invariant vanishes, i.e. $\mathrm{Fut}_{\xi, \beta}([\omega])=0$. Of course, this is none other than $\beta=\beta_*$. The proposition follows. 
\end{proof}

\section{Conic Calabi Ansatz: construction and asymptotics}
\label{sec calabi ansatz}
In this section, we rely on the Calabi Ansatz to construct (possibly incomplete) $\keb$ metrics $\omega_{\beta, L}$ on the normal bundle $L$ (with cone singularities along the zero section $D\subset L$) for each angle $\beta \in (0, 1)$ \--- the value $\beta=\beta_*$ being critical. We then analyze their asymptotic behavior as $\beta\searrow \beta_*$. 

\subsection{Reduction to an ODE}\label{sec:reduction-an-ode}
Let us record notations that will be used throughout the section

\[\begin{array}{|l|l|l|}
\hline
\beta_*=\frac{\alpha-1}n & a=\frac{\mu}{n+1} & C_\beta=\frac{\alpha^n}{n+1}(\beta-\beta_*) \\
 \mu= \frac{1}{\alpha}(\alpha+\beta-1) & b=\beta_*=\frac{\alpha-1}{n} &  \\
 \lambda=\frac{\alpha-1}{\mu} & \psi=-\phi-\lambda u &\\
\hline
\end{array}\]

The the last constant $C_\beta$ is recorded here for clarity but it will only be determined a few lines below. 
\bigskip

On $L=N_D$, we want to solve
\begin{equation}
\label{KE}
\Ric \omega_{\beta}= \mu \omega_\beta+(1-\beta)[D]
\end{equation}
with
\begin{enumerate}[label=(\roman*)]
\item $\omega_\beta=dd^c \phi(u)$ where $u=\log |v|^2_h$, and $h$ is a smooth hermitian metric on $L$ such that $i\Theta(L,h)=\frac{1}{\alpha-1} \omega_D$. 
\item $\omega_\beta$ has conic singularities with cone angle $2\pi \beta$ along $D$. 
\item $\phi(u)\sim - \alpha u$ when $u\to -\infty$. 
\end{enumerate}
The last condition just means that the current ${\omega_\beta}|_D$ belongs to $\alpha c_1(D)|_D$. It is an arbitrary normalization that imposes the value of the Einstein constant to satisfy
\[\mu \alpha =\alpha+\beta-1.\]
The fact that $\omega_\beta$ is a metric is equivalent to having 
\begin{equation}
\label{convex}
-\phi'>0, \quad \phi''>0.
\end{equation}
It is easy to check that
\begin{equation}
\label{OmegaL}
\Xi :=e^{-(\alpha-1)u}du\wedge d^c u \wedge \omega_D^{n-1}
\end{equation}
defines a Ricci flat volume form on $L$ with a pole of order $\alpha$ along the zero section. The metric 
\begin{equation}
\label{omb}
\omega_\beta=dd^c \phi(u)=\phi''(u) du \wedge d^c u -\frac 2{\alpha-1} \phi'(u)\omega_D
\end{equation}
satisfies
\begin{equation}
  \label{eq:35}
  \omega_\beta^n = a_{n,\alpha} \phi''(u)(-\phi'(u))^{n-1} du \wedge d^cu \wedge \omega_D^{n-1}, \quad
  a_{n,\alpha}=\frac{2^{n-1}n}{(\alpha-1)^{n-1}}.
\end{equation}
Therefore $\omega_\beta$ is a solution of
\begin{equation}
  \label{eq:32}
  \omega_\beta^n=a_{n,\alpha}e^{-\mu\phi}\Xi,
\end{equation}
and therefore a solution of \eqref{KE}, if $\phi$ solves
\begin{equation}
\label{MA}
(-\phi')^{n-1}\phi''= e^{-\mu \phi-(\alpha-1)u}.
\end{equation}
Multiplying each side by $-\mu \phi'-(\alpha-1)$ and integrating, we get
\begin{equation}
\label{MA2p}
(-\phi')^{n}(b+a\phi')= e^{-\mu(\phi+\lambda u) }-C_\beta
\end{equation}
for some constant $C_\beta$ to determine. 


By setting $\psi:=-\phi-\lambda u$, we get the equivalent autonomous equation
\begin{equation}
\label{MA2}
(\psi'+\lambda)^{n}(b-a(\psi'+\lambda))= e^{\mu \psi}-C_\beta
\end{equation}

In order to choose the right constant $C_\beta$, let us remember that when $u\to -\infty$, we want our solution $\psi$ to satisfy $\psi(u) \sim (\alpha-\lambda)u = \frac{\beta}{\mu} u$, so that $\psi'+\lambda \sim \alpha$. Therefore we want the constant $C_\beta$ to satisfy
\begin{equation}
\label{cbeta}
\alpha^n(b-a\alpha)=\frac{\alpha^n}{n+1}\big(\frac{\alpha-1}n-\beta\big)=-C_\beta.
\end{equation}
From now on, we impose the constant $C_\beta$ from \eqref{MA2} to be given by the equation \eqref{cbeta} above, i.e. 
\[C_\beta= \frac{\alpha^n}{n+1} \cdot (\beta-\beta_*).\]
 Observe that
\[C_{\beta_*=0} \qquad \mbox{and} \,\, C_\beta>0 \,\, \mbox{for } \, \beta>\beta_*.\]
In the sections that follow, we investigate the existence and the behavior of solutions of \eqref{MA2}. \\

\emph{From now on, we will assume that $\beta>\beta_*$ until the end of section~\ref{sec calabi ansatz} with the exception of section~\ref{sec beta petit}.}

\subsection{Construction of $\psi$ near $-\infty$}

In order for the problem to be well posed, we fix the initial value
$\psi(\hat u_0)=\hat \psi_0$ for some couple $(\hat u_0,\hat \psi_0)$ that will be determined later.

 Let $F(t)=t^n(b-at)$, defined on $[0,\alpha]$.
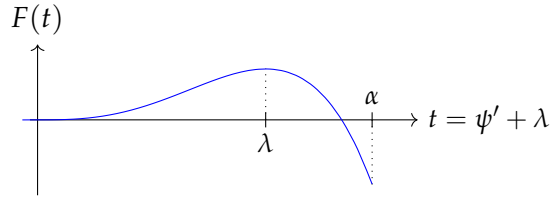
 \begin{center}
 \begin{tikzpicture}
 \tkzInit[xmax=2.5,ystep=0.5,ymax=1]
  \draw[->] (-0.1, 0) -- (5, 0) node[right] {\small $t=\psi'+\lambda$};
  \draw[->] (0, -1) -- (0, 1) node[above] {$F(t)$};
  \draw[dotted] (3, 0) -- (3, 0.68) node[above] {};
   \draw[dotted] (4.4, 0) -- (4.4, -0.9) node[above] {};

  \node[align=left] at (3,-0.3) {\small $\lambda$};
    \node[align=left] at (4.4,0.3) {\small $\alpha$};
    \draw (4.4, 0.1) -- (4.4, -0.1) node[above] {};
  \draw (3, 0.1) -- (3, -0.1) node[above] {};

  \draw[scale=2, domain=-0.1:2.2, smooth, variable=\x, blue] plot ({\x}, {2*\x*\x*\x*(0.2-0.1*\x)});
\end{tikzpicture}
\end{center}
 We have $F(0)=0$, $F(\alpha) \le 0$ with equality if and only if $\beta=\beta_*$. $F$ is increasing up until $t=\lambda$, and then decreasing. The value 
 \begin{equation}
 \label{Fl}
 F(\lambda)=\frac{\beta_*\lambda^n}{n+1}>0
 \end{equation}
  is explicit but does not play any role. $F$ admits two inverses
 \[G_1:[0,F(\lambda)]\to [0, \lambda] \qquad \mbox{and} \qquad G_2:[F(\alpha)=-C_\beta,F(\lambda)]\to [\lambda,\alpha].\]
 The equation \eqref{MA2} can be reformulated as 
 \begin{equation}
 \label{ODE}
 F(\psi'+\lambda)=e^{\mu \psi}-C_\beta
 \end{equation}
 which can be expressed as $\psi'=G_i(e^{\mu \psi}-C_\beta)-\lambda$, i.e.
 \begin{equation}
 \label{prim}
 \frac{d\psi}{G_i(e^{\mu \psi}-C_\beta)-\lambda}=du
 \end{equation}
 for either $i=1,2$. Note that neither function $G_i$ is differentiable at the nexus point $F(\lambda)$.\\
 
\noindent 
The equation
 \begin{equation}
 \label{psi} \int_{\hat \psi_0}^{\psi}\frac{dx}{G_2(e^{\mu x}-C_\beta)-\lambda}=u-\hat u_0
 \end{equation}
 uniquely defines a function $\psi=\psi(u)$ on a neighborhood of $u=\hat u_0$ solving \eqref{ODE}. Here, we need $\hat \psi_0\ll 0$ to make sure that $e^{\mu \hat \psi_0}\in (F(\alpha), F(\lambda))$ so that $G_2(e^{\mu x}-C_\beta)$ is well-defined near $x=\hat \psi_0$. Let us define 
 \begin{equation}
 \label{psi0}
 \psi_0:=\frac 1\mu \log(F(\lambda)+C_\beta),\quad  \mbox{i.e.} \quad  e^{\mu \psi_0}-C_\beta=F(\lambda).
 \end{equation}
 Since $G_2\ge \lambda$, $\psi$ is non-decreasing. Clearly, $\psi$ reaches any value less than $\psi_0$. Moreover, we claim that  $\psi(u)$ is defined for any negative values of $u$. 
 This follows from the fact that the integral
 \[ \int_{\hat \psi_0}^{-\infty}\frac{dx}{G_2(e^{\mu x}-C_\beta)-\lambda} = -\infty\]
 diverges. Indeed, $G_2(e^{\mu x}-C_\beta) \underset{x\to -\infty}{\sim} \alpha$ and $\alpha>\lambda$. In particular, one has $\psi(u)\to -\infty$ when $u\to -\infty$. One can actually say more. Indeed, $G_2(-C_\beta+s) = \alpha+\frac{s}{F'(\alpha)}+O(s^2)$ when $s\to 0^+$. Equivalently, one finds $G_2(e^{\mu x}-C_\beta) -\lambda= (\alpha-\lambda)+\frac{e^{\mu x}}{F'(\alpha)}+O(e^{2\mu x})$ when $x\to -\infty$. Expanding the integral defining $\psi$, one finds near $u=-\infty$:
 \[\psi(u)= (\alpha-\lambda)(u-\hat u_0)+\hat \psi_0+\frac{\alpha-\lambda}{F'(\alpha)}e^{\mu \psi}+O(e^{2\mu \psi}).\] 
Writing $\alpha-\lambda = \frac{\beta}{\mu}$, one infers that 
\[\phi(u)= -\alpha u + u_{\infty}-\frac{\beta}{\mu F'(\alpha)} e^{\beta u}+O(e^{2\beta u})\]
when $u\to -\infty$, for some $u_{\infty}\in \mathbb R$. Since $F'(\alpha)=-\beta\alpha^{n-1}<0$, $\phi$ satisfies the requirements \eqref{convex} at least near $-\infty$ and $dd^c \phi(u)$ has indeed a cone singularity of angle $2\pi \beta$ along $D=(u=-\infty)$. Moreover, thanks to \eqref{MA2p}, we obtain the following expansions for $\phi'$ and $\phi''$: 
\begin{equation}
\label{exp D}
\phi'(u)= -\alpha +c_1e^{\beta u}+O(e^{2\beta u}), \quad \phi''(u)=  c_2 e^{\beta u}+O(e^{2\beta u}),
\end{equation}
where $c_1=-\frac{\beta^2}{\mu F'(\alpha)}$ and $c_2=\beta c_1$. \\

\subsection{Extension of $\psi$ past $\psi_0$.}

We are going to show that $\psi$ can be smoothly extended past the value $\psi_0$ defined in \eqref{psi0}. First, we claim that the integral
 \[\int_{\hat \psi_0}^{\psi_0}\frac{dx}{G_2(e^{\mu x}-C_\beta)-\lambda}\]
is convergent. Indeed, perform the change of variable $s:=e^{\mu x}-C_\beta$ and set $\hat s_0:=e^{\mu \hat \psi_0}-C_\beta$ so that our integral becomes 
\[\frac 1 \mu \int_{\hat s_0}^{F(\lambda)}\frac{ds}{(G_2(s)-\lambda)(s+C_\beta)}=\frac 1 \mu \int_\lambda^{G(\hat s_0)}\frac{-F'(t)dt}{(t-\lambda)(F(t)+C_\beta)}\]
where we have used another change of variable $t=G(s)$ to obtain the RHS. Now an easy computation shows that $F'(t)=\mu t^{n-1}(\lambda-t)$ and the integral becomes $\int_\lambda^{G(\hat s_0)}\frac{t^{n-1} dt}{F(t)+C_\beta}$ which is obviously convergent. 
%
Now, the identity
 \begin{equation}
 \label{psi2} 
 \int_{\psi_0}^{\psi}\frac{dx}{G_2(e^{\mu x}-C_\beta)-\lambda}=u-u_0
 \end{equation}
 shows that $\psi$ is defined on $(-\infty,u_0)$ and satisfies $\psi(u)\to \psi_0$ when $u\to u_0$.  We set 
\begin{equation}
\label{u0}
u_0:=\hat u_0+\int_{\hat \psi_0}^{\psi_0}\frac{dx}{G_2(e^{\mu x}-C_\beta)-\lambda}=\hat u_0+ \int_\lambda^{G(\hat s_0)}\frac{t^{n-1} dt}{F(t)+C_\beta}.
\end{equation}
By the same arguments, one see that the integral
\[ \int_{\psi_0}^{\psi}\frac{dx}{G_1(e^{\mu x}-C_\beta)-\lambda}\]
is convergent, hence  
 \begin{equation}
 \label{psi3} 
 \int_{\psi_0}^{\psi}\frac{dx}{G_1(e^{\mu x}-C_\beta)-\lambda}=u-u_0
 \end{equation}
 defines an extension of $\psi$ for $u>u_0$ close enough to $u_0$ which is continuous across $u_0$ and satisfies 
 \[\psi(u)\underset{u\to u_0}{\longrightarrow} \psi_0, \quad  \psi'(u)\underset{u\to u_0}{\longrightarrow} 0.\]
  Let us now show that $\psi$ is smooth at $u_0$. Differentiating \eqref{ODE} we get  $\psi''F'(\psi'+\lambda)=\mu \psi' e^{\mu \psi}$ away from $u_0$. Since $F'(t)=-\mu t^{n-1}(t-\lambda)$, this implies that 
 \begin{equation}
 \label{psi''}
 \psi''=-\frac{e^{\mu \psi}}{(\psi'+\lambda)^{n-1}}
 \end{equation}
  which converges to the value $\frac{e^{\mu \psi_0}}{\lambda^{n-1}}$ when $u\to u_0$. In particular, $\psi$ is $C^2$ near $u_0$. We get smoothness iteratively thanks to \eqref{psi''}. \\
  
   Finally, since $G_1\le \lambda$, $\psi$ is non-increasing after $u_0$. The expression \eqref{psi3} yields a function $\psi$ satisfying $\psi\ge \psi_\beta$ for $\psi_\beta\in \mathbb R$ which is a solution of 
 \[e^{\mu \psi_\beta}=C_\beta.\] 
 We refer to Figure~\ref{graph} below for the qualitative behavior of $\psi$. \\
  
  We can actually provide an integral expression (or relation) for $\psi$ which does not involve $u_0$. Indeed, from \eqref{u0}, one get
  \begin{eqnarray*}
  \lambda (u_0-\hat u_0)&=&\hat \psi_0-\psi_0+\int_{\hat \psi_0}^{\psi_0}\frac{G_2(e^{\mu x}-C_\beta)dx}{G_2(e^{\mu x}-C_\beta)-\lambda}\\
  &=&\hat \psi_0-\psi_0+\int_\lambda^{G_2(\hat s_0)}\frac{t^{n} dt}{F(t)+C_\beta}
 \end{eqnarray*}
 and similarly, one gets from \eqref{psi3}
 \[\lambda(u-u_0)=\psi_0-\psi+\int_{G_1(s(\psi))}^\lambda \frac{t^n dt}{F(t)+C_\beta}.\]
 Adding the latter identities, we get
 \begin{equation}
 \label{forme integrale}
 \phi(\hat u_0) -\phi(u)=\int_{G_1(s(\psi))}^{G_2(\hat s_0)}\frac{t^{n} dt}{F(t)+C_\beta}
 \end{equation}
 where $s(\psi)=e^{\mu \psi}-C_\beta$. Recall that by construction, we have $\phi(\hat u_0)=-\hat \psi_0-\lambda \hat u_0$, and that $\hat \psi_0$ is constrained to $\hat s_0=s(\hat \psi_0)\in (F(\alpha), F(\lambda))$ where $s(x)=e^{\mu\hat x}-C_\beta$. We fix such a $\hat \psi_0$ once and for all. At this point, $\hat u_0$ is arbitrary, but from now on we will choose $\hat u_0$ so that 
 \[ \hat \psi_0+\lambda \hat u_0+\int_{0}^{G_2(\hat s_0)}\frac{t^{n} dt}{F(t)+C_\beta}=0.\]
As a result, we obtain from \eqref{forme integrale} the relation
 \begin{equation}
 \label{phi I}
 \phi(u)=\int_0^{G_1(s(\psi))} \frac{t^{n} dt}{F(t)+C_\beta}.
 \end{equation}

\subsection{Long time existence of $\psi$.}

There are two cases to consider depending on the sign of $C_\beta$. \\

$\bullet$ If $C_\beta>0$ (i.e. $\beta>\beta_*$), the integral 
\[\int_{\psi_0}^{\psi_\beta}\frac{dx}{G_1(e^{\mu x}-C_\beta)-\lambda}=\int_0^\lambda \frac{t^{n-1}dt}{F(t)+C_\beta}\] 
is convergent. In other words, there exists a finite time $u_\beta$ such that $\psi(u_\beta)=\psi_\beta$. In particular, $dd^c\phi$ cannot be extended as a metric past $u_\beta$. Moreover, the previous integral goes to $+\infty$ as $\beta \to \beta_*$ which shows that $u_\beta\to+\infty$ when $\beta \to \beta_*$.  \\
 
$\bullet$  If $C_\beta=0$ (i.e. $\beta=\beta_*$), then $\psi$ is defined for all times and $\psi(u)\underset{u\to +\infty}{\longrightarrow} -\infty$ since 
\[\int_{\psi_0}^{-\infty}\frac{dx}{G_1(e^{\mu x})-\lambda}=\int_0^\lambda \frac{dt}{t(b-at)}=+\infty\]
is divergent and $\int_{\psi_0}^{\psi_2}\frac{dx}{G_1(e^{\mu x})-\lambda}<+\infty$ is finite for any finite $\psi_2<\psi_0$. One can easily get from \eqref{forme integrale} the successive asymptotics $\psi(u)= -\lambda u + C_1 + o(1)$ and then $\psi(u)=  -\lambda u + C_1-C_2e^{-\frac{\lambda\mu}{n} u }+O(e^{-\frac{2\lambda \mu}{n} u })$ when $u\to +\infty$,  for some $C_2>0$. In terms of $\phi$, we get
 \begin{equation}
 \label{asymptotic TY}
 \phi(u)=   -C_1+C_2e^{-\beta_* u }+O(e^{-2\beta_* u}),
 \end{equation}
where we have used that $\frac{\lambda\mu}{n}=\beta_*$. 

One can actually recover the qualitative behavior of $\phi_{\beta_*}$ from an explicit formula. Indeed, $\phi:=\phi_{\beta_*}$ is a solution of $\phi''=-\phi'(1+\frac{\phi'}{\alpha}) \beta_*$ which can be integrated in 
\begin{eqnarray}
\label{phi beta*}
\phi_{\beta_*}(u)&=c_0-\alpha u +\frac{\alpha}{\beta_*}\log(1+c_1e^{\beta_* u})\\
&=c_2+\frac{\alpha}{\beta_*}\log(1+c_1^{-1}e^{-\beta_*u}) \nonumber
\end{eqnarray}
where $c_0,c_2\in \mathbb R$ and $c_1>0$ are constant.

\subsection{Analysis of the case $\beta<\beta_*$}

\label{sec beta petit}

Again, we fix the initial values $(\hat u_0, \hat \psi_0)$ such that $e^{\mu \hat \psi_0}\in (F(\alpha), F(\lambda))$. Under the assumption that $\beta<\beta_*$, we have $C_\beta<0$ hence the graph of $F$ has the shape below. 

 
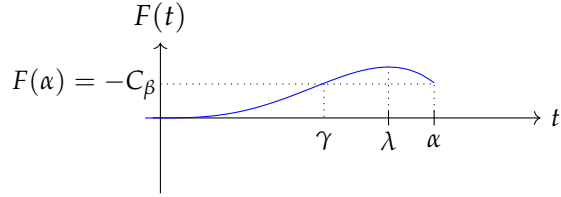
\begin{center}
 \begin{tikzpicture}
 \tkzInit[xmax=2.5,ystep=0.5,ymax=1]
  \draw[->] (-0.1, 0) -- (5, 0) node[right] {\small $t$};
  \draw[->] (0, -1) -- (0, 1) node[above] {$F(t)$};
  \draw[dotted] (3, 0) -- (3, 0.68) node[above] {};
   \draw[dotted] (3.6, 0) -- (3.6, 0.5) node[above] {};
   \draw[dotted] (0, 0.45) -- (3.6, 0.45) node[above] {};
   \draw[dotted] (2.15, 0.45) -- (2.15, 0) node[above] {};

  \node[align=left] at (3,-0.3) {\small $\lambda$};
   \node[align=left] at (3.6,-0.3) {\small $\alpha$};
   \node[align=left] at (2.15,-0.3) {\small $\gamma$};
    \node[align=left] at (-1,0.45) {\small $F(\alpha)=-C_\beta$};

    \draw (3.6, 0.1) -- (3.6, -0.1) node[above] {};
  \draw (3, 0.1) -- (3, -0.1) node[above] {};

  \draw[scale=2, domain=-0.1:1.8, smooth, variable=\x, blue] plot ({\x}, {2*\x*\x*\x*(0.2-0.1*\x)});
\end{tikzpicture}
\end{center}
Here we have set $\gamma:=G_1(F(\alpha))=G_1(-C_\beta)$. We can construct a solution $\psi$ near $u_0$ solving $\psi(\hat u_0)=\hat \psi_0$. 

\begin{claim}
The solution $\psi=\psi(u)$ is defined for all $u\in (-\infty, +\infty)$. 
\end{claim}

\begin{proof}
The existence of $\psi(u)$ for $u$ near $-\infty$ is justified by the divergence of the integral $\int_{-\infty}^{\hat \psi_0} \frac{dx}{G_2(e^{\mu x}-C_\beta)-\lambda}$ which is clear since $G_2(e^{\mu x}-C_\beta)\to \alpha$ when $x\to -\infty$. As in the case $\beta>\beta_0$, one see that there is a value $u_0$ such that $\psi(u)\to \psi_0$ when $u\nearrow u_0$ where $e^{\mu \psi_0}-C_\beta=F(\lambda)$, and that $\psi$ can be extended smoothly past $u_0$. 
The existence of $\psi(u)$ for $u$ near $+\infty$ is justified by the divergence of the integral $\int_{-\infty}^{\psi_0} \frac{dx}{G_1(e^{\mu x}-C_\beta)-\lambda}$ which in turn is a consequence of $G_1(e^{\mu x}-C_\beta)\to \gamma \in (0,\lambda)$ when $x\to -\infty$.
\end{proof}

\begin{claim}
The KE metric $\omega_\beta=dd^c \phi_\beta$ extends to a metric on the compactification $\mathbb P(L\oplus \mathcal O_D)$ with cone singularities along both $D$ and $D_\infty$. The cone angle along $D$ is $2\pi \beta$ and the cone angle along $D_\infty$ is $2\pi \mu (\lambda-\gamma)$ which ranges in $(0, 2\pi(\alpha-1))$ as $\beta$ ranges in $(0,\beta_*)$. 
\end{claim}

\begin{proof}
Let us start with the behavior near $u=-\infty$. One writes $G_2(e^{\mu x}-C_\beta)=\alpha-ce^{\mu x}+O(e^{2\mu x})$ with $c>0$. Plugging that in the identity  $\int_{\psi_0}^{\psi}\frac{dx}{G_2(e^{\mu x}-C_\beta)-\lambda}=u-u_0$, we infer iteratively $\psi-\psi_0 \approx (\alpha-\lambda)(u-u_0)$ and then $\psi-\psi_0=(\alpha-\lambda)(u-u_0)-c' e^{\beta(u-u_0)}+O(e^{2\beta u})$ when $u\to -\infty$, where we have used $\alpha-\lambda=\frac{\beta}{\mu}$. This shows that $\phi(u)-\phi(u_0)=-\alpha(u-u_0)+c''e^{\beta u}+O(e^{2\beta u})$ with $c''>0$, which holds for the derivatives of $\phi$ as well thanks to the MA equation that the latter solves. 

Near $u=+\infty$, the analysis is similar but we use instead $G_1(e^{\mu x}-C_\beta)=\gamma+ce^{\mu x}+O(e^{2\mu x})$ for some other $c>0$. Then $\psi -\psi_0 \simeq (\gamma-\lambda)(u-u_0)$ and $\phi(u)-\phi(u_0)=-\gamma(u-u_0)+c''e^{-\mu(\lambda-\gamma) u}+O(e^{-2\mu(\lambda-\gamma) u})$ hence the result.
\end{proof}

\subsubsection*{Degenerating $\beta$ to zero} In the remaining part of this section, we will explain how $\omega_\beta$ degenerates when $\beta\to0$. 

We fix $\hat u_0=0$ and choose any $\hat \psi_0 <\psi_0$, which yields a value of $u_0$ by \eqref{u0}. Replacing $\psi$ by $\psi(\cdot +u_0)$, one can assume wlog that $u_0=0$.  Since $\alpha-\lambda=\frac{\beta}{\mu}\to 0$, the values $\gamma$ and $\lambda$ converge to $\alpha$. Moreover, one can check that
\begin{equation}
\label{beta to zero}
e^{\mu \psi_0}=F(\lambda)+C_\beta \sim \frac{\alpha^{n}}{2(\alpha-1)}\cdot \beta^2.
\end{equation}
When coupled with the identity $F(t)=F(\lambda)+\frac{(\lambda-t)^2}{2}\cdot F''(\lambda)+O((\lambda-t)^3)$ evaluated at $t=\gamma$, we obtain 
\begin{equation}
\label{gamma beta}
\lambda-\gamma \sim \frac{\alpha}{\alpha-1} \cdot \beta 
\end{equation}
since $F''(\lambda)=-\mu \lambda^{n-1}\approx-\mu \alpha^{n-1}$. \\

\noindent
$\bullet$ \emph{Behavior near $D$.} Let us first focus on what happens near $D$. For $t\in (\lambda, \alpha)$, one writes
 $F(t)-F(\alpha)=\alpha^{n-1}(\alpha-t)(\beta-\frac \mu 2 (\alpha-t)+O(\beta^2))$. Neglecting the $O(\beta^2)$ term, one can solve for $t$ and find
 $t-\lambda= \frac{\beta}{\mu}\sqrt{1-\frac{2\mu}{\alpha^{n-1}}\frac{F(t)-F(\alpha)}{\beta^2}}$. Setting $t:=G_2(e^{\mu x}-C_\beta)$ and using \eqref{beta to zero} this yields $G_2(e^{\mu x}-C_\beta)-\lambda \approx \frac \beta \mu \sqrt{1-e^{\mu (x-\psi_0)}}$ hence
 \begin{eqnarray*}
 -\frac{\beta}{\mu}u&=&\int_{\psi-\psi_0}^0 \frac{dx}{\sqrt{1-e^{\mu x}}}\\
 &=&-(\psi-\psi_0)+c-\frac 1{2\mu}e^{\mu(\psi-\psi_0)}+O(e^{2\mu(\psi-\psi_0)})
 \end{eqnarray*}
 where $c>0$ is a positive harmless constant (it might depend on $\beta$ but does not blow up when $\beta\to 0$). In terms of $\phi$, we have
 \[\phi_\beta(u)=-\psi_0-c-\alpha u+\frac 1{2\mu}e^{\beta u}+O(e^{2\beta u}).\]
In particular, if one sets $r:=\sqrt{\frac 2 \mu}e^{\frac 12 \beta u}\in (0,\sqrt{2\mu^{-1}})$, one sees that near $D$, the Riemannian metric $g_\beta$ associated to $\omega_\beta$ is asymptotic (uniformly in $\beta$) to the Riemannian metric $dr^2+\beta^2\eta^2+\frac{2\alpha}{\alpha-1}g_D$ which collapses the circles. Here, $g_D$ is the Riemannian metric associated to $\omega_D$.  \\

\noindent
$\bullet$ \emph{Behavior near $D_\infty$.} Let us now discuss what happens near $D_\infty$. Since $F'(t)=-\mu t^{n-1}(t-\lambda)$, we derive from \eqref{gamma beta} that $F'(\gamma) \sim   \frac{\mu \alpha^{n}}{\alpha-1}  \beta$ and $F''(\gamma) \approx-\mu \alpha^{n-1}$. So for $t\in (\gamma, \lambda)$, one has $F(t)-F(\gamma)=\mu\alpha^{n-1}(t-\gamma)( \frac{\alpha}{\alpha-1} \beta-\frac{1}{2} (t-\gamma)+O(\beta^2))$. Neglecting the $O(\beta^2)$ term, one can complete the square and find the relation $t-\lambda=-\frac {\alpha}{\alpha-1} \beta \sqrt{1-\frac{2(\alpha-1)^2}{\mu \alpha^{n+1}\beta^2}e^{\mu x}}$. Setting $t=G_1(e^{\mu x}-C_\beta)$ and using \eqref{beta to zero}, this yields
$G_1(e^{\mu x}-C_\beta)-\lambda \approx-\frac {\alpha}{\alpha-1} \beta \sqrt{1-e^{\mu (x-\psi_0)}}$ hence 
 \begin{eqnarray*}
 \frac {\alpha \beta}{ \alpha-1} u&=&\int_{\psi-\psi_0}^0 \frac{dx}{\sqrt{1-e^{\mu x}}}\\
 &=&-(\psi-\psi_0)+c-\frac{1}{2\mu}e^{\mu(\psi-\psi_0)}+O(e^{2\mu(\psi-\psi_0)})
 \end{eqnarray*}
 hence
  \[\phi_\beta(u)=-\psi_0-c-\gamma u+\frac{1}{2\mu}e^{-\frac{\mu \alpha}{\alpha-1}\beta u}+O(e^{-2\frac{\mu \alpha}{\alpha-1}\beta u})\]
  and we conclude similarly to earlier for the case of $D$. The actual expansion has additionnal corrective terms, e.g. the cone singularity has angle $2\pi \mu (\lambda-\gamma)$ which differs from $2\pi\mu\frac{ \alpha}{\alpha-1}\beta$ by a $O(\beta^2)$. \\
  
From the analysis above, one sees that $\omega_\beta$ collapses to $\frac{2\alpha}{\alpha-1}\omega_D$ on the interior of $L\setminus D = \mathbb P(L\oplus \mathcal O_{D})\setminus (D\cup D_\infty)$. In conclusion the Gromov-Hausdorff limit of $(\mathbb P(L\oplus \mathcal O_{D}), g_\beta)$ is the union of two copies of $\Big([0,1]\times D, \frac{2\alpha}{\alpha-1}(dr^2+g_D)\Big)$ glued along $(r=1)\simeq D$ or, equivalently,  $\Big([0,2]\times D, \frac{2\alpha}{\alpha-1}(dr^2+g_D)\Big)$.

\subsection{Asymptotic expansion of $\phi$ near $+\infty$.}

 Now we want to investigate the behavior of $\phi(u)=\phi_\beta(u)$ both as $\beta$ decreases to $\beta_*$ and when $u\to +\infty$. 
 
First, we know that $\psi\to -\infty$ so that $s(\psi)\to 0$. Since $F(t)\sim bt^n$ when $t\to 0$ and $\frac{1}{F(t)+C_\beta}\le \frac 1{F(t)}$, \eqref{phi I} tells us that  $0\le \phi(u)\le G_1(s(\psi))(b^{-1}+o(1))$ which tends to zero. Recall that  $G(s)\sim b^{-\frac 1n } s^{\frac 1n}$ and $s(\psi) \le e^{\mu \psi}=e^{-n\beta_* u -\mu \phi}\le e^{-n\beta_* u}$ since $\phi$ is non-negative. In short, we get   
\begin{equation}
\label{psi5}
\phi(u)=O(e^{-\beta_* u}), 
\end{equation}
where the $O(\cdot)$ is uniform in $\beta$, when $u\to +\infty$.

Since $\psi(u)=-\lambda u - \phi(u)$, the estimate \eqref{psi5} above is already enough to see that the extinction time $u_\beta$ of $\psi_\beta$ (i.e. when the latter reaches the value $\frac 1\mu \log C_\beta$)  happens when $u_\beta \simeq -\frac 1{\lambda\mu} \log C_\beta$, or more precisely 
\begin{equation}
\label{ubeta}
u_\beta  = -\frac 1{n\beta_*} \log C_\beta +O(C_\beta^{\frac 1n}).
\end{equation} 

\noindent
Therefore we have the following picture

  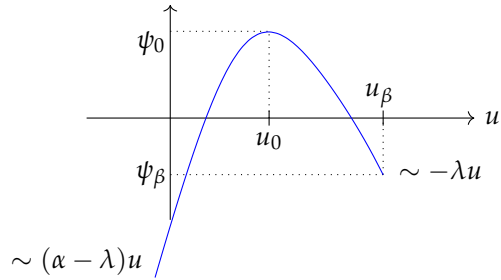
\begin{figure}[h]
   \centering
 \begin{tikzpicture}
 \tkzInit[xmax=2.5,ystep=0.5,ymax=1]
  \draw[->] (-0.1, 0) -- (5, 0) node[right] {\small $u$};
  \draw[->] (1, -1.35) -- (1, 1.5) node[above] {};
  \draw[dotted] (2.3, 0) -- (2.3, 1.15) node[above] {};
   \draw[dotted] (1, 1.15) -- (2.3, 1.15) node[above] {};
   \draw[dotted] (3.8, -0.75) -- (1, -0.75) node[above] {};
   \draw[dotted] (3.8, -0.75) -- (3.8, 0) node[above] {};

  \node[align=left] at (2.3,-0.3) {\small $u_0$};
   \node[align=left] at (3.75,0.3) {\small $u_\beta$};

    \node[align=left] at (0.75,1) {\small $\psi_0$};
        \node[align=left] at (0.75,-0.75) {\small $\psi_\beta$};

  \draw (2.3, 0.1) -- (2.3, -0.1) node[above] {};
    \draw (3.8, 0.1) -- (3.8, -0.1) node[above] {};

\node[text width=2cm] at (-0.1,-1.9) {\small  $\sim (\alpha-\lambda)u$};
\node[text width=2cm] at (5,-0.7) {\small  $\sim -\lambda u$};

\draw[scale=1, domain=0.8:3.8, smooth, variable=\x, blue] plot ({\x}, {e^(-(\x-2)^2)-0.8*(\x-2)^2+(\x-2) });
\end{tikzpicture}
 \caption{Graph of $\psi$}
  \label{graph}
\end{figure}

\bigskip

\medskip
\noindent
To refine \eqref{psi5}, we need to further analyze the integral $\phi=\int_0^{G_1(s(\psi))} \frac{t^{n} dt}{F(t)+C_\beta}$ as $s(\psi)\to 0$. Despite its simple form, it is actually slightly simpler to reverse the change of variable and work with the expression
\begin{equation}
\label{phi exp}
\phi=\frac{1}{\lambda \mu}\int_0^{s(\psi)} \frac{G(s)ds}{(1-\lambda^{-1}G(s))(s+C_\beta)}. 
\end{equation}
 
First, we claim that the function $G(s)$ has a polynomial expansion in $s^{\frac 1n}$ at any (finite) order. We prove this by induction, given that this is clear at order one. Set $\gamma(s)=G(s^n)$, which satisfies 
\begin{equation}
\label{g(s)}
\gamma(s)^n(b-a\gamma(s))=s^n.
\end{equation} Fix an integer $N\ge 0$, and assume that one can write $\gamma(s)=b^{-\frac 1n}s+s^2(P_N(s)+R_N(s))$ where $P_N$ is a polynomial of degree at most $N$ while $R_N(s)=O(s^{N+1})$. By \eqref{g(s)}, we get
\begin{eqnarray*}
b-a(P_N+R_N)=\Big(\frac{s}{\gamma(s)}\Big)^n&=&b^{\frac 1n}\Big(1+\sum_{k=1}^{N+1}(-1)^kb^{\frac kn}s^k(P_N+R_N)^k\Big)+O(s^{N+2})\\
&=&b^{\frac 1n}\Big(1+\sum_{k=1}^{N+1}(-1)^kb^{\frac kn}s^kP_N^k\Big)+O(s^{N+2})
\end{eqnarray*}
from which one deduces that $R_N(s)=c_{N+1}s^{N+1}+O(s^{N+2})$, as desired. 

The expansion of $G(s)$ induces an expansion of $\frac{G(s)}{1-\lambda^{-1}G(s)}$, which we can combine with $\frac{1}{s+C_\beta}= s^{-1}\Big(1-\frac{C_\beta}{s}+O\big(\big(\frac{C_\beta}{s}\big)^2\big)\Big)$ and \eqref{phi exp} to obtain
\begin{equation}
\label{expansion s}
\phi=P_N(s^{\frac 1n})+O(s^{\frac {N+1}n})+\frac{C_\beta}{ s^{1-\frac 1n}}\Big(c_0+O(s^{\frac 1n})+O\big(\frac{C_\beta}{s}\big)\Big)
\end{equation}
where $P_N$ is a polynomial of degree at most $N$ vanishing at $0$, $c_0=\frac{1}{(n-1)b^{1+\frac 1n}}$ while
\begin{equation}
\label{s psi}
s:=s(\psi)=e^{\mu \psi}-C_\beta.
\end{equation}
This asymptotic expansion is valid in the zone $C_\beta \ll s\ll 1$ or, equivalently by \eqref{psi5}, in 
\begin{equation}
\label{zone3}
(\beta-\beta_*)^{\frac1 n}\ll e^{-\beta_*u} \ll 1. 
\end{equation}
It will be convenient to introduce the variables 
\begin{equation}
\label{r rho}
\varrho:=  e^{-\frac 12 \beta_*u} \quad \mbox{and} \quad r=2\varrho
\end{equation}
 ($r$ will correspond to a Riemannian radius but $\varrho$ is more convenient to write expansions). We expand $\phi_\beta$ in powers of $\varrho^2$ and $\delta:=\frac{C_\beta}{\varrho^{2n}}$ (the latter is a $o(1)$ from \eqref{zone3}). Therefore, we need to expand $s$ in powers of $\varrho$ and $\ep$. 

First, one infers from \eqref{psi5} 
\begin{equation}
\label{C_0}
e^{\mu \psi}=  \varrho^{2n}(1+O(\varrho^{2n})).
\end{equation}
In particular, we find 
\begin{equation}
\label{s}
s=e^{\mu \psi}(1-C_\beta e^{-\mu \psi}) = \varrho^{2n} \left(1+O(\varrho^{2n})+\delta(-1+O(\varrho^{2n}))\right)
\end{equation}

We claim that we can improve \eqref{s} as follows. Given any order $N\ge 0$, one can find a polynomial $Q_N$ or degree at most $N$ such that 
\begin{equation}
\label{s(r)}
s^{\frac 1n}=h \varrho^2, \quad \mbox{with} \quad h(\varrho)=Q_N(\varrho)+O(\varrho^{N+1})+\delta(c_1+O(\varrho^2)+O(\delta)),
\end{equation}
for some (negative) constant $c_1$. From \eqref{s}, we have $Q_N(0)=1$. Let us now prove the claim \eqref{s(r)}. In what follows $g_\bullet$ will denote a function of the form $\mathrm{cst}+O(\varrho^2)+O(\delta)$. Note that this class of functions is stable under addition and multiplication.   We assume that the claim is known at some level $N\ge 1$ and we write $h(\varrho)=Q_N(\varrho)+R_N+\delta g_N$. We need to show that $R_N=c_{N+1}\varrho^{N+1}+O(\varrho^{N+2})$.

An easy computation shows that for any $k\in \mathbb N$, we have 
\begin{equation}
\label{hk}
h(\varrho)^k=\mathrm{polynomial}+kC_0^{\frac{k-1}{n}}R_N(\varrho)+\delta g_{h,k}.
\end{equation} 
In particular, we have $P_N(h\varrho^2)=\mathrm{polynomial}+\delta g_N+O(\varrho^{N+3})$, hence for any $k\in \mathbb N$, we have
 \begin{equation}
 \label{PN}
 \left(P_N(h\varrho^2)+O\big((h\varrho^2)^{N+1}\big)\right)^k=\mathrm{polynomial}+\delta g_{N,k}+O(\varrho^{N+3}).
 \end{equation}
 Next, we have $\frac{C_\beta}{s^{1-\frac 1n}}= \delta \varrho^2 h^{1-n}=\delta g_0$ so that for any $k\ge 1$, \eqref{expansion s} and \eqref{PN} yield $\phi_\beta(\varrho)^k=\mathrm{polynomial}+\delta g_{\phi,k} +O(\varrho^{N+3})$. Therefore, we find
 \begin{eqnarray*}
 h(\varrho)^n&=& e^{-\mu \phi}-\delta\\
 &=&e^{-\mu a_0}\Big(1+\sum_{k=1}^N \frac {(-\mu^k)}{k!} \phi^k\Big)+O(\varrho^{2N+2})-\delta\\
 &=& \mathrm{polynomial}+\delta g_{h,n}+O(\varrho^{N+3})
 \end{eqnarray*}
 and \eqref{hk} with $k=n$ yields the expected result \eqref{s(r)} (we actually get the improved result that $Q_N$ is even). It is now straightforward to infer the expansion of $\phi_\beta$ below
 \begin{eqnarray}
\label{expansion}
\phi_\beta(\varrho)&=&\sum_{k=0}^N a_{k} \varrho^{2k}+O(\varrho^{2N+2})\\
&&+a_L\frac{\beta-\beta_*}{\varrho^{2n-2}} \cdot \left(1+O(\varrho^2)+O\Big(\frac{\beta-\beta_*}{\varrho^{2n}}\Big)\right).\nonumber
\end{eqnarray}
Moreover, one can check that coefficients $a_i=a_i(\beta)$ are smooth in $\beta$ and the first ones are given by $a_0=0$, $a_1= \frac{1}{b^{1+\frac 1n}}$, $a_2=\frac{1}{2\lambda b^{1+\frac 2n}}\cdot(1+\frac 1{n+1})$, and $a_{L}=\frac{\alpha^n}{(n-1)n(n+1)b^{1+\frac 1n}}$. 
At this point, it is convenient to (slightly) rescale $L$, replacing $u$ by $u+(1+\frac 1n)\frac{\log b}{b}$. We therefore consider $\tilde \phi(u)=\phi(u-(1+\frac1N)\frac{\log b}b)$: this changes the equations (\ref{MA}) and (\ref{eq:32}) to
\begin{gather}
  \label{eq:37}
\boxed{  \tilde \phi'' (-\tilde \phi')^{n-1} = b^{n+1} e^{-\mu\tilde \phi-(\alpha-1)u} }\\
\boxed{  \omega^n = e^{-\mu\tilde \phi}b_{n,\alpha} \Xi }\label{eq:39}
\end{gather}
where now
\begin{equation}
  \label{eq:38}
  b_{n,\alpha} = \beta_*^{n+1} a_{n,\alpha}.
\end{equation}
As a result, the expansion \eqref{expansion} of $\tilde \phi$ has $a_1=1$.

We now consider that we have done the rescaling and we replace the previous $\phi$ by $\tilde \phi$, so $\phi$ is a solution of (\ref{eq:37}). In this way, $\phi_\beta$ coincides up to order $2$ in $\varrho$ with the unique potential $\phi_{\beta_*}$ defined in \eqref{phi beta*} with $c_1=1,c_2=0$, i.e. 
\begin{equation}
\label{phi beta*2}
\phi_{\beta_*}(\varrho)=\log(1+\varrho^2)=\log(1+\frac{r^2}4)=\frac{r^2}4 - \frac{r^4}{32}+ \cdots
\end{equation} 

Since the $a_i(\beta)$ are differentiable in $\beta$, the polynomial expansion of $\phi_\beta$ (i.e. the first line in the expansion above) differs from that of $\phi_{\beta_*}$ by a factor $O((\beta-\beta_*)r^4)$. Therefore, we have 

\begin{thm}
\label{thm expansion}
Given any $N>0$, the potential $\phi_\beta$ admits the following expansion, valid in the zone $(\beta-\beta_*)^{\frac 1n} \ll r^2 \ll 1$:

  \begin{equation}
\label{expansion2}
\phi_\beta(r)=\phi_{\beta_*}(r)+ a_{L}\frac{\beta-\beta_*}{r^{2n-2}} \cdot \left(1+O(r^2)+O\Big(\frac{\beta-\beta_*}{r^{2n}}\Big)\right)+O(r^{2N}), 
\end{equation}
where $\phi_{\beta_*}$ is given by \eqref{phi beta*2} and  $a_{L}>0$ is a positive constant. 
\end{thm}

\medskip

In the remaining part of this section, let us now explain how to get expansions similar to \eqref{expansion} for $\phi_\beta'$ and $\phi_\beta''$. Since $\varrho=e^{-\frac 12 \beta_*u}$, one has $\partial_u= -\frac {\beta_*}2 \varrho\partial_\varrho$ hence one can rewrite \eqref{MA} as $F(-\frac{\beta_*}{2} \varrho\partial_\varrho \phi)=\varrho^{2n}e^{-\mu \phi}-C_\beta$. Recall that we introduced the function $\delta=\frac{C_\beta}{\varrho^{2n}}$ which goes to zero but at a rate which is "independent" of $\varrho$. The above identity becomes
\[\varrho\partial_\varrho \phi = -\frac 2 {\beta_*}G\Big(\varrho^{2n}(e^{-\mu \phi}-\delta)\Big).\]
Since $G(s)$ has finite expansions in powers of $s^{\frac 1 n}$ near $0$ (cf paragraph around \eqref{g(s)}), any expansion of $\phi$ in powers of $\varrho$ and $\delta$ will yield an expansion of $\varrho\partial_\varrho \phi$. Moreover, the coefficients in the expansion are formally determined by the equation above. More precisely, if one combines the expansion of $G$ and the following expansion induced by \eqref{expansion} 
\begin{eqnarray*}
\varrho^2(e^{-\mu \phi}- \delta)^{\frac 1n}&=&(\sum_{k=1}^N \ast \varrho^{2k}+O(\varrho^{2N+2})) \cdot \Big(1-\frac{1}{n}\delta e^{\mu b_0}+O(\delta \varrho^2)+O(\delta^2)\Big)\\
&=&\sum_{k=1}^N \ast \varrho^{2k}+O(\varrho^{2N+2}) +\delta \varrho^2(\ast+ O( \varrho^2)+O(\delta))
\end{eqnarray*}
we get an expansion $\varrho \partial_\varrho\phi=\sum_{k=1}^N \ast \varrho^{2k}+O(\varrho^{2N+2}) +\delta \varrho^2(\ast+ O( \varrho^2)+O(\delta))$ and we recover the coefficients from \eqref{expansion} by integration, i.e. 

{\small
\begin{eqnarray}
\label{expansion2bis}
\varrho\partial_\varrho \phi_\beta(r)&=&\sum_{k=1}^N 2k \, a_{k} \varrho^{2k}+O(\varrho^{2N+2})+\\
&&+(2n-2)a_{L}\frac{\beta-\beta_*}{\varrho^{2n-2}} \cdot \left(1+O(\varrho^2)+O\Big(\frac{\beta-\beta_*}{\varrho^{2n}}\Big)\right).\nonumber
\end{eqnarray}}
Feeding this into \eqref{MA}, we obtain next
{\small
\begin{eqnarray}
\label{expansion3}
(\varrho\partial_\varrho)^2 \phi_\beta(\varrho)&=&\sum_{k=1}^N (2k)^2 \, a_{k} \varrho^{2k}+O(\varrho^{2N+2})+\\
&&+(2n-2)^2 a_{L}\frac{\beta-\beta_*}{\varrho^{2n-2}} \cdot \left(1+O(\varrho^2)+O\Big(\frac{\beta-\beta_*}{\varrho^{2n}}\Big)\right).\nonumber
\end{eqnarray}}
Combining \eqref{expansion2} and \eqref{expansion3} and coming back to the variable $r$, this leads to 
  \begin{equation}
\label{expansion4}
\omega_{\beta}=\omega_{\beta_*}+\Big(O\big(\frac{\beta-\beta_*}{r^{2n}}\big)+O(r^{2N-2}) \Big) \omega_\beta.
\end{equation}

  \subsection{Asymptotics of $g_\beta$}

In order to analyse the Riemannian metric $g_\beta$ associated to $\omb$, we introduce the variable $s:=-\phi'$ which decreases from $\alpha$ (on $D$, i.e. at $u=-\infty$) to $0$ (at the extinction time $u=u_\beta$ of the solution) corresponding to the moment map of the $S^1$ action on $L$ and write $g_\beta$ as 
\begin{align*}
g_\beta&=\phi''(du^2+4\eta^2)+\frac {2s}{\alpha-1}g_D\\
&=\frac{ds^2}{\phi''}+4\phi'' \eta^2+\frac {2s}{\alpha-1}g_D
 \end{align*}
 where $\eta$ is the connection $1$-form, satisfying $d\eta=-\frac{1}{\alpha-1} \omega_D$. Using \eqref{MA2}-\eqref{MA2p}, we see that 
 \[\phi''=\frac 1 {s^{n-1}}(s^n(b-as)+C_\beta).\]
 Given \eqref{cbeta}, we have $\phi''(\alpha)=0$; after factorizing, we get $\phi''(s)=s(\alpha-s)(a-(b-a\alpha)\frac{\alpha^n-s^n}{s^n(\alpha-s)})$. Using $\alpha a = \beta_*+\frac{1}{n+1}(\beta-\beta_*)$ and $b-a\alpha=\frac{1}{n+1}(\beta-\beta_*)$, we eventually find 
 \begin{equation}
 \label{phi''}
 \phi''=s(1-\frac s\alpha)\beta_*\left(1+\frac{1}{(n+1)\beta_*}(\beta-\beta_*)\cdot\frac{\alpha^{n+1}-s^{n+1}}{s^n(\alpha-s)}\right).
 \end{equation}
Setting $V_\beta(s)=\phi''$, we can rewrite $g_\beta$ as
  \begin{equation}
 \label{gb}
 g_\beta= \frac{ds^2}{V_\beta(s)}+V_{\beta}(s)4\eta^2+\frac {2s}{\alpha-1}g_D.
 \end{equation}
Using \eqref{phi''}, we find:
 \begin{enumerate}[label=$\bullet$]
 \item When $s\to \alpha$, $V_\beta(s)\simeq (\alpha-s)\beta$. Setting $r :=2\sqrt{\frac {\alpha-s}{\beta}}$, we see that $g_\beta \simeq dr^2+\beta^2r^2\eta^2+\frac{2s}{\alpha-1}g_D$, i.e. $g_\beta$ has a cone singularity of angle $2\pi \beta$ along $D$.
 \item When $s\to 0$, we have either 
 \begin{equation}
 \label{equiV}
 V_\beta(s)\simeq
 \begin{cases}
  \frac{\alpha^n}{n+1}(\beta-\beta_*)\frac1{s^{n-1}} & \mbox{if } \beta>\beta_* \\
  s\beta_* & \mbox{if } \beta=\beta_*
  \end{cases}.
  \end{equation}
  In the case where $\beta>\beta_*$, $g_\beta$ is an incomplete metric where the circles blow up and the divisor gets contracted. 
In the case where $\beta=\beta_*$, $g_{\beta_*}$ is incomplete again and contracts the divisor at infinity $D_\infty$ in $L$; i.e. $D_\infty=\mathbb P(L) \subset \mathbb P(L\oplus \mathcal O_D)$. The metric $g_{\beta_*}$ is equivalent near zero to the cone metric 
\begin{eqnarray*}
g_{\rm cone}&=&\frac{ds^2}{s\beta_*}+4s\beta_*\eta^2+\frac{2s}{\alpha-1}g_D\\
&=&dr^2+r^2\beta_*^2 \eta^2+\frac{\beta_* r^2}{2(\alpha-1)}g_D
\end{eqnarray*}
with $r:=2\sqrt{\frac {s}{\beta_*}}$, which is the asymptotic cone of the Tian-Yau metric. 
 \end{enumerate}
 
 \bigskip
 
 Let us now obtain an expansion in $\beta-\beta_*$ of the metric $g_\beta$, which we will be mostly interested in near $s=0$ where we will glue it to the Tian-Yau metric. So we set
\begin{eqnarray*}
g_{\beta_*}&:=&\frac{1}{s(1-\frac s \alpha)\beta_*}ds^2+4s(1-\frac s \alpha)\beta_*\eta^2+\frac{2s}{\alpha-1}g_D\\
&=&g_{\rm cone}+\frac{s}{\alpha}(\frac{ds^2}{s\beta_*}-4s\beta_*\eta^2)+O(s^2)\\
&=&g_{\rm cone}+\frac{\beta_*}{4\alpha}r^2dr^2+O(r^4).
\end{eqnarray*}
The expansion of $g_\beta$ is given by 
\[g_\beta=g_{\beta_*}+(\beta-\beta_*)g_1+O((\beta-\beta_*)^2)\]
where
\[g_1=\frac{\alpha^{n+1}-s^{n+1}}{s^n(\alpha-s)}\cdot \frac{1}{(n+1)\beta_*}\left(\frac{-ds^2}{s(1-\frac s\alpha)\beta_*}+4s(1-\frac s\alpha)\beta_* \eta^2\right).\]

\subsection{Finding the scaling factor and the gluing region}

 In order to perform a gluing construction, we need to ensure that most of the domain of the potential $\phi_\beta$ be included in an infinitesimally small (i.e. as $\beta\to \beta_*$) neighborhood of the zero section. This is certainly not the case as such but it can be achieved by using the radial automorphisms of $L$. More precisely, we use the scaling $\mathbb C^*$ action on $L$ by $t\cdot (x,v)=(x,t v)$ to get a new potential $\phi_{\beta}^t(v)=\phi_\beta(t v)$. In terms of the function $u$, we have $\phi_{\beta}^t(u)=\phi_\beta(u-\log |t|^2)$. 
 
 From now on we use a diffeomorphism of a small neighborhood of fixed euclidean size of $D\subset X$ with a small neighborhood of the zero section in $L$ to identify both sets. If we glue near $u$, we have to make sure that the size of $D$ measured with $\omb$ and $\ep \om_{\rm TY}$ is the same (here $\ep=\ep_\beta$), which yields the equation
 \begin{equation}
 \label{scale factor}
 (-\phi_\beta^t)'(u)\approx\ep  e^{-\beta_*u},
 \end{equation}
 where $\approx$ means that the log of the quotient is bounded. Given \eqref{expansion}, we find that in the zone $(\beta-\beta_*)^{\frac 1n} \ll e^{-\beta_*u} \ll 1$ we have $\phi_\beta'(u) \approx e^{-\beta_*u}$ hence the equation \eqref{scale factor} is satisfied provided
 \[\ep\approx |t|^{2\beta_*}, \quad \mbox{i.e.} \,\, \log |t|^2=\frac1{\beta_*} \log \ep+O(1)\]
 holds.  It will be convenient to introduce the scaling (shrinking) map $\lambda_\ep:L\to L$ defined by
 \begin{equation}
 \label{scaling}
 \lambda_\ep(v)=\ep^{\frac {1}{2\beta_*}} v.
 \end{equation}
 
 \noindent
 From now on we set 
 \[\wub:=-\frac {1}{n\beta_*} \log (\beta-\beta_*)\]
  Finally, we set 
  \[\psi_\beta:=\phi_\beta(\cdot - \frac 1{\beta_*}\log \ep)=(\lambda_\ep)_*\phi_\beta\]
   which is defined on $(u\le \wub+ \frac 1{\beta_*}\log \ep)$. The asymptotics of $\phi_\beta(v)$ match those of the conic Tian-Yau potential in the region where $1\ll v \ll -\frac 1{n\beta_*} \log(\beta-\beta_*)$, cf \eqref{expansion}. Therefore we need to glue $dd^c \psi_\beta(u)$ to $\ep_\beta \om_{\rm TY}$ near $u=\ub$ (i.e. in a zone like $ |u-u_\beta| \le 1$)  satisfying
 \[\frac 1\beta_*\log \ep_\beta \ll \ub \ll \frac 1{\beta_*} \log \frac{\ep}{(\beta-\beta_*)^{\frac 1n}}\]
 In term of the Tian-Yau radius $R=2|v|^{-\beta_*}$, this means 
 \begin{equation}
 \label{zone}
 (\beta-\beta_*)^{\frac 1n} \ll \ep R^2 \ll 1.
 \end{equation}
 In the gluing zone, \eqref{psi5} guarantees that $s(u)=-\phi'(u - \frac 1\beta_*\log \ep) \simeq \ep R^2$ hence $s^n \gg \beta-\beta_*$ by \eqref{zone}. Combined with \eqref{phi''}, this shows that 
 \begin{equation}
   \label{eq:42}
   g_\beta \approx g_{\rm cone}=\frac{ds^2}{s\beta^*}+4s\beta_*\eta^2+\frac{2s}{\alpha-1}g_D.
 \end{equation}

 \section{Automorphisms and related objects}
 \label{sec:about-automorphisms}
 
 \subsection{Relations between $\mathrm{Aut}^\circ(X,D)$ and $\mathrm{Aut}^\circ(D)$}
 Let $Y:=\mathbb P(L\oplus \mathcal O_D)$, let $p:Y\to \overline L$ be the contraction of the divisor at infinity $D_\infty:=\mathbb P(L\oplus 0)\subset Y$ and let $x_\infty=p(D_\infty)\in \overline L$, cf section~\ref{sec setup} for the definition of $\overline L$.  The zero section $D:=\mathbb P(L\oplus 1)$ can be equivalently seen inside $\overline L$ or $Y$.  We have a natural sequence
 \begin{equation}
 \label{ES}
 1\longrightarrow \mathbb C^*\longrightarrow \mathrm{Aut}^\circ(\overline L,D)  \longrightarrow  \mathrm{Aut}^\circ(D)\longrightarrow1
 \end{equation}
 where the action of $\mathbb C^*$ is given on $Y$ by $\lambda\cdot [v:z]:=[\lambda v:z]$; it obviously preserves $D=(v=0)$ and $D_\infty=(z=0)$ hence descends to an action of $\mathbb C^*$ on  $\overline L$ preserving $D$. The second arrow is simply the restriction to $D$.

  \begin{lem}
  \label{exact}
The short sequence \eqref{ES} is exact. 
 \end{lem}

\begin{proof}
The rightmost arrow is exact since any automorphism of $D$ induces a linear automorphism of $L=(\alpha-1)K_D^{-1}$. It remains to see that we have exactness in the middle. So let $f\in \mathrm{Aut}^\circ(\overline L,D) $ acting trivially on $D$; we need to show that $f$ comes from the $\mathbb C^*$ action. First, $f$ obviously preserves the singular locus $\overline L^{\rm sing}=\{x_\infty\}$. Since $p$ is simply the blow up of 
$x_\infty\in \overline L$, $f$ lifts to an automorphism $\widetilde f \in \mathrm {Aut}^\circ(Y)$ acting trivially on $D$. A classical argument due to Blanchard \cite{Blanchard56} shows that any element in the identity component of the automorphism group of a fiber bundle with compact fibers fixing the base has to preserve the fiber bundle structure. In our case it means that $\widetilde f$ preserves the $\mathbb P^1$ fibers of the projection $\rho:Y\to D$. Given $x\in D$, $\widetilde f$ acts on the fiber $\rho^{-1}(x)\simeq \mathbb P^1$ preserving $0$ (i.e. $D\cap \rho^{-1}(x)$) and $\infty$ (i.e. $D_\infty\cap \rho^{-1}(x)$),  hence it can be written as an homothety $[u:v]\mapsto [a(x)u,v]$. The function $a$ defines a holomorphic map $D\to \mathbb C^*$ hence it is constant. The lemma follows.
\end{proof}

%
%

\begin{lem}
\label{toric}
Assume $(X, D)\neq (\mathbb P^n, H)$ for $H$ an hyperplane and consider the morphism $\rho_D$ induced by restriction to $D$
\begin{equation}
\label{rhoD}
\rho_D:\mathrm{Aut}^\circ(X,D) \longrightarrow \mathrm{Aut}^\circ(D).
\end{equation}
The following hold.
\begin{enumerate}[label= (\roman*)]
\item $\rho_D$ has finite kernel. 
\item If $X$ is toric, then $\rho_D$ is a surjective, finite étale cover. 
\end{enumerate}
\end{lem}

\begin{proof}
Since the groups at stake are algebraic, the lemma can be rephrased by saying that 
\[H^0(X, T_X(-\log D))\to H^0(D, T_D)\]
is injective (first item) and surjective when $X$ is toric (second item). We have the following exact sequence 
\begin{equation}
\label{exact seq}
0\longrightarrow T_X(-D)\longrightarrow T_X(-\log D) \longrightarrow T_D \longrightarrow 0.
\end{equation}
By a result of  J. Wahl \cite{Wahl83}, we have $H^0(X, T_X\otimes \mathcal O_X(-D))=0$ since $(X,D)\neq (\mathbb P^n, H)$, so this shows $(i)$. 

Using \eqref{exact seq} again, $(ii)$ reduces to showing that $H^1(X, T_X(-D))=0$. Since $T_X\simeq \Omega_X^{n-1}\otimes K_X^{-1}$, we have to show that 
\[H^1(X,\Omega_X^{n-1}\otimes L)=0,\]
where $L=-K_X-D\simeq(\alpha -1)D$ is an ample line bundle. But this follows precisely from Bott's vanishing theorem for toric varieties, cf \cite[\textsection~3.3]{Oda}.
\end{proof}


If $D$ is allowed to be singular, then $\rho_D$ need not be isomorphic as shown by the following example, which was communicated to us by Thibaut Delcroix. 
\begin{exa}
\label{ex auto}
Let $X=Q_4=\{\sum_{i=0}^5 z_i^2=0\}\subset \mathbb P^5$ be the four-dimensional smooth quadric and let $D=\{z_0^2+z_1^2+z_2^2=0= z_3^2+z_4^2+z_5^2\}\subset X$. Then $D$ is irreducible and reduced, with an action of $\mathrm{SO}_3(\mathbb C)\times \mathrm{SO}_3(\mathbb C)$ with dense orbit, and satisfies $-K_X \sim_{\mathbb Z} 2 D$. But the one-parameter family of automorphisms 
\[t\cdot [z_0: \ldots: z_5]=[tz_0: tz_1: tz_2: t^{-1}z_3:t^{-1}z_4:t^{-1}z_5]\] does not extend to $\mathrm{PSO}_{6}(\mathbb C)$. Moreover, one can check that $D$ admits a Kähler-Einstein metric combining \cite{Del20} and \cite{Li22}. 
\end{exa}

\subsection{About the normal bundle}
\label{sec normal}
In this section, we investigate whether the formal neighborhood of $D$ in $X$ of order one (that is, the ringed space $(D, \mathcal O_X/\mathcal I_D^2)$ coincides with the formal neighborhood of $D$ in $L$ (that is, the ringed space $(D, \mathcal O_L/\mathcal I_D^2)$). From \cite[Proposition~1.5]{ABT09}, this is equivalent to the normal short exact sequence
\begin{equation}
\label{NES}
0\longrightarrow T_D \longrightarrow T_{X}|_D \longrightarrow N_D \longrightarrow 0
\end{equation}
being split. A sufficient condition for the splitting of \eqref{NES} to happen is to have the vanishing $H^1(D, T_D\otimes N_D^{-1})=0$ as we see by tensorizing the above sequence with $N_D^{-1}$ and observing that extensions of $\mathcal O_D$ by $T_D\otimes N_D^{-1}$ are parametrized by the image of $1\in H^0(D, \mathcal O_D)$ in $H^1(D, T_D\otimes N_D^{-1})$ under the connecting homomorphism in the long exact sequence in cohomology associated to \eqref{NES}.  We give a couple of situations where \eqref{NES} is split or non-split.  \\

\begin{exa}
\label{split Pn}
 If $X=\mathbb P^n$, then a theorem of Van de Ven \cite{VdV} shows that \eqref{NES} is split if and only if $D$ is a linear subspace of $\mathbb P^n$. Alternatively, if \eqref{NES} is split, then $T_D$ is a quotient of the ample bundle $T_{\mathbb P^n}|_D$ hence it is ample, too. By Mori theorem \cite{Mori79} (or Siu-Yau's proof of Frankel's conjecture \cite{SY80}), $D\simeq \mathbb P^n$
\end{exa}

\begin{exa}
\label{split toric}
 If $X$ is toric and $\alpha >2$, then \eqref{NES} is split if and only if $(X,D)=(\mathbb P^n,H)$ where $H$ is an hyperplane. Indeed, we have 
 {\small
 \[H^1(X, T_X(-2D))\simeq H^1(X, \Omega_X^{n-1}\otimes K_X^{-1}(-2D))=H^1(X, \Omega_X^{n-1}\otimes \mathcal O_X((\alpha-2)D))=0\] 
} by Bott vanishing theorem. In particular, the exact sequence
\[0\longrightarrow T_X(-2D) \longrightarrow T_{X}(-D) \longrightarrow T_{X}|_D\otimes N_D^{-1} \longrightarrow 0\]
shows that any section of $T_X|_D\otimes N_D^{-1}$ extends to a section of $T_X(-D)$. Now, if \eqref{NES} splits, we precisely get a non-zero section $H^0(D,T_X|_D\otimes N_D^{-1})$, hence we obtain a non-zero section $H^0(X, T_X(-D))$. By \cite{Wahl83}, this implies that $(X,D)$ is isomorphic to $(\mathbb P^n,H)$ as announced. 
\end{exa}

\begin{exa}
\label{split D toric} If $D$ is toric and $\alpha>2$, then \eqref{NES} is always split. Indeed, we have $H^1(D, T_D\otimes N_D^{-1})\simeq H^1(D, \Omega_D^{n-2}\otimes \mathcal O_D((\alpha-2)D))=0$ by Bott vanishing theorem.
\end{exa} 

\begin{exa}
\label{split quadric}
 If $X= Q_n\subset \mathbb P^{n+1}$ is the $n$-dimensional quadric and $D=Q_{n-1}= Q_n \cap H$ for an hyperplane $H$, then we have $K_X^{-1}=\mathcal O_X(n)\sim nD$, i.e. $\alpha=n$, and $D$ is Kähler-Einstein. Moreover, for any $m\ge 1$, we have the vanishing $H^1(Q_{m}, T_{Q_{m}}(-1))\simeq H^1(Q_m, \Omega_{Q_m}^{m-1}(m-1))=0$ by \cite[p. 174]{Snow}. Apply this to $m=n-1$ to get that \eqref{NES} is split. Conversely, if $D\subset Q_n$ is an hypersurface such that \eqref{NES} splits, then $D\simeq Q_{n-1}$ by \cite[Theorem~4.7]{Jahnke}.
 \end{exa}

%
%
%
%
%

\section{The tubular neighbourhood of $D$}
\label{sec:tubul-neighb-d}

Since $-K_X \sim \alpha D$, the normal bundle $L=N_D$ of $D$ in $X$ has a canonical volume form, which in local coordinates $(z^0,...,z^{n-1})$ such that $D=\{z^0=0\}$ can be written as
\begin{equation}
  \label{eq:18}
  \Omega_L = v(z^1,...,z^{n-1}) \frac{dz^0}{(z^0)^\alpha}\wedge dz^1 \wedge \cdots \wedge dz^{n-1}.
\end{equation}
This is not completely correct, since $(z^0)^\alpha$ a priori does not make sense: if we write $\alpha=\frac pq$ then we only have a section $\Omega^q$ of $K_X^q(pD)$ on $X$. But since at the end we need only to consider the real volume form $\Omega\wedge\overline \Omega$ which is well defined as $(\Omega^q\wedge \overline{\Omega^q})^{\frac 1q}$, we will simplify the notation and formulas by writing $\Omega$ (and $\Omega_L$ on $L$).

\subsection{The complex structure}
\label{sec:complex-structure}

Let $\Delta_L\subset L$ be the disc normal bundle of $D$. The connection 1-form $\eta$ on $L\setminus D$ gives a horizontal space $H\subset T\Delta_L$ which is transverse to each disc. If we have another complex structure $J$ on $\Delta_L$, we compare it to the complex structure $J_L$ of $L$ by a tensor $\phi\in \Omega^{0,1}_{J_L}\otimes T^{1,0}_{J_L}$ so that
\[ T^{0,1}_J = \{ X + \phi_X, X\in T^{0,1}_{J_L} \}. \]
The integrability of $J$ gives the equation $\dbar \phi + \frac12[\phi,\phi]=0$, where the bracket is constructed from the exterior product of forms and from the bracket of vector fields.

We will use the following canonical parametrization of a small tubular neighbourhood of $D$ by $\Delta_L$, see \cite[Theorem 4.1]{Biq02}:
\begin{prop}\label{prop:normal-form-1}
  There exists a small neighbourhood $U_L$ of $D$ in $X$ and a diffeomorphism $\Upsilon:\Delta_L\rightarrow U_L$ such that $\phi=\Upsilon^*J_X-J_L \in \Omega^{0,1}(T^{1,0})$ is a section of $H_{0,1}\otimes H^{1,0}$ (that is, $\phi$ is purely horizontal) which satisfies $\phi\lrcorner d\eta=0$, $\phi|_D=0$ and is holomorphic along the discs of $\Delta_L$.

  Moreover $\Upsilon^*\Omega = v (1-\phi)^*\Omega_L$ where $v$ is a function on $\Delta_L$ which is holomorphic along the discs and satisfies $v|_D=1$.\qed
\end{prop}

\subsection{The complex volume form}
\label{sec:complex-volume-form}

The meaning of Proposition \ref{prop:normal-form-1} is that any neighbourhood of $D$ which is close enough to a disc bundle $\Delta_L$ carries a unique fibration by holomorphic discs satisfying the conclusions of the proposition. But there is no particular choice of $U_L$. Therefore we get an additional freedom by perturbing $U_L$. Using this flexibility, one should be able to obtain $v\equiv 1$. As we need only finite developments, it will be sufficient to prove the following Proposition. We denote $\zeta$ the variable in the total space of $L$.

We decompose $T_{\mathbb{C}}L = \mathbb{C}\zeta \partial_\zeta \oplus T_{\mathbb{C}}D$. In the following proposition we will measure the norm of an endomorphism of $TL$ with respect to this decomposition (that is we consider $\zeta \partial_\zeta$ as having norm $1$).

\begin{prop}\label{prop:normal-form-2}
  There are purely horizontal tensors $\phi_i\in \Omega^{0,1}_D(T^{1,0}D\otimes L^{-i})$ for $i>0$ such that $\phi_i\lrcorner d\eta=0$ and the following is true.
  
  Given any $j>0$, one can choose a diffeomorphism $\Upsilon:\Delta_L\rightarrow U_L$ such that:
  \begin{enumerate}
  \item $\phi=\Upsilon^*J_X-J_L$ satisfies $\phi=\sum_1^{j-1} \zeta^i \phi_i+O(\zeta^j)$;
  \item $\Upsilon^*\Omega = v (1-\phi)^*\Omega_L$ with $v=1+O(\zeta^j)$. 
  \end{enumerate}
  where $O()$ is meant with derivatives. 
\end{prop}
Note that this statement is weaker than Proposition \ref{prop:normal-form-1} for the complex structure $\phi$ but gives more information on $v$. It is likely that one can obtain both the gauge of Proposition \ref{prop:normal-form-1} and $v\equiv1$, but this problem may involve a loss of derivatives which makes it difficult, so we will prove only Proposition \ref{prop:normal-form-2}.

The rest of this subsection is devoted to the proof of the Proposition.
So we are on the disc bundle $\Delta_L$, with a complex structure $J$ whose difference with $J_L$ is given by the tensor $\phi$ as described in Proposition \ref{prop:normal-form-1}.

We consider the infinitesimal action of the real part of the $(1,0)$ vector field (for $J_L$)
\begin{equation}
  \label{eq:25}
    \xi=f i \zeta \partial_\zeta + X
  \end{equation}
  where $X\in H^{1,0}$ is horizontal, and $i\zeta \partial_\zeta=R^{1,0}$ is the $(1,0)$ part of the Reeb vector field $R=\frac d{d\theta}$. The infinitesimal action on the complex structure $J_L$ is
  \begin{equation}
    \label{eq:26}
    \dbar \xi = 2\eta^{0,1} \otimes\big( \bar \zeta \partial_{\bar \zeta}f \zeta \partial_\zeta + \bar \zeta \partial_{\bar \zeta}X \big)
    + \big( \dbar_Hf + X \lrcorner d\eta \big) \otimes i\zeta\partial_\zeta + \dbar_HX.
  \end{equation}

  Observe that if $X= -\sharp \dbar_H f$ and $f$ is holomorphic along the fibers ($\partial_{\bar \zeta}f=0$), then $\dbar \xi=-\dbar_H\sharp \dbar_H f$ is purely horizontal, that is the vector field
  \begin{equation}
    \label{eq:27}
    \xi_f = f i \zeta \partial_\zeta -\sharp \dbar_H f
  \end{equation}
preserves infinitesimally the gauge of $\phi$ (of Proposition \ref{prop:normal-form-1}) if $f$ is holomorphic along the fibers. This is the flexibility alluded to on $U_L$ since $f$ is determined by its value on the boundary of $U_L$ and corresponds to perturbing this boundary. (This is actually the complexification of the action of the contactomorphisms of $\partial \Delta_L$ which was studied in \cite[§ 7]{Biq19}).


Let us now investigate the action of such deformations on the function $v$.
The model holomorphic volume form (\ref{eq:18}) can be written as
\begin{equation}
  \label{eq:19}
  \Omega_L = \eta^{1,0} \wedge \Omega_D, \quad \Omega_D \in H^0(K_DL^{\alpha-1}).
\end{equation}
Recall that $-K_D=(\alpha-1)L$ so the section $\Omega_D$ is canonical up to a multiplicative constant, that we can choose so that it satisfies
\begin{equation}
  \label{eq:33}
  i^{(n-1)^2} \Omega_D \wedge \overline{\Omega_D} = 8b_{n,\alpha} e^{-n\beta_*u} \omega_D^{n-1},
\end{equation}
where $b_{n,\alpha}$ was defined in (\ref{eq:38}). The holomorphic volume form for the complex structure $J$ can be written as
\begin{equation}
  \label{eq:20}
  \Omega = v (1-\phi)^*(\eta^{1,0} \wedge \Omega_D)
\end{equation}
for some function $v$. Then the infinitesimal action of $\xi_f$ is given by
\begin{align}
  (\mathcal{L}_{\xi_f}\Omega_L)^{1,0} &= \partial(\xi_f \lrcorner \Omega_L) \notag \\
                     &= \partial( \frac12 f \Omega_D + \eta^{1,0} \wedge \sharp\dbar_H f \lrcorner \Omega_D) \notag \\
                     &= \eta^{1,0} \wedge \big( i \zeta\partial_\zeta(f \Omega_D) - \partial_H( \sharp \dbar_H f \lrcorner \Omega_D) \big) \notag \\
                       &= i\eta^{1,0} \wedge \big(\zeta \partial_\zeta(f \Omega_D) -  2\Delta_\partial (f\Omega_D)\big).\label{eq:11}
\end{align}

We now prove the Proposition by induction on $j$. From Proposition \ref{prop:normal-form-1} we know that Proposition \ref{prop:normal-form-2} is true for $j=1$. (One of course does not need the full strength of Proposition \ref{prop:normal-form-1}, it is sufficient to construct by hand a good enough $\Upsilon$). So suppose it is true for $j$ and let us prove it for $j+1$. 

We begin by rectifying $\phi$ and constructing $\phi_j$. Let $\phi'$ contain the order $j$ terms of $\phi$, that is $\phi=\sum_1^{j-1}\phi_i\zeta^i + \phi' + O(\zeta^{j+1})$. We first consider the infinitesimal problem, so we consider the infinitesimal action of for some vector field $\xi=if \zeta\partial_\zeta+X$ of order $j$. Denote by $\phi'_{HH}$, $\phi'_{HV}$, $\phi'_{VH}$ and $\phi'_{VV}$ the various components of $\phi'$ in the decomposition $T_{\mathbb{C}}=\mathbb{C} \zeta\partial_\zeta\oplus H_{\mathbb{C}}$. From the first term in (\ref{eq:26}) we see that we can find $\xi=f i\zeta\partial_\zeta+X$ of order $j$ so that $(\dbar \xi)_{VH}=\phi'_{VH}$ and $(\dbar \xi)_{VV}=\phi'_{VV}$; note that $f$ and $X$ are defined up to fibrewise holomorphic terms, that is up to terms $\zeta^j(f_j i\zeta\partial_\zeta+X_j)$. Thanks to the action of $\xi$ we are now reduced to the case where $\phi'_{VH}=0$ and $\phi'_{VV}=0$, that is $\phi'$ has the form
\begin{equation}
  \label{eq:29}
  \phi' = \tilde \phi + \alpha \otimes i\zeta\partial_\zeta + \phi'_{HH},
\end{equation}
where $\alpha\in H_{0,1}$ and $\phi'_{HH}$ are terms of order $j$. The integrability equation $\dbar \phi'+\frac12[\phi',\phi']=0$ implies
\begin{equation}
  \label{eq:30}
  \bar \zeta \partial_{\bar \zeta}\phi' = O(\zeta^{j+1})
\end{equation}
and in particular $\partial_{\bar \zeta}\alpha=0$ and $\partial_{\bar \zeta}\phi'_{HH}=0$. Defining a horizontal vector field $X$ by $X\lrcorner d\eta=\alpha$ we have from (\ref{eq:26})
\[ \dbar X = \alpha \otimes i\zeta\partial_\zeta + \dbar_H X \]
and therefore the infinitesimal action of $X$ kills the residual term $\alpha$.

Summarizing, we get a vector field $\xi$ of order $j$ which infinitesimally brings $\phi'$ to a purely horizontal $\tilde \phi_j$. Choosing some smooth Kähler metric near $D$ and its exponential, we can now apply the diffeomorphism $\exp(\Re \xi)$ to $(J,\Omega)$: we consider $\Upsilon'=\Upsilon\circ \exp(\Re \xi)$ and the corresponding pair $(\phi',v')$. Since the dominant term of the action is given by the infinitesimal action of $\xi_f$ (the exponential gives only higher order terms) this implies that:
\begin{itemize}
\item $\phi'=\sum_1^{j-1}\phi_j\zeta^j + \tilde \phi_j + O(\zeta^{j+1})$ with $\tilde \phi_j$ purely horizontal;
\item we still have $v'=1+O(\zeta^j)$ since $\xi$ has order $j$.
\end{itemize}
Also note that (\ref{eq:30}) implies that $\tilde \phi_j$ is holomorphic with respect to $\zeta$, that is has the required form $\zeta^j\tilde \phi_j$. So we have proved the induction for $\phi$, and there remains to rectify the function $v$.

One can remark here that since $\xi$ is obtained by solving a problem disk per disk, it a priori has the same regularity as $\phi'$, when one would expect a better regularity (one more derivative on $X$ and two more on $f$). This explains the loss of derivatives alluded to above.

We now must kill the order $j$ term of $v'$. From $\dbar \Omega=0$ and the fact that $\phi$ is purely horizontal up to order $j$, we deduce that the $j$-th order term of $v'$ is holomorphic along the fiber (that is with respect to $\zeta$), which means that it can be written as $v'_j \zeta^j$ for a section $v'_j$ of $L^{-j}$. From the Weitzenböck formula $2\Delta_\partial=\nabla^*\nabla+j$ on sections of $K_DL^{\alpha-j-1}$, we deduce, for a section $f\Omega_D$ of $K_DL^{\alpha-j-1}$,
\begin{equation}
  \label{eq:28}
  - \zeta \partial_\zeta(f \Omega_D) +  2\Delta_\partial (f\Omega_D) = (\nabla^*\nabla+\alpha-1)(f \Omega_D).
\end{equation}
Since $\alpha>1$ the operator $\nabla^*\nabla+\alpha-1$ is an isomorphism, and therefore the term $v'_j\zeta^j$ can be killed infinitesimally by a vector $\xi_f$ for some $f=f_j \zeta^j$ and $f_j\Omega_D$ is a section of $L^{-j}$. We now apply as above the diffeomorphism $\exp(\Re \xi_f)$ to $(J,\Omega)$: we obtain a new pair $(\phi'',v'')$. Since the dominant term of the action is given by the infinitesimal action of $\xi_f$, this implies that:
\begin{itemize}
\item the order $j$ term of $v'$, that is $v'_j\zeta^j$, is killed by $\exp(\Re \xi_f)$, so we have now $v''=1+O(\zeta^{j+1})$;
\item as observed above, the vector field $\xi_f$ infinitesimally preserves the horizontal gauge, it actually modifies the order $j$ term of $\phi'$ by $-\dbar_H \sharp \dbar_H f$ so we get an order $j$ term $\zeta^j(\tilde \phi_j+\dbar_H\sharp\dbar_Hf_j)$.
\end{itemize}
So by choosing $\phi_j=\tilde \phi_j+\dbar_H\sharp\dbar_Hf_j$ we have finished to prove the Proposition for $j+1$.\qed

\begin{prop}\label{prop:def-j0}
  Given the normal form from Proposition \ref{prop:normal-form-2}, the smallest integer $j_0$ such that $\phi_{j_0}\neq0$ is intrinsic. Actually $j_0$ is the largest integer such that there exists a diffeomorphism $\Upsilon:\Delta_L\rightarrow \mathcal{U}_L$ such that $\Upsilon^*\Omega$ and $\Omega_L$ coincide up to order $j_0-1$.
\end{prop}
\begin{proof}
  Let $\ell$ be the order of coincidence of $\Omega$ and $\Omega_L$ as defined in the proposition, and $j_0$ the first integer so that $\phi_{j_0}\neq0$. By definition we have $\ell\geq j_0-1$. On the other hand, if $\Omega$ and $\Omega_L$ coincide up to order $\ell$, then the complex structures $J$ and $J_L$ also coincide up to order $\ell$, since they are determined by the holomorphic forms. If $j_0<\ell+1$ it therefore means that our term $\zeta^{j_0}\phi_{j_0}$ can be killed by a diffeomorphism: as it is the first nonvanishing term, it means that it can be killed by an infinitesimal diffeomorphism which also preserves the condition $v=1$. The infinitesimal action of diffeomorphisms was calculated in the proof of Proposition \ref{prop:normal-form-2}, and it turns out that at each step the vector field is unique (to preserve the normal form of $\phi$ one needs $\xi=\xi_f$ with $f$ holomorphic along the fibres; then the condition $v=1$ kills this remaining degree of freedom). Therefore it is impossible to kill the first term by an infinitesimal diffeomorphism, so we must have $\ell=j_0-1$.
\end{proof}

\begin{rem}
\label{rem j0}
The $T^{1,0}_D$-valued $(0,1)$-form $\phi_1\zeta$ is $\dbar$-closed and represent the extension class $\kappa$ of $N_D$ by $T_D$ in $H^1(D, T_D\otimes N_D^{-1})$. In particular, we have $j_0=1$ if, and only if $\kappa\neq 0$, or equivalently, iff the normal exact sequence \eqref{NES} does not split, cf the first paragraph of section~\ref{sec normal}. 
\end{rem}

\subsection{The real volume form}
\label{sec:real-volume-form}

Using Propositions \ref{prop:normal-form-1} and \ref{prop:normal-form-2} we can now give an optimal development of the volume form $\Omega\wedge \overline \Omega$ by comparison with $\Omega_L \wedge \overline{\Omega_L}$. We develop the tensor $\phi$ from Proposition \ref{prop:normal-form-1} in powers of $\zeta$ as
\begin{equation}
  \label{eq:1}
  \phi = \phi_1 \zeta + \phi_2 \zeta^2 + \cdots
\end{equation}
One can identify $\phi_j$ with a section on $D$ of $\Omega^{0,1}_D(T^{1,0}_D\otimes L^{-j})$. This is valid up to some order which we can take large enough for what follows.

Suppose for simplicity that $\phi_1\neq0$. The linearization with respect to $\phi$ of $(1-\phi)^*\Omega_L$ is $\phi\lrcorner \Omega_L$: here $\phi\in \Omega^{0,1}\otimes T^{1,0}$, $\Omega_L\in\Omega^{n,0}$ and $\phi\lrcorner \Omega_L\in\Omega^{n-1,1}$. It follows from Proposition \ref{prop:normal-form-2} that
\[ \Upsilon^*\Omega = \Omega_L - \phi_1\zeta \lrcorner \Omega_L + O( \zeta^2 ). \]
Therefore
\begin{align*}
  \Upsilon^*(\Omega\wedge \overline \Omega) &= \Omega_L \wedge \overline{\Omega_L} + (\phi_1\zeta\lrcorner\Omega_L) \wedge \overline{(\phi_1\zeta\lrcorner\Omega_L)} + O(\zeta^3) \Omega_L\wedge \overline{\Omega_L}\\
  &= \Omega_L\wedge \overline{\Omega_L}\big( 1 - |\phi_1\zeta|^2 + O(\zeta^3) \big).
\end{align*}
The essential feature in this formula is that no linear term in $\phi_1$ appears because $\phi_1\lrcorner\Omega_L$ is of type $(n-1,1)$. This makes possible for the correction term to exhibit a sign which will be important for us.

The same calculation when the first nonzero term is $\phi_{j_0}$ leads to:
\begin{prop}\label{prop:normal-form-3}
  Suppose $\Upsilon$ is given by Proposition \ref{prop:normal-form-2} and that the first nonzero term in (\ref{eq:1}) is of order $j_0$. Then
  \[ \Upsilon^*(\Omega\wedge\overline\Omega) = \Omega_L\wedge \overline{\Omega_L}\big( 1 - |\phi_{j_0}\zeta^{j_0}|^2 + O(\zeta^{2j_0+1}) \big). \]\qed
\end{prop}

\section{The Tian-Yau metric and its formal perturbation}
\label{sec:tian-yau-metric}

Fix a Kähler-Einstein metric $\omega_D$ on $D$ such that $\Ric \omega_D=\omega_D$ and a Hermitian metric $h$ on $L=N_D$ such that $i\Theta(L,h)=\frac1{\alpha-1}\omega_D$. Let $\sigma$ be the section of $[D]$ vanishing on $D$, and $\Omega$ the holomorphic volume form on $X\setminus D$, normalized as in section \ref{sec:complex-volume-form} (again, only some power is well-defined). The Tian-Yau metric \cite{TY2} of $(X,D)$ is defined in the following way: there is a unique Ricci flat Kähler metric $\omega_{\rm TY}=dd^c\Phi_{\rm TY}$ on $X\setminus D$, solving the Monge-Ampère equation
\begin{equation}
  \label{eq:4}
  (dd^c\Phi_{\rm TY})^n = i^{n^2} \Omega \wedge \overline \Omega
\end{equation}
and such that near $D$
\begin{equation}
  \label{eq:2}
  \Phi_{\rm TY} \sim |\sigma|^{-2\beta_*}.
\end{equation}
Here it may be necessary to adjust the metric on $L$ by a multiplicative constant in order to obtain exactly (\ref{eq:4}).

The asymptotic Calabi-Yau cone of the Tian-Yau metric is easily described using the formalism of section \ref{sec:reduction-an-ode}: it is the bundle $L$ equipped with the cone metric
\begin{equation}
  \label{eq:3}
  \omega_{{\rm TY},L} = dd^c e^{-\beta_*u}.
\end{equation}
Indeed from (\ref{eq:19}) and (\ref{eq:33}) the holomorphic volume form satisfies
\begin{align}
\notag
  i^{n^2}\Omega_L \wedge \overline{\Omega_L} &= 8b_{n,\alpha} i e^{-n\beta_*u} \eta^{1,0}\wedge \eta^{0,1}\wedge \omega_D^{n-1} \\
\notag &= b_{n,\alpha} e^{-n\beta_*u} du \wedge d^cu \wedge \omega_D^{n-1}\\
  &= b_{n,\alpha} \Xi.  \label{eq:31}
\end{align}
Comparing with (\ref{eq:35}), where $-\phi'_{{\rm TY}, L}(u)=\beta_*e^{-\beta_*u}$, we obtain exactly
\begin{equation}
  \label{eq:34}
  \omega_{{\rm TY},L}^n = i^{n^2}\Omega_L \wedge \overline{\Omega_L}
\end{equation}
which finishes the description of the Calabi-Yau cone. In particular the Calabi-Yau cone has radius $R=2e^{-\frac12 \beta_*u}=2|\sigma_L|^{-\beta_*}$, where $\sigma_L$ is the tautological section of $L$ over $L$.

\subsection{Asymptotic expansion of the Tian-Yau potential}
\label{sec:asympt-expans-tian}

We now use the diffeomorphism $\Upsilon:\Delta_L\rightarrow U_L$ of section \ref{sec:tubul-neighb-d} to give an asymptotic expansion of $\Phi_{\rm TY}$. It will be convenient to distinguish objects on $\Delta_L$ from objects pulled-back by $\Upsilon$ by an index (for example $d^c$ is for the complex structure $J$ on $U_L$ and $d^c_{J_L}$ for $J_L$ on $L$). To simplify notation, we will use Proposition \ref{prop:normal-form-2} as if $j=\infty$ that is as if $\phi$ was purely horizontal at all orders: indeed taking the order $j$ large enough in the Proposition leads only to small error terms which do not play any role, see section \ref{sec:estim-error-obstr} where the error terms are given.

On $L$ we have the exact potential $\Phi_0=|\sigma_L|^{-2\beta_*}=R^2/4$. Because the tensor $\phi$ expressing the difference $\Upsilon^*J-J_L$ is horizontal and $\Phi_0$ depends only on $R$ (or $|\sigma_L|$), it follows that 
\[ dd^c\Phi_0 = dd^c_{J_L}\Phi_0. \]
If we use $\Phi_0$ in equation (\ref{eq:4}), the first error therefore comes from the difference between $\Omega\wedge\overline\Omega$ and $\Omega_L\wedge\overline{\Omega_L}$:
\begin{equation}
  \label{eq:6}
  \frac{(dd^c\Phi_0)^n}{i^{n^2} \Omega\wedge\overline{\Omega} } = 1+|\phi_{j_0}|^2+O(|\sigma|^{2j_0+1}).
\end{equation}
Extend $\Phi_0$ as a potential on $X\setminus D$ and define $\omega_0=dd^c\Phi_0$ and $g=\frac{(dd^c\Phi_0)^n}{i^{n^2} \Omega\wedge\overline{\Omega} }$, so that the expansion of $g$ at infinity is given by the RHS of (\ref{eq:6}). To find the exact solution $\Phi_{\rm TY}$ we must solve the Monge-Ampère equation on $X\setminus D$
\begin{equation}
  \label{eq:5}
  (\omega_0+dd^cf)^n = g \omega_0^n
\end{equation}
Observe that $|\phi_{j_0}|^2=O(|\sigma|^{2j_0})=O(R^{-\frac{2j_0}{\beta_*}})$, and that since $\alpha<n+1$, the decay rate satisfies
\begin{equation}
  \label{eq:9}
  \frac{2j_0}{\beta_*} > 2j_0 \geq 2.
\end{equation}
We can now apply \cite[Theorems 2.1 and 2.4]{CH1}. There are two cases, depending whether the decay rate $\frac{2j_0}{\beta_*}$ allows the Green's function ($R^{-2n+2}$) to be the first perturbation of $\Phi_0$, and an additional limit case:
\begin{itemize}
\item if $\frac{2j_0}{\beta_*} < 2n$, that is if $\alpha\geq1+j_0$, then one can solve (\ref{eq:5}) with $f=O(R^{-\frac{2j_0}{\beta_*}+2})$ (and the corresponding weighted decays for the derivatives);
\item if $\frac{2j_0}{\beta_*} > 2n$, that is if $\alpha<1+j_0$, then one can solve (\ref{eq:5}) with
  \begin{equation}
    \label{eq:7}
    f = \frac a{R^{2n-2}} + O(R^{-2n+2-\varepsilon});
  \end{equation}
\item if $\frac{2j_0}{\beta_*} = 2n$, we have an indicial root of the Laplacian (corresponding to the decay $R^{-2n+2}$ of the Green function) so we need to first kill by hand the term of order $R^{-2n}$ of $g$: we correct the potential by a function $f_0=aR^{-2n+2}\log R+O(R^{-2n+2})$ so that $\omega_0+dd^cf_0$ now satisfies (\ref{eq:5}) with $g-1=O(R^{-2n-\varepsilon})$, then we can proceed as in the previous case.
\end{itemize}
In each case the solution is unique. We will introduce the notation 
\[\nu_0:=\frac{j_0}{\beta_*}\]
 so that the qualitative asymptotics of the solution $f$ depend on the value of $\nu_0$ in relation to $n$. This gives the first part of the following result:
\begin{thm}\label{thm:dev-TY}
  The potential of the Tian-Yau metric on $X\setminus D$ has the expansion
  \begin{equation}
    \label{eq:8}
    \Upsilon^*\Phi_{\rm TY} = \frac{R^2}4 +
    \begin{cases}
      O(R^{-2\nu_0+2}) & \text{ if } \nu_0< n, \\
      a \frac{\log R}{R^{2n-2}}+ O(R^{-2n+2}) & \text{ if }\nu_0=n, \\
      \frac a{R^{2n-2}} + O(R^{-2n+2-\varepsilon}) & \text{ if }\nu_0>n,
    \end{cases}
  \end{equation}
  where $\varepsilon>0$ is small. Moreover in the last cases we have $a>0$.
\end{thm}

\begin{proof} Here we work with the Tian-Yau metric as background metric. 
  There remains to prove the statement on the sign of $a$. The two cases $\nu_0=n$ and $\nu_0>n$ are different in nature, since the argument is local in the first case, and global in the second case.

In the first case ($\nu_0=n$) the coefficient $a$ is formally determined by the asymptotic terms of equation (\ref{eq:5}): the linearization of the LHS of (\ref{eq:5}) is $-\Delta f$, and on the asymptotic cone $-\Delta(R^{-2n+2}\log R)=(2n-2)R^{-2n+2}$; since $g=1+bR^{-2n+2}+O(R^{-2n+2-\varepsilon})$ with $b>0$ (see (\ref{eq:6}), we have $a>0$.

The second case ($\nu_0>n$) requires more work. For a one form $u$ on $X\setminus D$ denote by $\nabla^-u$ the $J$-anti-invariant part of $\nabla u$:
  \begin{align*}
    (\nabla^-u)_{X,Y} &= \frac12 \big( (\nabla_Xu)_Y - (\nabla_{JX}u)_{JY} \big) \\
    &= -\frac12 \omega\big( (\mathcal{L}_{\sharp u}J)X,Y \big).
  \end{align*}
Here $\sharp$ denotes the Riemannian duality between $TX$ and $T^*X$.
  In particular $\nabla^-u=0$ if and only if $\sharp u$ is a holomorphic vector field. Given a function $f$ the $(0,2)$ part of $\nabla^-df$ is the familiar operator $\dbar \sharp \dbar f$. Moreover one has the Weitzenböck type formula \cite[(2.51) and (2.53)]{Besse}
  \begin{equation}
    \label{eq:10}
    \delta \nabla^-u = \frac12 \Delta u - \Ric u.
  \end{equation}

  The Tian-Yau potential $\Phi=\Phi_{\rm TY}$ satisfies $\Delta\Phi=-4n$ and we deduce from (\ref{eq:10}) that
  \begin{equation}
    \label{eq:12}
    \delta \nabla^- d\Phi = \frac12 \Delta d \Phi = \frac12 d \Delta \Phi = 0.
  \end{equation}
  By integration by parts, denoting $\vec n$ the outward normal vector to the hypersurface $S_R$ of level $R$:
  \begin{align}
    0 &= \int_{X\setminus D} \langle \delta\nabla^-d\Phi , d\Phi \rangle \notag \\
    &= \int_{X\setminus D} |\nabla^-d\Phi|^2 - \lim_{R\rightarrow\infty} \int_{S_R} \nabla^-d\Phi(\vec n, \nabla\Phi) \vec n \lrcorner d\vol. \label{eq:13}
  \end{align}
  From the development (\ref{eq:1}) of the complex structure and the development (\ref{eq:8}) of the potential, and denoting $p=\inf(2n,\frac{j_0}{\beta_*})$, we have:
  \begin{align*}
    g &= g_{cone} + O(R^{-p}), \\
    \vec n &= \partial_R + O(R^{-p}), \\
    \nabla\Phi &= 2R\partial_R + O(R^{-p+1}), \\
    \nabla^-dR^2 &= O(R^{-p}).
  \end{align*}
  Under our hypothesis that $\frac{2j_0}{\beta_*}>2n$, we have that $p>n$ and therefore
  \begin{equation}
    \label{eq:14}
    \lim_{R\rightarrow\infty} \int_{S_R} \nabla^-dR^2(\vec n, \nabla\Phi) \vec n \lrcorner d\vol =
    \lim_{R\rightarrow\infty} \int_{S_R} \nabla^-dR^2(\partial_R, 2R\partial_R) \partial_R \lrcorner d\vol
  \end{equation}
  since the other (quadratic) terms are too small to give a nonzero limit. But the term at order $R^{-j_0/\beta_*}$ of $\nabla^-dR^2$ is linear in $\phi_{j_0}$ and $\overline{\phi_{j_0}}$, so has Fourier coefficients of orders only $\pm j_0$ so its integral on $S_R$ against constant terms vanishes.

  Therefore our boundary term reduces to
  \begin{equation}
    \label{eq:15}
    \lim_{R\rightarrow\infty} \int_{S_R} a \nabla^-dR^{-2n+2}(\partial_R, 2R\partial_R) \partial_R \lrcorner d\vol.
  \end{equation}
  But $\nabla^-dR^{-2n+2}=4n(n-1) \frac{((dR)^2)^-}{R^{2n}} + l.o.t.$ and it follows that the limit in (\ref{eq:15}) is $a\ell$ for some $\ell>0$. Coming back to (\ref{eq:13}) we finally obtain
  \begin{equation}
    \label{eq:16}
    \int_{X\setminus D} |\nabla^-d\Phi|^2 = a \ell
  \end{equation}
  and it follows that $a\geq0$, with $a=0$ if and only if $\nabla^-d\Phi=0$, that is if $\nabla\Phi$ is a (real) holomorphic vector field and an homothety of the metric. Since $\nabla\Phi\sim 2R\partial_R$ we have a cone and $\nabla\Phi$ is the dilation vector field. Considering the asymptotics at infinity, the cone can only be the asymptotic cone of the Tian-Yau metric, that is $\overline L\setminus D$, where $\overline L$ is the one point compactification of $L$ considered in section \ref{sec:about-automorphisms}. This can be smooth only in the case where $X=\mathbb{P}^n$ and $D$ is an hyperplane, which was excluded.
\end{proof}




\begin{rem}
  The result on the sign of $a$ in this Theorem is opposite to that obtained on crepant resolutions of a finite quotient of $\mathbb{C}^n$, see \cite[Theorem 8.2.3]{Joy00}.
\end{rem}
\begin{rem}
  The proof of Theorem \ref{thm:dev-TY} is reminiscent of the argument in \cite[§ 3.2]{BiqHei23}. The difference is that the argument in \cite{BiqHei23} is in real Riemannian geometry, so formula (\ref{eq:10}) is replaced by $\nabla \Delta f = \Delta \nabla f + \Ric \nabla f$, and the complex Hessian $\nabla^-df$ by the real Hessian of $f$.
\end{rem}

\subsection{The formal perturbation}
\label{sec:formal-perturbation}

The Tian-Yau metric is Ricci flat, but we want to glue it with a metric with $\Ric=\mu$. Given the scale $\varepsilon$, this means we will construct a perturbation $\omega_{{\rm TY},\varepsilon}$ of $\omega_{\rm TY}$ such that $\Ric(\omega_{{\rm TY},\varepsilon})=\varepsilon\mu \omega_{{\rm TY},\varepsilon}$. This perturbation will match better the metric constructed on the normal bundle $L$ and will be needed to successfully perform the gluing later, cf Remark~\ref{rem TY formel nec}.

Since $X\setminus D$ is not compact, this is of course not possible, but it can be done formally up to any order. Denote the development of the potential $\phi_{\beta_*,L}$ near the cone singularity by
\begin{equation}
  \label{eq:23}
  \phi_{\beta_*,L}(x) = \frac14 r^2 + a_2r^4 + a_3r^6 + \cdots
\end{equation}

We want to solve the Kähler-Einstein equation on $X\setminus D$:
\begin{equation}
  \label{eq:24}
  P_\varepsilon(\varphi):=\log \frac{(\omega_{\rm TY}+dd^c \varphi)^n}{i^{n^2}\Omega\wedge \overline \Omega} + \varepsilon\mu(\Phi_{\rm TY}+\varphi)=0.
\end{equation}

\begin{prop}
\label{prop formal expansion}
  For any integer $k$ there is a development $\varphi=\varepsilon\varphi_1+\cdots+\varepsilon^k\varphi_k$ such that
  \begin{enumerate}
  \item there is a formal development $P_\varepsilon(\varphi)=\sum_{i>k} \varepsilon^ip_i$ with $p_i=O(R^{2i})$ (and the same for the weighted derivatives $R^j\nabla^j$); moreover, on any domain $R\leq C\varepsilon^{-\theta/2}$ the series for $P_\varepsilon(\varphi)$ converges if $\varepsilon$ is small enough, and one has an estimate
\[ |R^j\nabla^jP_\varepsilon(\varphi)| \leq C_{k,j} \varepsilon^{k+1}R^{2k+2}. \]
  \item at infinity one has $\varphi_k(y)=a_{k+1}R^{2k+2}+O(R^{2k+2-2\nu_0})+O(R^{-2n+2})$; again the same estimates hold on the weighted derivatives $R^j\nabla^j$.
  \end{enumerate}
\end{prop}
\begin{proof}
  We solve inductively the equation (\ref{eq:24}) in powers of $\varepsilon$. The first equation is
  \[ -\Delta \varphi_1 + \mu \Phi_{\rm TY} = 0. \]
  Because $\phi_{\beta_*,L}$ is an exact solution on the cone, the function $a_2R^4$ is a solution of $\Delta(a_2R^4)=\mu\frac{R^2}4$ up to the order where the equations for $J$ and $J_L$ differ, that is up to the order where $\Omega\wedge \overline \Omega$ and $\Omega_L\wedge \overline \Omega_L$ differ:  so given (\ref{eq:8}) we obtain
  \[ -\Delta (a_2R^4) + \mu \Phi_{\rm TY} = O(R^{2-2\nu_0}). \]
  Therefore we can find $\varphi_1=\chi a_2R^4+\psi_1$ (where $\chi$ is a cutoff function near infinity) by solving $\Delta \psi_1=-\Delta (\chi a_2R^4)+\mu \Phi_{\rm TY}$, which is possible with $\psi_1=O(R^{4-2\nu_0})+O(R^{-2n+2})$. (The second term is present for the case where $4-2\nu_0<-2n+2$). The error term in $P_\varepsilon(\varepsilon\varphi_1)$ is:
  \begin{itemize}
  \item the nonlinear terms in Monge-Ampère: $dd^c(\varepsilon\phi_1)=O(\varepsilon R^2)$ so the nonlinear terms are $O((\varepsilon R^2)^i)$ with $i>1$; applying the $\log$ gives the formal development in $\varepsilon$;
  \item $\varepsilon^2\mu\varphi_1$ which is also $O(\varepsilon^2R^4)$.
  \end{itemize}
  On the domain $R\leq C\varepsilon^{-\theta/2}$, we have $\varepsilon R^2\leq C\varepsilon^{1-\theta}$ which is small when $\varepsilon\rightarrow0$, and it follows that the series of the logarithm converges, which implies the estimate.
  
  Continuing this process inductively we get the proposition.
\end{proof}

\section{Gluing and linear analysis}
\label{gluing}

\subsection{Gluing with the conic Calabi Ansatz and the Tian-Yau metric}
Recall that on $L$, we have the coordinate $u:=\log |v|^2_h$, as well as a Kähler-Einstein potential $\phi_{\beta, L}=\phi_{\beta, L}(u)$ defined for larger and larger values of $u$ as $\beta$ approaches $\beta_*$; more precisely it is defined in the zone $(u\le u_\beta)$ where $u_\beta = -\frac{1}{n\beta_*}\log(\beta-\beta_*)$ up to a constant, cf \eqref{ubeta}. 

In order to glue the potential $\phi_{\beta, L}$ to the Tian-Yau potential via $\Upsilon$, one first needs to rescale $L$ so that most of the domain of $\phi_{\beta, L}$ is included in the domain of $\Upsilon$. We have already identified the correct rescaling $\lambda_\ep : L \to L$ in \eqref{scaling}. It will be convenient to introduce the following notation
\begin{equation}
\label{phi ep}
\phi_{\beta, L, \ep}:=(\lambda_\ep)_*\phi_{\beta, L}
\end{equation}
as well as
\[\Upsilon_\ep:=\Upsilon \circ \lambda_\ep\]
which is defined on larger and larger regions in $L$. Finally, if $\zeta$ is the variable in $L$, we will write 
\begin{equation}
\label{zeta ep}
\zeta_\ep:=\lambda_\ep^* \zeta= \ep^{\frac{1}{2\beta_*}}\zeta.
\end{equation}


Moreover, in order for the potential to match the Tian-Yau potential, the gluing needs to be done near $u_\ep$ (i.e. on a zone of the form $|u-u_\ep| \le 1$) satisfying
\[1 \ll u_\ep \ll  -\frac{1}{n\beta_*} \log(\beta-\beta_*),\]
or equivalently in terms of the variable $r=2e^{-\frac 12 \beta_* u}$ 
\begin{equation}
\label{zone beta}
(\beta-\beta_*)^{\frac 1n} \ll r_\ep^2 \ll 1,
\end{equation}
where $r_\ep:=2e^{-\frac 12 \beta_*u_\ep}$. We will choose $r_\ep$ of the form $r_\ep=\ep^{\frac{1-\theta}{2}}$ for some $\theta\in (0,1)$ to be determined later so that 
\begin{equation}
\label{ep re}
\frac{\ep}{r_\ep^2}=\ep^\theta.
\end{equation} 

We fix once and for all a point $x_0\in X\setminus U_L$ and we set $R=d_{g_{\rm TY}}(\cdot, x_0)$. Next, we introduce a positive function $\rho=\rho(\beta, \ep)$ on $X$ by 
  \[\rho^2:=\begin{cases}
-\frac{4}{\beta_*}\phi'_{\beta,L}\circ \Upsilon_\ep^{-1}  &\mbox{on} \quad \Upsilon_\ep(\{u \le u_\ep\})\\
\ep \, R^2  &\mbox{on} \quad \Upsilon_\ep(\{u \ge u_\ep\}\cap \Delta_L)\\
\ep  &\mbox{on} \quad X\setminus U_L
  \end{cases}
  \]
The function $\rho$ has values in $[\sqrt{\ep}, 2\sqrt{\alpha/\beta_*}]$. More precisely, it is equal to $2\sqrt{\alpha/\beta_*}$ on $D$, then "decreases" to reach the value $r_\ep$ near the gluing zone, and from there keeps decreasing to $\sqrt{\ep}$.  \\

We fix two positive integers $j,k$ and we consider the diffeomorphism $\Upsilon$ from Proposition~\ref{prop:normal-form-2} as well as the $k$-th order formal perturbation of the Tian-Yau potential $\Phi_{{\rm TY}, \ep}=\Phi_{\rm TY}+\varphi$ constructed in section~\ref{sec:formal-perturbation}. We will determine suitable values for $j$ and $k$ further in the text.

Let $\chi: [0,+\infty)\to [0,1]$ be a non-increasing function which is identically $0$ on $[0, \frac 12]$ and identically $1$ on $[2, +\infty)$, and let $\chi_\ep : X\to [0,1]$ be defined by $\chi_\ep = \chi(\frac \rho{r_\ep})$, so that
\[\chi_\ep =
\begin{cases}
1 & \mbox{on } \quad  \rho\ge 2 r_\ep \\
0 & \mbox{on } \quad \rho\le \frac 12 r_\ep 
\end{cases}\]
Finally, we introduce the potential
\[\vp_{\beta, \ep}:=\chi_\ep \,  (\Upsilon_\ep)_*\phi_{\beta, L} +(1-\chi_\ep) \, \ep \Phi_{{\rm TY}, \ep}\]
which is well-defined globally on $X$. In the gluing zone, we have
\begin{equation}
\label{pot gluing}
\Upsilon_\ep^*\vp_{\beta, \ep}=\phi_{\beta, L}+(1-\chi_\ep) (\Upsilon_\ep^*(\ep \Phi_{{\rm TY}, \ep})-\phi_{\beta, L}).
\end{equation}

We set $\om_{\beta, \ep}:=dd^c \vp_{\beta, \ep}$, which is a Kähler metric on $X\setminus D$ with cone singularities of angle $2\pi \beta$ along $D$. We denote by $g_{\beta, \ep}$ the corresponding Riemannian metric. \\

We will solve the Kähler-Einstein equation $\Ric(\omega_{\beta,\ep}+dd^c \varphi)=\mu \omega_{\beta,\ep}+ dd^c \varphi$ under the form
\begin{equation}
\label{eq:P}
 P_{\beta,\ep} (\varphi) := \log \frac{(\omega_{\beta,\ep}+dd^c \varphi)^n}{\ep^n i^{n^2}\Omega \wedge \overline \Omega} + \mu(\varphi_\beta + \varphi) = 0,
\end{equation}
where the normalization factor $\ep^n$ comes from the rescaling $\lambda_\ep$. Indeed, on $L$ the constants were choosen so that one has exactly $(dd^c \phi_{\beta,L})^n= e^{-\mu \phi_{\beta,L}} i^{n^2}\Omega_L\wedge \overline{\Omega}_L$, see (\ref{eq:39}) and (\ref{eq:31}); therefore, we get 
\begin{equation}
\label{MA ep}
(dd^c \phi_{\beta,L, \ep})^n=\ep^n e^{-\mu \phi_{\beta, L,\ep}} i^{n^2}\Omega_L\wedge \overline{\Omega}_L
\end{equation} 
with the notation \eqref{phi ep}, cf also \eqref{scale vol} further below. We introduce the notations
\begin{equation}
\label{decomp P}
L_{\beta, \ep}:=\Delta_{\omega_{\beta, \ep}}+\mu_\beta, \quad \mbox{and} \quad Q_{\beta, \ep}:=P_{\beta, \ep}-P_{\beta, \ep}(0)-L_{\beta, \ep}.
\end{equation}
An elementary computation shows that if $|dd^c f|_{\omega_{\beta, \ep}}\le 1$, then 
\begin{equation}
\label{Q}
|Q_{\beta, \ep}(f)|\le C |dd^c f|_{\omega_{\beta, \ep}}^2
\end{equation}
for some $C>0$ independent of $f, \beta, \ep$. 
\subsection{Functional spaces}
Fix a real number $\delta$. For any function $f$ on $X$, we define the weighted norms (which depends on $\beta$ and $\ep$):
\begin{equation}
  \label{norm 1}
   \|f\|_{C^{0}_\delta} = \sup_X \rho^{\delta} |f|, \quad 
 \|f\|_{C^{0,\alpha}_\delta} = \|f\|_{C^{0}_\delta}+ [\rho^{\delta}f]_\alpha
\end{equation}
where the semi-norm $[f]_\alpha$ is also weighted:
\begin{equation}
  \label{norm 2}
  [f]_\alpha = \sup_{d_{g_{\beta,\ep}}(x,y) < \mathrm{inj}} \min\{\rho^{ \alpha}(x), \rho^{\alpha}(y)\} \cdot \frac{|f(x)-f(y)|}{d_{g_{\beta,\ep}}(x,y)^\alpha}.
\end{equation}
As for the $C^{2, \alpha}$ norm, since we will in general be working with cone angles $2\pi \beta >\pi$, we cannot require to control all the second order derivatives but rather only the mixed ones (normal to $D$), as observed by Donaldson \cite{Don}. Therefore we need to distinguish two regions $X=A\cup B$ where $A=(r\ge 1)$ and $B=(r\le 1)$. On $B$, i.e. away from $D$, we set
\begin{equation}
  \label{norm 3}
  \|f\|_{C^{2,\alpha}_\delta(B)} = \sup \sum_{0\leq j\leq 2} \rho^{\delta+j}|\nabla^jf|_{g_{\beta,\ep}} + [\rho^{\delta+2}\nabla^2 f]_\alpha. 
\end{equation}

Now we need to define $ \|f\|_{C^{2,\alpha}_\delta(A)}$; on $A$ the weight function is essentially one so it is irrelevant. The function $\Upsilon^*f|_A$ lives on a smaller and smaller neighborhood of the zero section in $\Delta_L$ as $\ep$ decreases to $0$. Consider the dilation $\lambda_\ep : L\to L$, $ (x,v)\mapsto (x,\ep^{\frac 1{2\beta_*}}v)$. Then $f_\ep:=\lambda_\ep^*\Upsilon^*f|_A$ lives on $\Delta_L$, and we are going to set
\[ \|f\|_{C^{2,\alpha}(A)}:= \|f_\ep\|_{C^{2,\alpha}(\Delta_L)}\]
where the $C^{2,\alpha}(\Delta_L)$ norm is the one defined by Donaldson, whose construction we recall in the paragraph below. Finally, we set
\begin{equation}
  \label{norm 4}
  \|f\|_{C^{2,\alpha}_\delta} = \|f\|_{C^{2,\alpha}(A)}+\|f\|_{C^{2,\alpha}_\delta(B)}.
\end{equation}

\bigskip

In this paragraph, let $h$ he a function on $\Delta_L$. Let $V$ be a  chart $V$ with a holomorphic system of coordinates $(z_1, \ldots, z_n)$ where the zero section $D$ is given by $(z_1=0)$; set $z'=(z_2, \ldots, z_n)$. The flat cone Kähler metric is given by $dd^c (|z_1|^{2\beta}+|z'|^2)$ and the associated Riemannian metric is \[\overline g_{\beta}:=(dr^2+\beta^2r^2 d\theta^2)+g_{\mathbb C^{n-1}}\] where $r:=|z_1|^{\beta}$ and $\theta:= \mathrm{arg}(z_1)$. We consider the derivatives 
\[D_1:=\partial_r, \, D_2:=\frac{1}{\beta r}\partial_\theta \quad \mbox{ and} \quad  D_{2j-1}=\d_{x_j}, D_{2j}=\d_{ y_j}\]
 where $z_j=x_j+iy_j$ for $2\le j \le n$.
The $C^{0,\alpha}$ norm of $h$ is simply 
\[\|h\|_{C^{0,\alpha}(V)}=\sup_V |h|+ \sup_{x,y \in V}\frac{|h(x)-h(y)|}{d_{\gk}(x,y)^\alpha},\]
and we define the $\mathcal C^{2,\alpha}$ norm of $h$ on $V$ as
\[\|h\|_{C^{2,\alpha}(V)}=\sup_V(|h|+|\nabla^{\overline g_\beta}h|)+\sum_{i=1}^{2n}\sum_{j=3}^{2n} \|D_iD_jh\|_{C^{0,\alpha}(V)}+\|\Delta_{\overline g_\beta} h\|_{C^{0,\alpha}(V)}.\] 
Finally, one can cover $\Delta_L$ with finitely many coordinate charts $V_1, \ldots, V_N$, and we set $\|h\|_{C^{2,\alpha}(\Delta_L)}= \sum_{k=1}^N \|h\|_{C^{2,\alpha}(V_k)}$.\\

\begin{prop}
\label{schauder}
Let $L_{\beta,\ep}:=\Delta_{g_{\beta, \ep}}+\mu_\beta$. Then there exists a constant $C>0$ independent of $\beta, \ep$ such that for any $f\in C^{2,\alpha}_\delta(X)$, we have
\[ \|f\|_{C^{2,\alpha}_\delta} \le C\Big( \|f\|_{C^{0}_\delta}+\|L_{\beta, \ep}f\|_{C^{0,\alpha}_{\delta+2}}\Big).\]
\end{prop}

\begin{proof}
We first look at the zone $A=(\rho\ge 1)$ and we need to bound $\|f\|_{C^{2,\alpha}(A)}$. Given the definition of the norms involved, it amounts to proving the equivalent Schauder estimate for the operator $\Delta_{g_\beta}+\mu_\beta$ on $\Delta_L$, uniformly in $\beta$. This has been proved for the flat cone metric (i.e. for $\Delta_{\overline g_\beta}+\mu_\beta$) in \cite{Don} for a given fixed $\beta$, and the uniformity in the cone angle (even when $\beta\to 0$, which is irrelevant here though) has been proved in \cite[Theorem~6.1]{BG22}. It is classical (cf e.g. \cite[Remark~6.8]{BG22}) to reduce the estimate for the operator $\Delta_{g_\beta}+\mu_\beta$ to the one for the flat cone metric by proving that $g_\beta$ is $C^{\alpha}$-close to $\overline g_\beta$.  This, in turn, follows from the explicit description of $\omega_\beta$ in \eqref{omb} and the expansion \eqref{exp D} (use the relation $\rho = e^{\frac 1 2 \beta (u+ \phi_h)}$ where $\phi_h$ is a local weight of the hermitian metric $h$ on $L$).  

Next, in the zone $(1\ge \rho\ge \frac12 r_\varepsilon)$, one can use as before the diffeomorphism $\Upe$ and the problem is equivalent to that with the Calabi metric $dd^c\phi_{\beta, L}$ on a compact subset $K_\varepsilon \Subset L\setminus D$.  Over $K_\varepsilon$, the metric is uniformly comparable to the cone metric (\ref{eq:42}), so weighted Schauder estimates follow.

In the zone $(\rho \leqslant \frac12 \rho_\ep)$, we have $g_{\beta, \ep}=\ep g_{\rm TY}$. The Tian-Yau metric is asymptotically conical hence satisfies the desired Schauder estimate with weight function $R$. Since $\rho \approx \sqrt{\ep} R$ globally on the zone $(\rho\ge \frac 12 \rho_\ep)$, the desired estimate follows from the classical one for the Tian-Yau metric by the definition of our norms. 
\end{proof}

\subsection{Vector fields and eigenfunctions of the Laplacian on the cone $L$} 
We begin with the following elementary computation relating eigenfunctions of the Laplacian for $g_D$ and $g_\beta$. 
\begin{lem}
\label{vp}
Let $g\in C^\infty(D)$ such that $(\Delta_{\omega_D}+1)g=0$. Define for $\beta \ge \beta_*$ the function $f_\beta=\phi_\beta' g$ on an open subset of $L$. Then we have
\[(\Delta_{\omega_\beta}+\mu_\beta)f_\beta=0.\]
\end{lem}

\begin{proof}
We omit the index $\beta$ to lighten notation. Let $\psi=\psi(u)$ be a smooth function of $u$. We have 
\[dd^c(\psi g)=g(\psi'' du\wedge d^c u-\frac{2\psi'}{\alpha-1}\omega_D)+\psi dd^c g.\] 
Given \eqref{omb}, we deduce 
\begin{eqnarray}
\label{delta}
\Delta_{\omega_\beta}(\psi g)&=& \frac{\psi''}{\phi''}g+(n-1)\frac{\psi'}{\phi'}g-(\alpha-1)\frac{\psi}{\phi'} \Delta_{\omega_D}g\\
&=&(\frac{\psi''}{\psi \phi''}+(n-1)\frac{\psi'}{\psi \phi'}+(\alpha-1)\frac{1}{\phi'}) \psi g. \nonumber
\end{eqnarray}
Now, \eqref{MA} can be rewritten as 
\[\log \phi''+(n-1)\log \phi'+\mu \phi+(\alpha-1)u=0,\]
and after differentiation, we obtain 
\begin{equation}
\label{DE}
\frac{\phi'''}{\phi''}+(n-1) \frac{\phi''}{\phi'}+\mu \phi'+(\alpha-1)=0.
\end{equation}
Plugging $\psi:=\phi'$ into \eqref{delta} and using \eqref{DE}, we obtain 
\[\Delta_{\omega_\beta}(\phi' g)=-\mu \phi'g\]
as announced. 
\end{proof}

\begin{rem}
\label{radial}
The $\mathbb C^*$ action on $L$ yields another eigenvector of $\Delta_{\omega_\beta}$, which is $\phi_\beta'+\lambda_\beta$. More precisely, the computation above shows that we have $(\Delta_{\omega_\beta}+\mu_\beta)(\phi_\beta'+\lambda_\beta)=0$. 
\end{rem}

\begin{defi}
Let $(f_1, \ldots, f_\ell)$ be a basis of $\{f\in C^{\infty}(D); (\Delta_{g_D}+1)f=0\}$. We define $E\subset C^0(L)$ to be the finite dimensional real vector space
\[E:=\mathrm{Span}\left\{\phi'_{\beta_*}+\lambda_{\beta_*}, \phi'_{\beta_*}f_1, \ldots, \phi'_{\beta_*}f_\ell\right\}.\]
\end{defi}

\begin{prop}
\label{prop vp}
Let $h\in C^{\infty}(L\setminus D)$ be a solution of $(\Delta_{g_{\beta_*}}+\mu_{\beta_*})h=0$ such that $\|h\|_{C^{2,\alpha}_\delta}<+\infty$ for some $\delta\in(0,1)$ small enough and let $v\in H^0(\overline L, T_{\overline L}(-\log D))$. Then 
\begin{enumerate}[label=(\roman*)]
\item The vector field $v_h:=\mathrm{grad}^{1,0}_{g_{\beta_*}}h$ is holomorphic and tangent to $D$. 
\item There exists a unique $h_v$ such that $v=\mathrm{grad}^{1,0}_{g_{\beta_*}}h_v$ and $(\Delta_{g_{\beta_*}}+\mu_{\beta_*})h_v=0$. Moreover, $h_v\in E$. 
\end{enumerate}
The two constructions are inverse to each other, in the sense that $h_{v_h}=h$. 

\end{prop}

\begin{proof}

We start by proving the first assertion. It turns out that we'll need to refine the condition $h \in C^{2,\alpha}_\delta(L)$ in the following computations. This is the aim of the following

\begin{claim}
\label{improvement}
One has
\[\sup_{0<r<1} |\nabla h|_{g_{\beta_*}}+ r|\nabla^2 h|_{g_{\beta_*}}<+\infty,\]
where $r=e^{-\frac 12 \beta_* u}$. 
\end{claim}

\begin{proof}[Proof of Claim~\ref{improvement}]
  By Remark~\ref{radial}, $\phi'_\beta+\lambda_\beta$ satisfies the same linear equation as $h$ and clearly it satisfies the bounds in the claim. Therefore, up to replacing $h$ with $h-(\phi'_\beta+\lambda_\beta)$, one can assume that $h(0)=0$. Then the behaviour of $h$ near $r=0$ is governed by the first positive exceptional weight $\delta_1$ of $\Delta$, which satisfies $\delta_1>1$ (see for example \cite[Remark 2.10]{CH1}). It follows that $h\in C^{2,\alpha}_{\delta'}$ for any $\delta'<\delta_1$ and in particular for $\delta'=1$. So $\nabla h \in C^{1,\alpha}_0$ which gives the claim.
\end{proof}


To lighten notation, we will drop the index $\beta_*$ in the following and simply write $v$ for $v_h$. Since $\Ric \om=\mu\omega$, a classical computation shows that 
\begin{equation}
\label{dbarstar}
\dbar^*\dbar v=\mathrm{grad}^{1,0}\left[(\Delta_g+\mu)h\right]=0.
\end{equation}
Moreover, we have the identity
\begin{equation}
\label{stokes}
\left(\langle \dbar v, \dbar v\rangle_g - \langle \dbar^*\dbar v, v\rangle_g  \right) d\mathrm{vol}_g= d(\langle \ast  \dbar v, v\rangle_g),
\end{equation}
where $\ast: \Omega_X^{0,1}\otimes T_X^{1,0}\to \Omega_X^{n,n-1}\otimes T_X^{1,0}$ is the Hodge star operator.\\

For $\ep>0$, we set $L_\ep:=\{\log\ep \le u \le -\log \ep\}$. Integrating \eqref{stokes} over $L_\ep$ and using \eqref{dbarstar} yields
\begin{eqnarray*}
\int_{L_\ep} \|\dbar v \|^2 d\mathrm{vol}_g &=& \int_{\partial L_\ep}\langle \ast  \dbar v, v\rangle_g \\
&=&\int_{u=\log \ep}\langle \ast  \dbar v, v\rangle_g+\int_{u=-\log \ep}\langle \ast  \dbar v, v\rangle_g \\
&=:&I(\ep)+J(\ep)
\end{eqnarray*}
and we now have to show that $I(\ep)$ and $J(\ep)$ go to zero when $\ep\to 0$. Since $g$ is has conic singularities along $D$ and $h$ belongs do Donaldson's $C^{2,\alpha, \beta_*}$ space, the claim for $I$ has already been showed in the proof of \cite[Proposition~8]{Don}. So we now focus on $J(\ep)$. At infinity in $L$, recall from \eqref{equiV} that we have
\[g\approx dr^2+r^2\left( \eta^2+g_D\right)\] 
where $r= e^{-\frac 12 \beta_*u}$. Next, we have
\begin{eqnarray*}
\|\langle \ast  \dbar v, v\rangle\|_g &\le& \|\dbar v\|_g\cdot\|v\|_g\\
&\le &\|\nabla^{0,2} h\|_g\cdot\|\nabla h\|_g \\
&\le & r^{-1}
\end{eqnarray*}
by Claim~\ref{improvement}. On $\{u=-\log \ep\}$, we have $r\simeq r_\ep:=\ep^{\frac{\beta_*}2}$ hence 
\[|J(\ep)|\lessapprox r_\ep^{2n-2}\]
goes to zero as desired. \\

We now move on to proving the second assertion. 
Let $\xi$ be the radial holomorphic vector field. By Lemma~\ref{exact}, we can decompose $v$ as
\[v=a\xi+w\]
where $a\in \mathbb C$ and $w$ is induced by vector field $v_D\in H^0(D, T_D)$; of course we have $w|_{D}=v_D$. We want to find a function $h$ such that $v=\mathrm{grad}^{1,0}h$ or, equivalently, $i_v\omega=\dbar h$. 

Thanks to Remark~\ref{radial}, we have $i_\xi\omega=\dbar \phi'$. Therefore, it is enough to show the claim when $a=0$, i.e. when $v=w$ is induced by a vector field $v_D$ on $D$. By the classical case, one can write $i_{v_D}\om_D=\dbar h_{D}$ for some smooth function $h_D$ on $D$ satisfying $(\Delta_{g_D}+1)(h_D)=0$. Set $h:=-\frac1{\alpha-1}\phi'h_D$; since $(\Delta_{g_{\beta_*}}+\mu_{\beta_*})(h)=0$ by Lemma~\ref{vp}, all we have left to prove is 
\begin{equation}
\label{pot}
i_v\omega=\dbar h.
\end{equation} 
In order to prove \eqref{pot}, let $x\in D$ and let $(w_1, \ldots, w_{n-1})$ be a system of coordinates defined on a open set $U\subset D$. On $U$, $L=-\frac1{\alpha-1} K_D$ is trivialized by $\sigma:=(dw_1\wedge \ldots \wedge dw_{n-1})^{\otimes -\frac 1{\alpha-1}}$ and we get another coordinate $z$ on $p^{-1}(U)$ given by $v=z\sigma$ if $v\in L_x$. On $U$, one can write $v_D=\sum a_i \frac{\partial}{\partial w_i}$. We have $L_{v_D}\sigma =\frac {\mathrm{div}(v_D)}{\alpha-1}\cdot \sigma$ so that the lift of $v_D$ to $p^{-1}(U)$ is
\[v=\frac {\mathrm{div}(v_D)}{\alpha-1}\cdot  z\frac{\partial}{\partial z}+\sum_{i=1}^{n-1} a_i \frac{\partial}{\partial w_i}.\]
At $x$, we can always assume up to doing a linear change of coordinates that the local weight $\varphi$ of the hermitian metric on $-K_D$ satisfies $\varphi(x)=0$ and $d\varphi(x)=0$ so that $\partial u=\frac{dz}{z}$ at $x$. In particular, we have at any point in $p^{-1}(x)$
\begin{equation*}
\omega=\phi'' i\frac{dz}{z}\wedge \frac{d\bar z}{\bar z}-\frac{\phi'}{\alpha-1} \omega_D, \quad \mbox{and} \quad \dbar (\phi'h_D)=\phi''h_D \frac{d\bar z}{\bar z}+\phi' \dbar h_D.
\end{equation*}
Since $i_{v_D}\omega_D=\dbar h_D$, we have $\mathrm{div}(v_D)=-i\Delta_{g_D}h_D=ih_D$. This implies that
\begin{eqnarray*}
i_v\omega&=&\frac {\mathrm{div}(v_D)}{\alpha-1} \phi'' i\frac{d\bar z}{\bar z}-\frac{\phi'}{\alpha-1}i_{v_D}\omega_D\\
&=&\frac{-1}{\alpha-1}\left(\phi''h_D\frac{d\bar z}{\bar z}+\phi'\dbar h_D\right)\\
&=&\dbar h
\end{eqnarray*}
which shows \eqref{pot} at any point in the fiber of $x$, hence everywhere on $L$.\\
 
As for the last assertion, since $h_{v_h}$ and $h$ satisfy $i_{v_h}\omega=\dbar h = \dbar h_{v_h}$, we have that $h-h_{v_h}$ is constant, hence zero since it lies in the kernel of $\Delta_g+\mu$.
\end{proof}

\subsection{Reduction to a one-dimensional problem}
Let $G_D=\mathrm{Isom}^\circ(D, \omega_D)$ be the connected component of the group of holomorphic isometries of the Kähler-Einstein manifold $(D,\omega_D)$. By Matsushima theorem \cite{Mat57}, $G_D$ is a compact connected Lie group such that $G_D^\mathbb C=\mathrm{Aut}^\circ(D)$. 

From now on, we assume that the restriction morphism \eqref{rhoD} is surjective. Thanks to Lemma~\ref{toric}, there exists a surjective, finite étale map $G\to G_D$ of algebraic groups  where $G$ is a (compact) subgroup of $\mathrm{Aut}^\circ(X,D)$. We define the space of $G$-invariant functions $C^{2,\alpha}_{\delta}(X)^G:=\{f \in C^{2,\alpha}_{\delta}(X);\,\, \forall \sigma\in G, \, \sigma^*f=f\}$.

Let $(h_1, \ldots, h_\ell)$ be an orthonormal basis of $\{h\in C^{\infty}(D); (\Delta_{g_D}+1)h=0\}$, and let $\xi_i:=\mathrm{grad}_{g_D}^{1,0}h_i$ be the corresponding holomorphic vector field on $D$. By our assumption, there exists a unique vector field on $X$ tangent to $D$ extending $\xi_i$, cf proof of Lemma~\ref{toric}. We will abusively still denote it by $\xi_i$. 

Next, recall from Remark~\ref{radial} that $\phi_\beta'+\lambda_\beta$ is the potential of the scaling $\mathbb C^*$ action on $(L,\omega_{\beta, L})$. Borrowing the notation from  the beginning of \textsection~\ref{gluing}, we set 
\begin{equation}
\label{tau}
\tau=\tau_{\beta,\ep}:=\chi_\ep (\Upsilon_\ep)_* (\phi'_{\beta, L})+\lambda_\beta
\end{equation}
which is well-defined globally on $X$, and satisfies $\|\tau_{\beta, \ep}\|_{C^{0}_\delta}= \alpha(\alpha+\lambda)$. It is asymptotical in the kernel of $L_{\beta, \ep}$ in the following sense

\begin{lem}
\label{lem tau}
Let  $f \in C^{2,\alpha}_\delta(X)$ and let $\delta\in (0,2n)$, then we have 
\[\langle L_{\beta, \ep} f, \tau_{\beta, \ep} \rangle = O\Big(r_\ep^{2n-\delta} \|f\|_{C^{0}_\delta}\Big)\]
If, moreover, one has $\delta<n-2$ and $|dd^c g|_{\omega_{\beta, \ep}} \le 1$ then we have 
\[\langle Q_{\beta, \ep} (f), \tau_{\beta, \ep} \rangle = 
\begin{cases}
O\Big(\|f\|^2_{C^{2,\alpha}_{\delta}}\Big) & \mbox{if } n>2,\\
O\Big(\ep^{-\delta}\|f\|^2_{C^{2,\alpha}_{\delta}}\Big) & \mbox{if } n=2,
\end{cases}
\]
where the $O()$ are uniform in $f,g,\beta, \ep$. 
\end{lem}

\begin{proof}
Let us start with the first estimate. On $(\rho\ge 2\rho_\ep)$, we have $L_{\beta, \ep}\tau_{\beta, \ep} \equiv 0$ thanks to Remark~\ref{radial}. 

On $(\rho\le \frac 12 \rho_\ep)$, we have $L_{\beta, \ep}\tau_{\beta, \ep}=\mu_\beta\lambda_\beta=O(1)$. Since $f\le \rho^{-\delta}  \|f\|_{C^{0}_\delta}$ and $\int_{\rho \le \frac 12 \rho_\ep}\rho^{-\delta}  d\mathrm{vol}_{\ep g_{{\rm TY},\ep}} \approx  \int_{\sqrt{\ep}}^{\frac 12 \rho_\ep} t^{2n-1-\delta}dt=O(\rho_\ep^{2n-\delta})$, we are done. 

It remains to analyze the contribution of the gluing zone $(\frac 12 \rho_\ep \le \rho \le 2\rho_\ep)$. In that zone, we know from \eqref{psi5} and \eqref{MA2p} that $\phi^{(k)}_{\beta, \ep, L}=O(r^2)$ for any $k\ge 0$, where the derivative is to take with respect to $u$ and $r=e^{-\frac 12 \beta_* u} \approx \rho$. In particular, $d\phi'_{\beta, \ep, L}=O(r^2) du$ and $dd^c \phi'_{\beta, L, \ep}=O(r^2)\cdot (du\wedge d^cu+\omega_D)$. Since $g_{\beta, \ep,L}$ is asymptotic to $dr^2+r^2g_D$, we find $|\nabla \phi'_{\beta, \ep, L}|=O(r)$ and $|\Delta\phi'_{\beta, \ep, L}|=O(1)$. Finally, we have $|\nabla^k \chi_\ep |=O(r_\ep^{-k})$ so that in the end, $|L_{\beta, \ep}\tau_{\beta, \ep}|=O(1)$. One can now conclude with the same arguments as for the previous zone. \\

Let us now consider the integral involving $Q_{\beta, \ep}$; unlike the previous one we will need to look at the normal bundle zone as well. From \eqref{Q}, we have $|Q(f)| \le \rho^{-2\delta-4} \|f\|^2_{C^{2,\alpha}_{\delta}}$ pointwise. The desired estimate now follows from the same computations as in the previous step combined with the fact that the integral $\int_{X} \rho^{-\kappa} \omega_{\beta, \ep}^n$ is uniformly bounded when $\kappa <2n$ while $\int_{\rho \le 2\rho_\ep} \rho^{-2\delta} \omega_{\beta, \ep}^2 =O(\ep^{-\delta})$ when $n=2$. 
\end{proof}

Finally, we introduce 
\[C^{2,\alpha}_{\delta}(X)^\perp:=\left\{f\in C^{2,\alpha}_{\delta}(X); \, \int_X f \tau_{\beta, \ep} \, d \mathrm{vol}_{g_{\beta, \ep}}=0\right\},\]
 set $ C^{2,\alpha}_{\delta}(X)^{G, \perp}:=C^{2,\alpha}_{\delta}(X)^G\cap C^{2,\alpha}_{\delta}(X)^\perp$ and write ${}^\perp: C^{2,\alpha}_{\delta}(X)\to C^{2,\alpha}_{\delta}(X)^\perp$ the orthogonal projection, i.e. $f^\perp=f-\frac{\langle f,\tau_{\beta, \ep} \rangle}{|\tau_{\beta, \ep}|^2}\cdot \tau_{\beta, \ep}$.

\begin{prop}\label{prop:uniform-inverse}
Fix $0<\delta<1$. There is a constant $C>0$ independent of $\beta, \ep$ (with $|\beta-\beta_*|+\ep \ll 1$) such that for every $f\in C^{2,\alpha}_{\delta}(X)^{G, \perp} $, we have
\[\|f\|_{C^{2, \alpha}_\delta} \le C \|(L_{\beta, \ep} f)^\perp\|_{C^{0,\alpha}_{\delta+2}}.\]
\end{prop}

\begin{proof}
We argue by contradiction. We can extract a sequence $(\beta_j, \ep_j)\to 0$ and find functions $f_j\in C^{2,\alpha}_{\delta}(X)^G\cap C^{2,\alpha}_{\delta}(X)^\perp$ such that $\|f_j\|_{C^{2, \alpha}_\delta}=1$ but satisfying $\|(L_{\beta, \ep} f_j)^\perp\|_{C^{0,\alpha}_{\delta+2}}\to 0$ when $j\to +\infty$. 

Thanks to the first item in Lemma~\ref{lem tau}, we have $\langle L_{\beta, \ep}f_j, \tau_{\beta, \ep}\rangle \to 0$, hence  $\|L_{\beta, \ep} f_j\|_{C^{0,\alpha}_{\delta+2}}\to 0$ since $L_{\beta, \ep} f_j = (L_{\beta, \ep} f_j)^\perp+\frac{\langle f_j, L\tau_{\beta, \ep}\rangle}{|\tau_{\beta, \ep}|}\tau_{\beta, \ep}$. By Proposition~\ref{schauder}, $\|f\|_{C^{0}_\delta}$ does not converge to zero. For simplicity, one will assume that $\|f\|_{C^{0}_\delta}=1$ and we choose $x_j\in X$ such that $r^\delta |f_j|(x_j)=1$. \\

{\it Case 1. } $\rho(x_j)>\eta>0$.

Using $\Upsilon_\ep$, one can view the $f_j$ as functions on larger and larger open subsets of $L$, with controlled $C^{2,\alpha}$ bound with respect to $g_{\beta, L}$. Since $\rho(x_j)\ge \eta$, the points $x_j$ belong to a compact set of $L$ and one can assume that  $(x_j)$ converges to $x_\infty\in L$. Thanks to Arzela-Ascoli theorem, one can extract a limit function $f\in C^{2,\alpha}_\delta(L)$ with respect to the cone metric $g_{\beta_*}$ satisfying $f(x_\infty)\neq 0$ and $(\Delta_{g_{\beta_*}}+\mu_{\beta_*})(f)=0$. From Proposition~\ref{prop vp}, it follows that there are constants $a_0, \ldots, a_\ell \in \mathbb R$ such that 
\[f=a_0(\phi'_{\beta_*}+\lambda_{\beta_*}) +\sum_{i=1}^\ell a_i h_i.\]

Since $f_j$ is $G$-invariant, we have $L_{\xi_i} f_j=0$ on $X$ for any $i=1, \ldots, \ell$. Since $\xi_i$ is tangent to $D$, one can restrict the previous identity to $D$ and let $j\to +\infty$. It yields $L_{\xi_i} (f|_D)=0$ on $D$. Since $\xi_i=\mathrm{grad}^{1,0}h_i$, we have 
\[\int_D L_{\xi_i} f d\mathrm{vol}_{g_D}=\sum_{k=1}^\ell a_k \langle \mathrm{grad}^{1,0} h_i,  \mathrm{grad}^{1,0} h_k\rangle= -\sum_{k=1}^\ell a_k \langle h_i, \Delta h_k \rangle=a_i\]
since $\Delta h_k=-h_k$. In particular, we get $a_i=0$ for all $i=1,\ldots, \ell$ so that $f=a_0(\phi'_{\beta_*}+\lambda_{\beta_*})$.

Next, the same computations as in the proof of Lemma~\ref{tau} show that  
\[\Big|\int_{r\le 2r_\ep} f_j \tau_{\beta_j, \ep_j} d\mathrm{vol}_{g_{\beta_j,\ep_j}}\Big| \le C\int_{r\le 2r_\ep} r^{-\delta} d\mathrm{vol}_{g_{\beta_j,\ep_j}}=O(r_\ep^{2n-\delta}).\]
Since $f_j$ is orthogonal to $\tau_{\beta_j, \ep_j}$, it follows from the dominated convergence theorem that $\int_L f(\phi'_{\beta_*}+\lambda_{\beta_*}) d\mathrm{vol}_{g_{\beta_*}}=0$.  Therefore, we get $a_0=0$, hence $f\equiv 0$, a contradiction. \\

{\it Case 2. } $1\gg \rho(x_j) \gg \sqrt{\ep}$. The set $\Upe^{-1}(U_L)$ is of the form $\{+\infty \ge r \ge \sqrt \ep\}$ with the notation of \eqref{r rho}. The set $\Upe^{-1}(1\gg \rho \gg \sqrt{\ep})$ in $L$ is simply $(1 \gg r \gg \sqrt \ep)$. We claim that in that zone, $\Upe^*g_{\beta, \ep}$ is asymptotic to $g_{{\rm TY}, L}$ whose Kähler form is $\omega_{{\rm TY}, L}=dd^c \frac {r^2}{4}$. Let us briefly justify the claim in each of the three relevant zones. In the Tian-Yau part (i.e. when $g_{\beta, \ep}=\ep^2 g_{{\rm TY}, \ep}$), it follows from Theorem~\ref{thm:dev-TY} since the first order asymptotics of the formal perturbation are the same as the ones for the genuine Tian-Yau metric by Proposition~\ref{prop formal expansion}. On the normal bundle part (i.e. before the gluing) the claim follows from  \eqref{expansion4} since $dd^c \phi_{\beta_*, L}$ has the same asymptotic cone as $\omega_{{\rm TY}, L}$. Finally, in the gluing zone, it follows from \eqref{order 2} below. 

Now, set $r_j:=\rho(x_j)$, $\hat r:=\frac{r}{r_j}$ and consider the function $\hat f_j:= \Upe^*(r_j^{\delta}f_j|_{U_L})$ on $(\frac{1}{r_j} \gg \hat r \gg \frac{\sqrt \ep}{r_j})$. It is bounded in the usual weighted $C^{2,\alpha}$ norm with respect to the metric $ \Upe^*(r_j^{-2}g_{\beta, \ep})$ (whose Kähler form is asymptotic to $dd^c \frac{\hat r^2}{4}$)  and the weight $\hat r$. Moreover, it satisfies  $\Delta_{ \Upe^*(r_j^{-2}g_{\beta, \ep})}  \hat f_j + r_j^2\mu_\beta \hat f_j \to 0$ in $C^\alpha_{\rm loc}(L^\times)$ in the usual sense. By  Arzela-Ascoli, $\hat f_j$ converges (up to extracting a subsequence) in the usual $C_{\rm loc}^2$ topology on $L^\times$ to a non-zero function $f$ satisfying
\[\Delta_{g_{{\rm TY}, L}} f=0 \quad \mbox{and} \quad \sup _{L^\times}r^\delta |f|<\infty.\]
 Let $\ep>0$; since $\delta\in (0,1)$ and $n>1$, the function $\pm f-\frac{\ep}{r^{2n-2}}$ is harmonic and tends to $-\infty$ near the apex $r=0$ and $0$ when $r\to +\infty$. By the maximum principle, it is non-positive hence $\pm f \le \frac{\ep}{r^{2n-2}}$. Since this holds for any $\ep>0$, we infer that $f\equiv 0$, hence the desired contradiction. \\

{\it Case 3. } $ \rho(x_j) =O( \sqrt{\ep})$.

In this case, $x_j$ belongs to a fixed compact subset of $X\setminus D$ and $\rho(x_j) \approx \sqrt \ep_j$. Thanks to Arzela-Ascoli theorem, a diagonal extraction argument allows one  to find a subsequential non-zero limit $f$ of the sequence of functions $(\ep_j^{\frac \delta 2} f_j)$ on $X\setminus D$ for the $C^{2,\alpha}_{\rm loc}$ topology. Moreover, $f$ satisfies $\sup_{X\setminus D} \rho^{\delta} |f|<+\infty$ as well as $\Delta_{g_{\rm TY}}f=0$. By the maximum principle, we must have $f\equiv 0$, a contradiction.
\end{proof}

\section{Estimating the error and the obstruction}\label{sec:estim-error-obstr}

\subsection{The error}
Let us start by estimating $P_{\beta, \ep}(0)$, where $P_{\beta, \ep}$ is defined in \eqref{eq:P}. There are three zones: \\

$\bullet$ {\it Zone} $(\rho \ge 2r_\ep)$. 

\smallskip
\noindent
Since $\Upsilon^*J_X-J_L$ is horizontal up to order $j$ and $\phi_{\beta, L}$ only depends on $u$, an elementary computation shows that the $(1,1)$-component of $\Upsilon_\ep^*dd^c\varphi_{\beta, \ep}- dd^c \varphi_{\beta, L}$ is dominated by a multiple of 
\[|\zeta_\ep|^j\big[(\phi_{\beta, L}''-\phi_{\beta, L}') du \wedge d^cu -\phi_{\beta, L}' \omega_D\big],\]
where $\zeta_\ep$ is defined in \eqref{zeta ep}.  Given \eqref{omb}, we have 
\[\Upsilon_\ep^*(dd^c\varphi_{\beta, \ep})^n=(dd^c \phi_{\beta, L})^n \Big(1+ O\Big(|\zeta_\ep|^j \Big(\big|\frac{\phi_{\beta, L}'}{\phi_{\beta, L}''}\big|+\big|\frac{\phi_{\beta, L}''}{\phi_{\beta, L}'}\big|\Big)\Big)\Big). \]
Now, the ratio $\frac{\phi_{\beta, L}''}{\phi_{\beta, L}'}$ is globally bounded while $-\frac{\phi'_{\beta, L}}{\phi_{\beta, L}''}$ is uniformly bounded away from $u=-\infty$ and behaves like $e^{-\beta u}=|\zeta|^{-2\beta}$ near $D$, cf \eqref{exp D}. Given \eqref{eq:31} and \eqref{OmegaL}, we have
\begin{equation}
\label{scale vol}
\lambda_\ep^*(\Omega_L\wedge \overline \Omega_L)= \ep^{-n} \Omega_L\wedge\overline \Omega_L
\end{equation}
so that Proposition~\ref{prop:normal-form-3} yields
\begin{equation}
\label{scale vol 2}
\Upsilon_\ep^*(\ep^n\Omega \wedge \overline \Omega)= \Omega_L\wedge\overline \Omega_L \big(1-|\phi_{j_0} \zeta_\ep^{j_0}|^2+O(\zeta_\ep^{2j_0+1})\big)
\end{equation}
Plugging \eqref{scale vol 2} into \eqref{MA ep}, we obtain
\begin{eqnarray}
\Upsilon_\ep^*P_{\beta, \ep}(0)&=&|\phi_{j_0}\zeta_\ep^{j_0}|^2+O(\zeta_\ep^{2j_0+1}) +  O(|\zeta_\ep|^{j-2}) \label{estimee 1}\\
&=& |\phi_{j_0}|^2 \cdot |\zeta_\ep|^{2j_0}+O(\zeta_\ep^{2j_0+1}) \nonumber
\end{eqnarray}
for $j> 2(j_0+1)$. We choose such a $j$ from now on.  In terms of the coordinate $u$, we have $\Upsilon_\ep^*P_{\beta, \ep}(0)=O(\ep^{\frac{j_0}{\beta_*}}e^{j_0u})=O(\big(\frac{\ep}{r^2}\big)^{\nu_0})$. For any $\delta<2\nu_0$, the weighed supremum of the latter quantity is attained at $r_\ep$, i.e. 
\begin{equation}
\label{estimee 2}
\sup_{\rho \ge 2r_\ep} \rho^\delta |P_{\beta, \ep}(0)| \lessapprox r_\ep^\delta \ep^{\theta \nu_0} .
\end{equation}
and the same bound holds as well for the $C^{0,\alpha}_{\delta}$ norm.

\medskip

$\bullet$ {\it Zone} $(2r_\ep \ge \rho \ge \frac 12r_\ep)$. 

\smallskip

\noindent
Given \eqref{pot gluing}, we have to evaluate $\Upsilon_\ep^*(\ep \Phi_{{\rm TY}, \ep})-\phi_{\beta, L}$. Since $\lambda_\ep^*(\ep R^2)=r^2$, the second item of Proposition~\ref{prop formal expansion} yields
{\small
\begin{eqnarray}
\Upsilon_\ep^*(\ep\Phi_{{\rm TY}, \ep})&=&\Upsilon_\ep^*(\ep\Phi_{\rm TY}) +\sum_{j=2}^k\big[ a_jr^{2j}+O(r^{2j} \big( \frac{\ep}{r^2}\big)^{\nu_0})+O(\ep^{j}\big( \frac{\ep}{r^2}\big)^{n-1})+O(r^{2k+2})\big] \nonumber \\
&=&\Upsilon_\ep^*(\ep\Phi_{\rm TY})+\Big(\sum_{j=2}^k a_jr^{2j}\Big)+O(r^{4} \big( \frac{\ep}{r^2}\big)^{\nu_0})+O(\ep r^{2}\big( \frac{\ep}{r^2}\big)^{n})+O(r^{2k+2}) \label{YeTY}\\
&=&(\Upsilon_\ep^*(\ep\Phi_{\rm TY})-\frac{r^2}4)+\phi_{\beta_*}(r)+ O(r^{4} \big( \frac{\ep}{r^2}\big)^{\nu_0})+O(\ep r^{2}\big( \frac{\ep}{r^2}\big)^{n})+O(r^{2k+2}) \nonumber.
\end{eqnarray}
}
Now, Theorem~\ref{thm:dev-TY} says that 
  \begin{equation}
    \label{YTY}
    \Upsilon_\ep^*(\ep \Phi_{\rm TY}) = \frac{r^2}4 +
    \begin{cases}
      O(r^2\big(\frac {\ep}{r^2}\big)^{\nu_0}) & \text{ if } \nu_0< n, \\
      ar^2 \big(\frac {\ep}{r^2}\big)^{n}\big(\log \frac{r}{\sqrt \ep}+O(1)\big)\\
      ar^2 \big(\frac {\ep}{r^2}\big)^{n} + O(r^2 \big(\frac {\ep}{r^2}\big)^{n+\sigma}) & \text{ if }\nu_0>n,
    \end{cases}
  \end{equation}
  where $\sigma>0$ is small. The error terms in \eqref{YTY} are larger than the ones in \eqref{YeTY}, at least in the gluing zone and for $k$ large enough. 

Combining \eqref{YTY} with \eqref{expansion2}, one infers that in the gluing zone 
{\small
\begin{equation}
\label{diff}
\Upsilon_\ep^*(\ep \Phi_{{\rm TY}, \ep})-\phi_{\beta, L}=
    \begin{cases}
      -a_{L}\frac{\beta-\beta_*}{r^{2n-2}}+O(r^2\big(\frac {\ep}{r^2}\big)^{\nu_0}+\frac{\beta-\beta_*}{r^{2n-2}} \cdot \delta(r))   & \text{ if } \nu_0< n, \\
      \frac{ a \ep^n \log (\frac{r}{\sqrt\ep})-a_L(\beta-\beta_*)}{r^{2n-2}}+O(\frac{\ep^n}{r^{2n-2}}+\frac{\beta-\beta_*}{r^{2n-2}} \cdot \delta(r)) & \text{ if } \nu_0= n,\\ 
        \frac {a\ep^n-a_{L}(\beta-\beta_*)}{r^{2n-2}} + O(r^2 \big(\frac {\ep}{r^2}\big)^{n+\sigma}+\frac{\beta-\beta_*}{r^{2n-2}} \cdot \delta(r)) & \text{ if }\nu_0>n,
    \end{cases}
   \end{equation}
    \small}
   where 
   \begin{equation}
   \label{delta r}
   \delta(r):=r^2+\frac{\beta-\beta_*}{r^{2n}},
   \end{equation}
   and the same estimates hold as well for the weighed derivatives $r^j \nabla^j$.  Next, the tensors $r^j\nabla^j \chi_\ep$ are bounded for any $j$. In light of \eqref{pot gluing} and relying again on the fact that $J-J_L$ can be chosen horizontal up to an arbitrary large order, we infer the coarser estimate   \begin{equation}
   \label{order 2}
   \|\Upsilon_\ep^*\omega_{\beta, \ep}-\omega_{\beta, L}\|_{\omega_{\beta, L}}=
    \begin{cases}
      O( \frac{\beta-\beta_*}{r_\ep^{2n}}+\ep^{\theta \nu_0}  )   & \text{ if } \nu_0< n, \\
     O( \frac{\beta-\beta_*}{r_\ep^{2n}}+\ep^{\theta n} \log \frac 1 \ep) & \text{ if } \nu_0= n, \\
        O( \frac{\beta-\beta_*}{r_\ep^{2n}}+\ep^{\theta n}  ) & \text{ if }\nu_0>n,
    \end{cases}
   \end{equation}
  Similarly, thanks to the estimates of $P_{\beta, \ep}(0)$ obtained in the previous zone we also have 
\begin{equation}
\label{estimee 3}
\sup_{2r_\ep \ge \rho \ge \frac 12r_\ep} |P_{\beta, \ep}(0)| =
    \begin{cases}
      O( \frac{\beta-\beta_*}{r_\ep^{2n}}+\ep^{\theta \nu_0}  )   & \text{ if } \nu_0< n, \\
     O( \frac{\beta-\beta_*}{r_\ep^{2n}}+\ep^{\theta n} \log \frac 1 \ep) & \text{ if } \nu_0= n, \\
        O( \frac{\beta-\beta_*}{r_\ep^{2n}}+\ep^{\theta n}  ) & \text{ if }\nu_0>n,
    \end{cases}
\end{equation}
and the weighted $C^{0,\alpha}_\delta$ estimate on the considered zone follows immediately from \eqref{estimee 3} since $\rho \simeq r_\ep$. \\

$\bullet$ {\it Zone} $(\rho \le \frac 12r_\ep)$. 

\smallskip
\noindent
The first second item of Proposition~\ref{prop formal expansion} shows that 
\begin{equation}
\label{estimee 3 bis}
\sup_{\rho \le \frac 12r_\ep} \rho^\delta |P_{\beta, \ep}(0)| \lessapprox r_\ep^{2k+2+\delta}.
\end{equation}
where $k$ can be chosen arbitrary large and the same bound holds for the $C^{0,\alpha}_\delta$ norm. Said otherwise, the "Tian-Yau" zone will not contribute to the error terms.\\

In conclusion, we find that for any $\delta \in (0,2\nu_0)$, we have

\begin{equation}
\label{P(0) norm}
\|P_{\beta, \ep}(0)\|_{C^{0,\alpha}_\delta} = 
    \begin{cases}
      O( \frac{\beta-\beta_*}{r_\ep^{2n-\delta}}+r_\ep^{\delta}\ep^{\theta \nu_0}  )   & \text{ if } \nu_0< n, \\
     O( \frac{\beta-\beta_*}{r_\ep^{2n-\delta}}+r_\ep^{\delta}\ep^{\theta n} \log \frac 1  \ep) & \text{ if } \nu_0= n, \\
        O( \frac{\beta-\beta_*}{r_\ep^{2n-\delta}}+r_\ep^{\delta}\ep^{\theta n}  ) & \text{ if }\nu_0>n.
    \end{cases}
\end{equation}

\subsection{The obstruction}
In this section, we estimate 
\[\langle P_{\beta, \ep}(0), \tau_{\beta, \ep}\rangle=\int_X P_{\beta, \ep}(0) \tau_{\beta, \ep} \, d \mathrm{vol}_{g_{\beta, \ep}}\] where $\tau_{\beta, \ep}$ has been defined in \eqref{tau}. Here again, we study each of the three zones separately. \\

$\bullet$ {\it Zone} $(\rho \ge 2r_\ep)$. 

\smallskip
\noindent
Here, we have $\Upsilon_\ep^*\tau_{\beta, \ep}=\phi'_{\beta, L}+\lambda_\beta$. Since we can work as if $\phi$ were horizontal, the estimate \eqref{estimee 1} shows that the main contribution to be computed is $\ep^{\nu_0}\cdot I_{\ep,\beta}$ where
\begin{eqnarray*}
I_{\ep,\beta}&:=&\int_{u\le u_\ep-\log 4}|\phi_{j_0}|^2  e^{j_0 u}  (\phi'_{\beta, L}+\lambda_\beta)\cdot  \phi_{\beta, L}''(-\phi'_{\beta, L})^{n-1} du \wedge d^cu \wedge \omega_{D}^{n-1}\\
&=&\int_{u\le u_\ep-\log 4}|\phi_{j_0}|^2  e^{\beta_*(\nu_0-n)u}  (\phi'_{\beta, L}+\lambda_\beta)\cdot  e^{-\mu \phi_{\beta, L}} du \wedge d^cu \wedge \omega_{D}^{n-1}
\end{eqnarray*}
Performing fiberwise integral we are reduced to studying the 1D integral
\[J_{\beta, \ep}:=\int_{ -\infty}^{u_\ep-\log 4} e^{\beta_*(\nu_0-n)u}  (\phi'_{\beta, L}+\lambda_\beta)\cdot  e^{-\mu \phi_{\beta, L}} du.\]
Near $u=-\infty$, the integral is indeed convergent since $e^{-\mu \phi_{\beta, L}} \approx e^{(n\beta_*+\beta)u}$. Near $0$, the integrand behaves like $e^{\beta_*(\nu_0-n)u}$ hence it will converge (when $r_\ep\to 0$) if and only if $\nu_0<n$. The sign of the integral at the limit is crucial but not obvious at this point since $\phi'_{\beta, L}+\lambda_\beta$ does not have a sign (it is negative near $D$ and positive near the conical point). Let us work out this issue next and separate three cases. \\

$a.$ Case $\nu_0<n$.  One writes
\[e^{\beta_*(\nu_0-n)u}  (\phi'_{\beta, L}+\lambda_\beta)\cdot  e^{-\mu \phi_{\beta, L}}=-\frac1\mu e^{\beta_*\nu_0 u} \cdot \big(e^{-\mu \phi_{\beta, L}-n\beta_*u}\big)'\]
and integrates by parts to obtain
\[J_{\beta, \ep}=\frac {\beta_*\nu_0} \mu \int_{ -\infty}^{u_\ep-\log 4}e^{\beta_*(\nu_0-n) u-\mu \phi_{\beta, L}}du+O(e^{\beta_*(\nu_0-n)u_\ep})\]
hence there is a number $a_\beta>0$ converging to a positive number $a_{\beta_*}$ when $\beta\to \beta_*$ such that
\begin{equation}
\label{Ieb}
I_{\beta, \ep}=a_\beta+O(r_\ep^{2(n-\nu_0)}). 
\end{equation}

\medskip

$b.$  Case $\nu_0=n$.  Thanks to \eqref{expansion}, we find a positive number $a_\beta'$ such that
\begin{equation}
\label{Ieb2}
I_{\beta, \ep}=a_\beta' u_\ep+O(1).
\end{equation}

\medskip

$c.$  Case $\nu_0>n$.  The rough estimate $I_{\ep,\beta}=O(r_{\ep}^{2(n-\nu_0)})$ will suffice. 
 All in all, we have (up to scaling $a_\beta'$)
\begin{equation}
\label{obstruction0}
\int_{\rho \ge 2r_\ep} P_{\beta, \ep}(0) \tau_{\beta, \ep} \, d \mathrm{vol}_{g_{\beta, \ep}}=
\begin{cases}
a_\beta \ep^{\nu_0}+O(r_\ep^{2n} \ep^{\theta \nu_0}) & \mbox{if} \quad \nu_0<n,\\
a_\beta' \ep^n \log \frac 1{r_\ep}+O(\ep^n) &\mbox{if}\quad  \nu_0=n \\
O(\ep^{\theta \nu_0} r_{\ep}^{2n})& \mbox{if} \quad  \nu_0>n. 
\end{cases}
\end{equation}

\medskip

$\bullet$ {\it Zone} $(2r_\ep \ge \rho \ge \frac 12r_\ep)$. 

\smallskip
\noindent
To lighten the notation in this paragraph, let us set $A_\ep:=\Upsilon_\ep^{-1}(\{2r_\ep \ge \rho \ge \frac 12r_\ep\})\approx (2r_\ep \ge r \ge \frac 12r_\ep)$ and
\[\psi_{\beta,\ep}:=(1-\chi_\ep)(\Upsilon_\ep^*(\ep \Phi_{{\rm TY}, \ep})-\phi_{\beta, L})\]
so that 
\begin{equation}
\label{diff metriques}
\Upsilon_\ep^*\omega_{\beta, \ep}=dd^c \phi_{\beta, L}+dd^c \psi_{\beta, \ep}+h.o.t.
\end{equation}
by \eqref{pot gluing}. If we additionnally set 
\begin{equation}
\label{def E}
E:=
    \begin{cases}
      \frac{\beta-\beta_*}{r_\ep^{2n}}+\ep^{\theta \nu_0}     & \text{ if } \nu_0< n, \\
       \frac{\beta-\beta_*}{r_\ep^{2n}}+\ep^{\theta n} \log \frac 1 \ep  & \text{ if } \nu_0= n, \\
          \frac{\beta-\beta_*}{r_\ep^{2n}}+\ep^{\theta n}   & \text{ if }\nu_0>n,
    \end{cases}.
    \end{equation}
we see from \eqref{diff} and \eqref{estimee 3} that
\begin{equation}
\label{estimee 4}
\sup_{A_\ep} |\psi_{\beta, \ep}|= O(r_\ep^2 E), \quad \mbox{and} \quad \sup_{A_\ep} \big(|dd^c \psi_{\beta, \ep}|_{\omega_{\beta, L}}+|\Upsilon_\ep^*P_{\beta, \ep}(0)|\big)=O(E). 
\end{equation}
Thanks to \eqref{diff metriques} and the almost horizontality of $\phi$ as well as
\begin{equation}
\label{error term}
\Upsilon_\ep^* \omega_{\beta, \ep}^n=\om_{\beta, L}^n(1+O(E)). 
\end{equation}
The integral we need to estimate is 
\begin{eqnarray}
\int_{A_\ep} \Upsilon_\ep^*( P_{\beta,\ep}(0)\tau_{\beta, \ep} \omega_{\beta, \ep}^n)&=& \int_{A_\ep} \Upsilon_\ep^*( P_{\beta,\ep}(0)\tau_{\beta, \ep} )\omega_{\beta, L}^n+ O\big(r_\ep^{2n}  E \sup_{2r_\ep \ge \rho \ge \frac 12r_\ep}|P_{\beta, \ep}(0)|)\big) \nonumber\\
&=& \int_{A_\ep} \Upsilon_\ep^*( P_{\beta,\ep}(0)\tau_{\beta, \ep} )\omega_{\beta, L}^n+O(r_\ep^{2n} E^2) \label{integral error}
\end{eqnarray}
thanks to \eqref{estimee 4}. So from now on we focus on the integral $\int_{A_\ep} \Upsilon_\ep^*( P_{\beta,\ep}(0)\tau_{\beta, \ep} )\omega_{\beta, L}^n$. \\

First, we write 
\begin{equation}
\label{zone 2}
\Upsilon_\ep^* P_{\beta,\ep}(0)= \log \frac{|\Omega_L|^2}{\Upsilon_\ep^*(\ep^n|\Omega|^2)}+L_{\beta} \psi_{\beta, \ep}+Q(dd^c \psi_{\beta, \ep})
\end{equation}
where $L_{\beta}= \Delta_{\omega_{\beta, L}}+\mu_\beta$ and $Q(dd^c \psi_{\beta,\ep})=O(E^2)$. In the gluing zone where we are working, the estimate \eqref{scale vol 2} reads
\begin{equation}
\label{ration vol}
\log \frac{|\Omega_L|^2}{\Upsilon_\ep^*(\ep^n|\Omega|^2)}=O(\ep^{\theta\nu_0}).
\end{equation}
Since $\Upsilon_\ep^*\tau_{\beta, \ep}\approx \lambda_\beta$, we have
\begin{equation}
\label{easy term}
 (\mu_\beta \psi_{\beta, \ep}+Q(dd^c \psi_{\beta, \ep})) \cdot \Upsilon_\ep^*\tau_{\beta, \ep} =O(r_\ep^{2}E+E^2)
\end{equation}
by \eqref{estimee 4}. We are left to evaluating 
\[I:=\int_{A_\ep} \Delta_{\omega_{\beta, L}} \psi_{\beta, \ep} \cdot (\chi_\ep \phi_{\beta, L}'+\lambda_\beta) \, \om_{\beta, L}^n=I_1+I_2+I_3,\]
where 
\[I_1=\int_{A_\ep}  \psi_{\beta, \ep} \cdot \Delta_{\omega_{\beta, L}} (\chi_\ep \phi_{\beta, L}') \, \om_{\beta, L}^n\]
and 
\[I_2=-\int_{r=\frac 12 r_\ep} \partial_\nu \psi_{\beta, \ep} \cdot (\phi_{\beta, L}'+\lambda_\beta) d\sigma, \quad I_3=\int_{r=\frac 12 r_\ep} \psi_\ep \cdot \partial_\nu \phi_{\beta, L}'d\sigma\]
thanks to Stokes theorem since $\psi_{\beta, \ep}=0$ near $r=2r_\ep$, where $\nu$ is the outward normal vector and $\sigma$ is the measure on $(r=cst)$ induced by $\om_{\beta, L}$. From \eqref{omb} and \eqref{expansion3}, we see that $\partial_\nu = (c_1+O(\delta(r)) \partial_r$ and $d\sigma = (c_2+O(\delta(r)) r^{2n-1} \eta \wedge d\mathrm{vol}_{g_D}$ where $c_1, c_2$ are positive constants and $\delta(r)$ has been defined in \eqref{delta r}. 

Let us first deal with $I_1$. We have $|\Delta_{\omega_{\beta, L}} \chi_\ep|+|\nabla \chi_\ep|^2 = O(r_\ep^{-2})$. Moreover, the estimate $\phi'_{\beta, L}=O(r^2)$ derived from \eqref{expansion2} can be iterated to higher orders by \eqref{MA} to obtain $\Delta_{\omega_{\beta, L}}(\chi_\ep  \phi'_{\beta, L}) = O(1)$. Combined with \eqref{estimee 4}, this shows that the integrand in $I_1$ is a $O(r_\ep^2 E)$ so that 
\begin{equation}
\label{I1}
I_1=O(r_\ep^{2n+2}E).
\end{equation}

Let us now get to $I_2$. We have $ \phi'_{\beta,L}+\lambda_\beta= \lambda_\beta+O(r_\ep^2)$ and thanks to \eqref{diff}
\begin{equation*}
\frac{-1}{2n-2}\partial_r \psi_{\beta, \ep}=
    \begin{cases}
       -a_{L}\frac{\beta-\beta_*}{r^{2n-1}}+O(r\ep^{\theta\nu_0}+\frac{\beta-\beta_*}{r^{2n-1}} \cdot \delta(r))  & \text{ if } \nu_0< n, \\
        \frac{ a \ep^n \log (\frac{r}{\sqrt\ep})-a_L(\beta-\beta_*)}{r^{2n-1}}+O(\frac{\ep^n}{r^{2n-1}}+\frac{\beta-\beta_*}{r^{2n-1}} \cdot \delta(r)) & \text{ if } \nu_0= n,\\ 
        \frac {a\ep^n-a_{L}(\beta-\beta_*)}{r^{2n-1}} + O(r \ep^{\theta(n+\sigma)}+\frac{\beta-\beta_*}{r^{2n-1}} \cdot \delta(r)) & \text{ if }\nu_0>n.
    \end{cases}
   \end{equation*}
   Therefore, there exists a positive constant $\kappa=\kappa(\nu_0,\beta)>0$ such that 
   \begin{equation}
\label{I2}
I_2=
    \begin{cases}
       -\kappa(\beta-\beta_*)+O(r_\ep^{2n}\ep^{\theta\nu_0 }+(\beta-\beta_*) \cdot \delta(r))  & \text{ if } \nu_0< n, \\
        \kappa\big( a \ep^n \log (\frac{r_\ep}{\sqrt \ep})-a_L(\beta-\beta_*)\big)+O(\ep^n+(\beta-\beta_*) \cdot \delta(r)) & \text{ if } \nu_0= n,\\ 
       \kappa (a\ep^n-a_{L}(\beta-\beta_*))+ O(r_\ep^{2n} \ep^{\theta(n+\sigma)}+(\beta-\beta_*) \cdot \delta(r)) & \text{ if }\nu_0>n.
    \end{cases}
   \end{equation}
 
Finally, we have $\psi_{\beta,\ep} \cdot \partial_\nu  \phi'_{\beta, L} =O(r_\ep^3E)$ which shows that
\begin{equation}
\label{I3}
I_3=  O(r_\ep^{2n+2}E).
\end{equation}

Given \eqref{integral error},  \eqref{ration vol}, \eqref{easy term}, \eqref{I1} and \eqref{I3}, all the terms contributing to $\int_{A_\ep} \Upsilon_\ep^*( P_{\beta,\ep}(0)\tau_{\beta, \ep} \omega_{\beta, \ep}^n)$ except for $I_2$ are dominated by
\[r_\ep^{2n}E^2+r_\ep^{2n+2}E+r^{2n} \ep^{\theta \nu_0}\]
and it is straightforward to see that these terms are at most of the same order as the error term in the integral $I_2$, cf \eqref{I2}. In conclusion, 
{\small
\begin{equation}
\label{obstruction}
\int_{2r_\ep \ge \rho \ge \frac 12r_\ep} P_{\beta, \ep}(0) \tau_{\beta, \ep} \, d \mathrm{vol}_{g_{\beta, \ep}}=
    \begin{cases}
       -\kappa(\beta-\beta_*)+O(r_\ep^{2n}\ep^{\theta\nu_0 }+(\beta-\beta_*) \cdot \delta(r))  & \text{ if } \nu_0< n, \\
        \kappa\big( a \ep^n \log (\frac{r_\ep}{\sqrt \ep})-a_L(\beta-\beta_*)\big)+O(\ep^n+(\beta-\beta_*) \cdot \delta(r)) & \text{ if } \nu_0= n,\\ 
       \kappa (a\ep^n-a_{L}(\beta-\beta_*))+ O(r_\ep^{2n} \ep^{\theta(n+\sigma)}+(\beta-\beta_*) \cdot \delta(r)) & \text{ if }\nu_0>n.
    \end{cases}
\end{equation}
}

\medskip

$\bullet$ {\it Zone} $(\rho \le \frac 12r_\ep)$. 

\smallskip
\noindent
Since the volume of the zone is of order $\ep^n R_\ep^{2n}=r_\ep^{2n}$, the second item of Proposition~\ref{prop formal expansion} shows that 
\begin{equation}
\label{estimee 6}
\int_{\rho \le \frac 12r_\ep} P_{\beta, \ep}(0) \tau_{\beta, \ep} \, d \mathrm{vol}_{g_{\beta, \ep}} =O( r_\ep^{2n+2k+2}).
\end{equation}
where $k$ can be chosen arbitrary large. As before, this zone will not contribute to the obstruction. \\

In conclusion, the estimates \eqref{obstruction0}, \eqref{obstruction} and \eqref{estimee 6} imply that they are positive constants $\kappa_i=\kappa_i(\nu_0, \beta)$ for $i=1, \ldots, 6$ (where the dependence in $\beta$ is continuous up to $\beta=\beta_*$ and $\kappa_i(\nu_0, \beta_*)>0$) such that
\begin{equation}
\label{obstruction final}
\langle P_{\beta, \ep}(0), \tau_{\beta, \ep}\rangle=
\begin{cases}
\kappa_1 \ep^{\nu_0}-\kappa_2(\beta-\beta_*)+O(r_\ep^{2n}\ep^{\theta \nu_0}+(\beta-\beta_*)\cdot \delta(r_\ep)) & \text{ if } \nu_0< n, \\
\kappa_3 \ep^n \log \frac 1 \ep -\kappa_4(\beta-\beta_*) +O(\ep^n+(\beta-\beta_*)\cdot \delta(r_\ep))& \text{ if } \nu_0= n, \\
\kappa_5 \ep^n-\kappa_6(\beta-\beta_*)+O(r_\ep^{2n}\ep^{n\theta +\sigma}+(\beta-\beta_*)\cdot \delta(r_\ep))& \text{ if } \nu_0> n. 
\end{cases}
\end{equation}
This is clear except maybe in the case where $\nu_0=n$, where the first term  $\kappa_3 \ep^n \log \frac 1\ep$ comes from the contributions of the obstruction from both the normal bundle (term $a_\beta'\ep^n \log \frac 1{r_\ep} = \frac 12 (1-\theta)a_\beta' \ep^n \log \frac 1 \ep$) and the gluing zone (term $\kappa a\ep^n \log \frac{r_\ep}{\sqrt \ep}=\frac 12 \theta \kappa a \ep^n \log \frac 1\ep$) which are both {\it positive}. \\

 \subsection{Choice of $\ep=\ep(\beta)$}
Given \eqref{obstruction final}, it is natural to introduce the scaling parameter $\ep_\beta$ satisfying 
\begin{equation}
\label{beta}
\beta-\beta_*=
\begin{cases}
\frac {\kappa_1}{\kappa_2} \cdot \ep_\beta^{\nu_0}  & \text{ if } \nu_0< n, \\
\frac {\kappa_3}{\kappa_4} \cdot \ep_\beta^n \log \frac 1 {\ep_\beta}  & \text{ if } \nu_0= n, \\
\frac {\kappa_5}{\kappa_6} \cdot \ep_\beta^n  & \text{ if } \nu_0> n
\end{cases}
\end{equation}
and then we choose from now on $\ep$ close to $\ep_\beta$ in the sense that 
\begin{equation}
\label{ep eta}
\ep=\ep_\beta(1+\eta) \quad \mbox{with} \quad |\eta|\ll 1
\end{equation}
where $\eta\in \mathbb R$ will be allowed to vary later for the purpose of solving out Monge-Ampère equation; in particular it will be important that $\eta$ takes positive and negative values. 
Let us be more precise now and distinguish the three possible cases. \\

$\bullet$ {\it Case $\nu_0<n$.}

\noindent
Condition \eqref{zone beta} translates into 
\begin{equation}
\label{theta bound}
\theta>1-\frac{\nu_0}{n},
\end{equation}
which can be fulfilled. Next, we have 
\[\kappa_1 \ep^{\nu_0}-\kappa_2(\beta-\beta_*)=\nu_0\kappa_1 \ep_\beta^{\nu_0}  \cdot  (\eta+O(\eta^2))\]
 while the error term in \eqref{obstruction final} is of order
 \[r_\ep^{2n}\ep^{\theta \nu_0}+(\beta-\beta_*)\cdot \delta(r_\ep)\approx \ep^{\nu_0}(\ep^{(1-\theta)(n-\nu_0)}+\ep^{1-\theta}+\ep^{\nu_0-n(1-\theta)})\] and the latter is a $O(\ep_\beta^{\nu_0+(1-\theta)\min\{1,n-\nu_0\}})$ as soon as $\theta>1-\frac{\nu_0-1}{n}$, which we assume from now on. Then we have
\begin{equation}
\label{gamma 11}
\langle P_{\beta, \ep}(0), \tau_{\beta, \ep}\rangle=\nu_0\kappa_1 \ep_\beta^{\nu_0} \cdot \big(\eta+O(\ep^{(1-\theta)\min\{1,n-\nu_0\}})+O(\eta^2)\big).
\end{equation} 
Moreover, the quantity $E$ in \eqref{def E} behaves like $\ep_\beta^{\nu_0-n(1-\theta)}$. \\

$\bullet$  {\it Case $\nu_0=n$.}

\noindent 
Condition \eqref{zone beta} is automatic. Next, we have
\[\kappa_3 \ep^n \log \frac 1 \ep -\kappa_4(\beta-\beta_*)=n\kappa_3 \ep_\beta^n \log \frac 1{\ep_\beta}  \cdot (\eta+O(\frac {\eta} {\log \frac {1}{\ep_\beta}})+O(\eta^2))\]
while the order of the error term $O(\ep^n+(\beta-\beta_*)\cdot \delta(r_\ep))$ is simply $O(\ep_\beta^n)$. 
We then have
\begin{equation}
\label{gamma 22}
\langle P_{\beta, \ep}(0), \tau_{\beta, \ep}\rangle=n\kappa_3 \ep_\beta^n \log \frac 1{\ep_\beta} \cdot \big(\eta+O(\frac {1} {\log \frac {1}{\ep_\beta}})+O(\eta^2)\big).
\end{equation}
Moreover, $E$ behaves like $\ep_\beta^{n\theta}\log \frac 1{\ep_\beta}$. \\

$\bullet$    {\it Case $\nu_0>n$.} 

\noindent
Condition \eqref{zone beta} is automatic too and 
\[\kappa_5 \ep^n-\kappa_6(\beta-\beta_*)=n\kappa_5\ep_\beta^n \cdot (\eta+O(\eta^2))\]
while the order of the error term $O(r_\ep^{2n}\ep^{n\theta +\sigma}+(\beta-\beta_*)\cdot \delta(r_\ep))$ is $O(\ep_\beta^{n+\sigma})$ as soon as $\theta>\frac 1{n+1}$. 
Therefore, we get
\begin{equation}
\label{gamma 33}
\langle P_{\beta, \ep}(0), \tau_{\beta, \ep}\rangle=n\kappa_5\ep_\beta^n \cdot \big(\eta+O(\ep_\beta^\sigma)+O(\eta^2)\big).
\end{equation}
Moreover, $E$ behaves like $\ep_\beta^{n\theta}$. \\

In summary, we get from \eqref{P(0) norm} and the case by case analysis just above

\begin{prop}
\label{prop obstruction}
If $\beta ,\ep, \eta$ are chosen as in \eqref{beta}-\eqref{ep eta}, then there exists $\kappa\in \mathbb R$ satisfying
\begin{equation}
\label{P gamma}
\langle P_{\beta, \ep}(0), \tau_{\beta, \ep}\rangle=   (\beta-\beta_*) \cdot\big (\kappa \eta +O(F)+O(\eta^2)\big)
\end{equation}
and $\kappa$ is such that $\kappa_0 \le |\kappa |\le \kappa_0^{-1}$ for some constant $\kappa_0>0$ independent of $\beta, \ep$ and 
\[F:=
    \begin{cases}
       \ep_\beta^{(1-\theta)\min\{1,n-\nu_0\}}  & \text{ if } \nu_0< n, \\
     \frac {1} {\log \frac {1}{\ep_\beta}}& \text{ if } \nu_0= n, \\
       \ep_\beta^\sigma   & \text{ if }\nu_0>n.
    \end{cases}
\]
Moreover, if $\delta\in (0, 2\nu_0-2)$, we have 
\begin{equation}
\label{P(0)}
\|P_{\beta, \ep}(0)\|_{C^{0,\alpha}_{\delta+2}} = O(r_\ep^{\delta+2} E)
\end{equation}
where the $O()$ is uniform in $\beta, \ep,\eta$ and $E$ is defined by
\begin{equation}
\label{def E 2}
E:=
    \begin{cases}
       \ep^{\nu_0-n(1-\theta)}   & \text{ if } \nu_0< n, \\
     \ep^{ n\theta} \log \frac 1 \ep & \text{ if } \nu_0= n, \\
       \ep^{ n\theta}   & \text{ if }\nu_0>n.
    \end{cases}
\end{equation}
\end{prop}

\section{Resolution of the Monge-Ampère equation}

\subsection{Resolution modulo the obstruction}
The first step is to solve the Monge-Ampère equation up to the obstruction term. More precisely, we have

\begin{prop}
\label{resolution mod obstruction}
Let $\delta \in (0,1)$. Suppose $\ep$ close enough to $0$ and $0<\beta-\beta_*<c \varepsilon^{\inf(\nu_0,n)}$ (or $c \varepsilon^n \log \frac1\varepsilon$ is $\nu_0=n$), where $c$ is some fixed constant. Then there exists a unique function $f_{\beta, \ep}\in C^{2,\alpha}_{\delta}(X)^{G, \perp}$ such that 
\[P_{\beta, \ep}(f_{\beta, \ep})^\perp=0 \quad \mbox{and} \quad \|f_{\beta, \ep}\|_{C^{2,\alpha}_{\delta}} \le C \|P_{\beta, \ep}(0)^\perp\|_{C^{0, \alpha}_{\delta+2}}\]
for some constant $C>0$ independent of $\beta, \ep$.
\end{prop}

\begin{rem}
\label{rem TY formel nec}
We will see in the proof that it is crucial to have $\|P_{\beta,\varepsilon}(0)\|_{C^\alpha_{2+\delta}} \ll \varepsilon^{1+\frac\delta2}$. If one had not replaced the Tian-Yau metric by its formal perturbation (or had only done it up to, say, order $R^4$ in the potential so that $\Ric (\ep \om_{{\rm TY},\ep} )\approx \ep \om_{{\rm TY},\ep}$), the Tian-Yau part would contribute a factor $r_\ep^{4+\delta}=\ep^{(1-\theta)(2+\frac \delta 2)}$ (resp. $r_\ep^{6+\delta}=\ep^{(1-\theta)(3+\frac \delta 2)}$) in the norm $\|P_{\beta,\varepsilon}(0)\|_{C^\alpha_{2+\delta}} $ and the latter need not be small compared to  $\varepsilon^{1+\frac\delta2}$ since $\theta$ will have to be chosen arbitrarily close to $1$ in the proof of Proposition~\ref{obs gam} below. 
\end{rem}

\begin{proof}
Given any $\nu'<\inf(\nu_0,n)$ (which we take very close to $\inf(\nu_0,n)$), Proposition~\ref{prop obstruction} shows that for $\theta<1$ close enough to $1$ we have
\begin{equation}
  \label{eq:36}
  \|P_{\beta,\varepsilon}(0)\|_{C^{0,\alpha}_{2+\delta}} = O(\varepsilon^{\nu'}).
\end{equation}

The quadratic term $Q_{\beta,\varepsilon}$ of the equation, defined by (\ref{decomp P}), is controled exactly as in \cite[§ 1.4]{BiqMin11}: it satisfies
\begin{equation}
  \label{eq:40}
  |Q_{\beta,\varepsilon}(f_1)-Q_{\beta,\varepsilon}(f_2)| \leq C |dd^c(f_1-f_2)| \big( |dd^cf_1|+|dd^cf_2| \big)
\end{equation}
as soon as $|dd^cf_j|\leq 1$ for $j=1,2$. Our weighted space $C^0_{\delta+2}$ carries a (small) weight $\varepsilon^{1+\frac\delta2}$ on the Tian-Yau part, and it follows that
{\small
\begin{equation}
  \label{eq:41}
  \| Q_{\beta,\varepsilon}(f_1)-Q_{\beta,\varepsilon}(f_2) \|_{C^0_{2+\delta}} \leq C \varepsilon^{-1-\frac\delta2} \|dd^c(f_1-f_2)\|_{C^0_{2+\delta}} \big( \|dd^cf_1\|_{C^0_{2+\delta}}+\|dd^cf_2\|_{C^0_{2+\delta}} \big).
\end{equation}}
This estimate extends to the Hölder spaces $C^\alpha_{2+\delta}$ by local change of scale.

Finally we have a uniform estimate for the inverse of the linearization provided by Proposition \ref{prop:uniform-inverse}. One can then apply an inverse function theorem on the equation $P_{\beta,\varepsilon}(f)^\perp=0$, provided that $\|P_{\beta,\varepsilon}(0)\|_{C^\alpha_{2+\delta}} \ll \varepsilon^{1+\frac\delta2}$, see for example \cite[Lemma 1.3]{BiqMin11}. Therefore from (\ref{eq:36}) it is sufficient to have $\nu'>1+\frac\delta2$, which can be achieved provided $\inf(\nu_0,n)>1$, that is provided that $\nu_0>1$ which is true since $\beta_*<1$.
\end{proof}

\subsection{Killing the obstruction by varying $\ep$}
The outcome of Proposition~\ref{resolution mod obstruction} is that for fixed $(\ep, \beta)$ close enough to $(0,\beta_*)$, there exists a unique couple of a function and a scalar $(f_{\beta, \ep}, a_{\beta, \ep})\in C^{2,\alpha}_{\delta}(X)^{G, \perp} \times \mathbb R$ such that 
\[P_{\beta, \ep}(f_{\beta, \ep})=a_{\beta, \ep} \tau_{\beta, \ep}.\]
The object of the next proposition is to precisely measure the effect that varying the cone angle (or the scale parameter $\ep$) has on the obstruction $a_{\beta, \ep}$. Most of the work has already been done in Proposition~\ref{prop obstruction}. 

\begin{prop}
\label{obs gam}
Let $\beta, \ep, \eta$ as in \eqref{beta}-\eqref{ep eta}. For a fixed choice of $\delta, \theta \in (0,1)$ such that $\delta+(1-\theta) \ll 1$, we have 
\[\langle P_{\beta, \ep}(0), \tau_{\beta, \ep}\rangle=   (\beta-\beta_*) \cdot\big (\kappa \eta +O(F)+O(\eta^2)\big)\]
where $\kappa, F$ are the ones from Proposition~\ref{prop obstruction}.
\end{prop}

\begin{proof}
Given \eqref{P gamma}, all we have left to prove is that 
\begin{equation}
\label{obstruction o gamma}
\langle L_{\beta, \ep}(f_{\beta, \ep}), \tau_{\beta, \ep}\rangle+\langle Q_{\beta, \ep}(f_{\beta, \ep}), \tau_{\beta, \ep}\rangle=o(\eta(\beta-\beta_*)).
\end{equation}

We proceed in several steps. \\

\noindent
{\bf Step 1. }{\it Preliminaries.}

\medskip

\noindent
We first claim that 
\begin{equation}
\label{fC2}
\|f_{\beta, \ep}\|_{C^{2,\alpha}_{\delta}} =O(r_\ep^{\delta+2} E),
\end{equation}
where $E$ is given in \eqref{def E 2}. Indeed, \eqref{P gamma} shows that 
\[|P_{\beta, \ep}(0)^\perp| \le |P_{\beta, \ep}(0)|+\frac 1{|\tau_{\beta, \ep}|}\cdot |\langle P_{\beta, \ep}(0), \tau_{\beta, \ep}\rangle| =O(E+|\gamma|).\]
Now one can easily check that $|\gamma| \ll E$ since $\theta<1$. The estimate \eqref{fC2} is then a consequence of Proposition~\ref{resolution mod obstruction} and \eqref{P(0)}. \\

Next, we claim that 
\begin{equation}
\label{asymptotic}
|dd^c f_{\beta, \ep}|_{\omega_{\beta,\ep}} = o(1).
\end{equation}
This is an immediate consequence of \eqref{fC2} on $\rho \ge \frac 12 r_\ep$ given the definition of our norms and the fact that $E=o(1)$. So it remains to check the claim on the Tian-Yau part. There, one can just write $|dd^c f_{\beta, \ep}|_{\omega_{\beta,\ep}}\le \rho^{-\delta-2} \|f_{\beta, \ep}\|_{C^{2,\alpha}_{\delta}} = O( \ep^{-\theta(\frac \delta2 +1)} E)$. If $\nu_0<n$, the latter goes to zero provided that $\theta> \frac{n-\nu_0}{n-1-\frac \delta 2}$ which can always be achieved for some $\delta,\theta \in (0,1)$ since $\nu_0>1$. If $\nu_0\ge n$, it goes to zero provided $\delta <2(n-1)$ which can also be achieved since we have assumed that $n>1$. \\

\noindent
{\bf Step 2. }{\it The term $\langle L_{\beta, \ep}(f_{\beta, \ep}), \tau_{\beta, \ep}\rangle$.}

\medskip
\noindent
By Lemma~\ref{lem tau}, the integral $I:=\langle L_{\beta, \ep}(f_{\beta, \ep}), \tau_{\beta, \ep}\rangle$ satisfies 
\[I=O(r_\ep^{2n-\delta} \|f\|_{C^{0}_{\delta}}) =O(r_\ep^{2n+2}E)\]
where the last identity follows from \eqref{fC2}. We have
\[r_\ep^{2n+2}E=
\begin{cases}
       \ep^{\nu_0+(1-\theta)}   & \text{ if } \nu_0< n, \\
     \ep^{ n+1-\theta} \log \frac 1 \ep & \text{ if } \nu_0= n, \\
      \ep^{ n+1-\theta}   & \text{ if }\nu_0>n.
\end{cases}
\]
and the latter is a $O((\beta-\beta_*)F)$ since $\ep_\beta^{1-\theta}=O(F)$ \--- as soon as $\sigma<1-\theta$ in the case $\nu_0>n$.\\

\noindent
{\bf Step 3. }{\it The term $\langle Q_{\beta, \ep}(f_{\beta, \ep}), \tau_{\beta, \ep}\rangle$.}

\medskip
\noindent
Thanks to \eqref{asymptotic}, the second estimate from Lemma~\ref{lem tau} is valid for $f=f_{\beta, \ep}$, hence $\langle Q_{\beta, \ep}(f_{\beta, \ep}), \tau_{\beta, \ep}\rangle=O(\|f_{\beta,\ep}\|^2_{C^{2,\alpha}_{\delta}})$ if  $n>2$ and $\delta<n-2$ and $\langle Q_{\beta, \ep}(f_{\beta, \ep}), \tau_{\beta, \ep}\rangle=O(\ep^{-\delta}\|f_{\beta,\ep}\|^2_{C^{2,\alpha}_{\delta}})$ if $n=2$. \\

Assume that $n>2$ for now. Given \eqref{fC2}, we have to prove that 
\[r_\ep^{2\delta+2} E^2 = O((\beta-\beta_*)F).\]
We will show actually show that 
\begin{equation}
\label{Q2}
r_\ep^2E^2=\ep^{1-\theta}E^2=O((\beta-\beta_*)F)
\end{equation}
for suitable $\theta$. More precisely, if $\nu_0<n$, we choose $\theta> 1-\frac{\nu_0}{n}$ while if $\nu_0\ge n$ we choose $\theta>\frac {n-1}{2n-1}$. In the first case, one has automatically $(1-\theta)+2\nu_0-2(1-\theta)n>(1-\theta)$. In the second case, we have $1-\theta+2n\theta>n$. Since
\[E^2=
\begin{cases}
\ep^{2\nu_0-2(1-\theta)n}  & \text{ if } \nu_0< n, \\
\ep^{2n\theta} \log^2 \frac 1 \ep  & \text{ if } \nu_0= n, \\
\ep^{2n\theta} & \text{ if } \nu_0> n, 
\end{cases}
\]
the estimate \eqref{Q2} is satisfied. \\

If $n=2$, then we have to show instead 
\begin{equation}
\label{Q3}
\ep^{-\delta \theta}r_\ep^{2} E^2 = \ep^{1-(1+\delta)\theta}E^2= O((\beta-\beta_*)F)
\end{equation}
and we can apply the previous case to $(1+\delta)\theta$ (in place of $\theta$) for $\delta\ll 1$ and $\theta$ as before so that \eqref{Q3} follows from \eqref{Q2}. 
\end{proof}

\subsection{Proof of the main theorem}

We now get on to proving the main theorem. There is one last technical result needed to reach the conclusion. 

\begin{lem}
\label{continuity gamma}
There exists $\delta_0>0$ such that the functional $\Psi:(\beta,\ep)\mapsto \langle P_{\beta, \ep}(f_{\beta, \ep}), \tau_{\beta, \ep}\rangle$ is continuous on $(\beta_*, \beta_*+\delta_0)\times(0, \delta_0)$. 
\end{lem}

\begin{proof}
Let us choose $\sigma:=(\beta, \ep)$ close enough to $(\beta_*,0)$ and let $\sigma_k\to \sigma$ be any sequence. By Arzela-Ascoli theorem and the bound \eqref{fC2} one can extract a subsequence such that the functions $f_{\sigma_k}$ converge locally smoothly on $X\setminus D$ to a function $f_{\infty} \in C^{2, \alpha}_{\delta}(X)^{G}$. In particular, 
\begin{equation}
\label{convergence P}
P_{\sigma_k}(f_{\sigma_k})\to P_{\sigma}(f_{\infty})
\end{equation}
on $X\setminus D$. 
Since $|dd^c f_{\sigma_k}|_{\omega_{\sigma_k}}$ is bounded, Lebesgue domination theorem shows that 
\begin{equation}
\label{conv 2}
\langle f_{\sigma_k}, \tau_{\sigma_k}\rangle \to \langle f_{\infty}, \tau_{\sigma}\rangle, \quad \mbox{and} \quad \Psi(f_{\sigma_k})\to \langle P_\sigma(f_\infty),\tau_\sigma \rangle=:a_\infty.
\end{equation} Since $P_{\sigma_k}(f_{\sigma_k})^\perp=0$, it follows from \eqref{convergence P} that
\[P_{\sigma}(f_{\infty})=a_\infty \frac{\tau_{\sigma}}{|\tau_{\sigma}|^2} \quad \mbox{and} \quad \langle f_{\infty}, \tau_{\sigma}\rangle=0.\]
Given the uniqueness property in Proposition~\ref{resolution mod obstruction}, we infer that $f_\infty=f_\sigma$. By the second item in \eqref{conv 2}, $\Psi$ is continuous at $\sigma$. 
\end{proof}

\begin{proof}[Proof of Main Theorem]
We first prove the existence and then the convergence results. 

1. We fix $\beta$ and vary $\ep$ as prescribed by \eqref{ep eta}. Proposition~\ref{obs gam} then shows that $\langle P_{\beta, \ep}(f_{\beta, \ep}), \tau_{\beta, \ep}\rangle$ achieves positive (resp. negative) values whenever $\eta$ satisfies $1\gg \eta \gg F$ (resp. $1\gg -\eta \gg F$). By Lemma~\ref{continuity gamma} above, this quantity is continuous with respect to $\eta$. Therefore, given any $\beta>0$ close enough to $\beta_*$, there exists $\ep=\ep(\beta)$ as in \eqref{ep eta} such that 
\begin{equation}
\label{solution}
P_{\beta, \ep}(f_{\beta, \ep})=0,
\end{equation}
i.e. $\omega_{\beta, \ep}+dd^c f_{\beta, \ep}$ is Kähler-Einstein. Moreover, one has
\begin{equation}
\label{asym 2}
(1-o(1))\cdot \om_{\beta, \ep} \le \omega_{\beta, \ep}+dd^c f_{\beta, \ep} \le (1+o(1))\cdot \om_{\beta, \ep} 
\end{equation}
by \eqref{asymptotic}. \\

2. We need to identify the Gromov-Hausdorff limit of $(X,\omega_{{\rm KE},\beta})$. 
By \eqref{asym 2} above, it is enough to show the result for the model metric $\omega_{\beta, \ep}$. To lighten notation, let us write $\omega_{{\rm KE},\beta}:=\omega_{\beta, \ep}+dd^c f_{\beta, \ep}$. Since the diameter of the both the gluing zone and the Tian-Yau part are a $O(r_\ep)$, and since the diameter of $\big((r\le r_\ep), \omega_{\beta_*, L}\big)$ is also a $O(r_\ep)$,  it is enough to show that 
\[\sup_{r\ge r_\ep} |\omega_{\beta, L}-\omega_{\beta_*,L}|_{\omega_{\beta_*, L}} \longrightarrow 0\]
when $\ep \to 0$ or, equivalently, when $\beta\to \beta_*$. Now the latter follows from \eqref{expansion4} and the choice of $r_\ep$ which satisfies \eqref{zone beta}. \\

3. Finally, the locally smooth convergence $\ep_\beta^{-1}\omega_{{\rm KE},\beta}$ to the Tian-Yau metric on $X\setminus D$ is an immediate consequence of \eqref{asym 2}. Indeed, the latter gives $C^0$ convergence of the tensor on each compact subset of $X\setminus D$, and the usual bootstrapping arguments allow to improve it to smooth convergence thanks to the Monge-Ampère equation \eqref{eq:P}-\eqref{solution} solved by $\ep_\beta^{-1}\omega_{{\rm KE},\beta}$. 
\end{proof}

\section{Examples}
\label{sec examples}
\subsection{Projective hypersurfaces} 

\subsubsection{$\mathbb P^n$ and the quadric} Here, we consider $X=\mathbb P^n$ and $D=Q_{n-1}$ for $n\ge 2$. We have $\alpha=\frac{n+1}{2}$, $D$ is homogeneous hence Kähler-Einstein and the vector fields on $D$ clearly lift to $X$. Moreover, we have $j_0=1$ by Remark~\ref{rem j0} and Example~\ref{split Pn}. In particular, we get
\[\ep_\beta= 
\begin{cases}
(\beta-\beta_*)^{\frac 12} & \mbox{if } \, n=2 \\
\left(\frac{\beta-\beta_*}{-\log(\beta-\beta_*)}\right)^{\frac 13} & \mbox{if } \, n=3 \\
(\beta-\beta_*)^{\frac {n-1}{2n}} & \mbox{if } \, n\ge 4. \\
\end{cases}
\]

\subsubsection{$\mathbb P^n$ and a large degree hypersurface}
We consider now the case where $X=\mathbb P^n$ and $D$ is a general hypersurface of degree $n$ (resp. $n-1$). We may assume that $n\ge 3$ (resp. $n\ge 4$) so that we are not back to the case of the quadric above. We have $\alpha=\frac{n+1}{n}$ (resp. $\frac{n+1}{n-1}$), $D$ is Kähler-Einstein by \cite[Theorem~4.3]{Tian87} and the condition on the automorphisms is fulfilled thanks to Lemma~\ref{toric} $(ii)$. Actually, when $n\ge  3$ one has $\mathrm{Aut}^\circ(D)=\{1\}$ by \cite{MM63}. Moreover, we have $j_0=1$ by Remark~\ref{rem j0} and Example~\ref{split Pn} and $\alpha <2$ hence
 \[\ep_\beta=(\beta-\beta_*)^{\frac 1n}.\]
 
\subsection{Homogeneous varieties}
\subsubsection{The quadric and a hyperplane section}
Let $X=Q_n\subset \mathbb P^{n+1}$ be the smooth quadric of dimension $n$ for $n\ge 2$  and let $D\simeq Q_{n-1}$ be a general hyperplane section. We have $\alpha=n$ and all the assumptions of the main theorem are clearly met. The normal exact sequence is split by Example~\ref{split quadric}, hence $j_0\ge 2$. It would be interesting to determine the precise value of $j_0$, but we will not pursue this here. 

\subsubsection{Some grassmannians and a hyperplane section} Here, we let $X=\mathrm{Gr}(2,2n)$ for some $n\ge 2$ and $D$ be a general hyperplane section under the Plücker embedding. Note that $X$ has Picard number one. Since $D$ parametrizes $2$-planes in the kernel of a symplectic $2$-form, $D$ is homogeneous under $\mathrm{Sp}(2n)$ and therefore the assumptions of the theorem are met.   	

\bibliographystyle{alpha}
\bibliography{biblio}

\end{document}